\documentclass[smallextended]{svjour3}       
\smartqed  

\usepackage{graphicx}
\usepackage{amsmath,amssymb}
\usepackage{url,hyperref}
\usepackage[usenames,dvipsnames]{xcolor}
\usepackage{xfrac}
\usepackage{booktabs,multirow,paralist} 
\usepackage{algorithm}
\usepackage{algpseudocode}
\usepackage{pifont}
\usepackage[letterpaper, margin=1.5in]{geometry}
\usepackage{afterpage}

\newtheorem{assumption}{Assumption}

\newcommand{\R}{\mathbb{R}}
\newcommand{\Rext}{\R\cup\{+\infty\}}

\newcommand{\set}[1]{\left\{#1\right\}}
\newcommand{\sets}[1]{\{#1\}}
\newcommand{\norm}[1]{\left\Vert#1\right\Vert}
\newcommand{\norms}[1]{\Vert#1\Vert}
\newcommand{\Eproof}{\hfill $\square$}
\newcommand{\prox}{\mathrm{prox}}
\newcommand{\tprox}[2]{\mathrm{prox}_{#1}\left(#2\right)}
\newcommand{\proj}{\mathrm{proj}}

\newcommand{\argmin}{\mathrm{arg}\!\displaystyle\min}
\newcommand{\dom}[1]{\mathrm{dom}(#1)}

\newcommand{\iprods}[1]{\langle #1\rangle}
\newcommand{\Exp}[1]{\mathbb{E}\left[#1\right]}
\newcommand{\Exps}[2]{\mathbb{E}_{#1}\left[#2\right]}
\newcommand{\Probn}{\mathbb{P}}
\newcommand{\Prob}[1]{\mathbb{P}\left(#1\right)}

\newcommand{\pb}{\mathbf{p}}
\newcommand{\Bc}{\mathcal{B}}
\newcommand{\Xc}{\mathcal{X}}

\newcommand{\Sc}{\mathcal{S}}

\newcommand{\Tc}{\mathcal{T}}
\newcommand{\Rc}{\mathcal{R}}

\newcommand{\Fc}{\mathcal{F}}

\newcommand{\Nbb}{\mathbb{N}}

\newcommand{\trace}[1]{\mathrm{trace}\left(#1\right)}
\newcommand{\BigO}[1]{\mathcal{O}\left(#1\right)}
\newcommand{\Uni}[1]{\mathbb{U}\left(#1\right)}
\newcommand{\Unip}[2]{\mathbb{U}_{#1}\left(#2\right)}
\newcommand{\Grad}{\mathcal{G}}
\newcommand{\bsigma}{\boldsymbol{\sigma}}

\newcommand{\cmark}{\ding{51}}%
\newcommand{\xmark}{\ding{55}}%
\newcommand{\done}{\rlap{$\square$}{\raisebox{2pt}{\large\hspace{1pt}\cmark}}%
\hspace{-2.5pt}}
\newcommand{\wontfix}{\rlap{$\square$}{\large\hspace{1pt}\xmark}}

\newcommand{\myeq}[2]{\vspace{-0.25ex}
\begin{equation}\label{#1}
{#2}
\vspace{-0.25ex}
\end{equation}
}
\newcommand{\myeqn}[1]{\vspace{-0.25ex}
\begin{equation*}
{#1}
\vspace{-0.25ex}
\end{equation*}
}

\newcommand{\mytxtbi}[1]{\textbf{\textit{#1}}}

\newcommand{\mytxtit}[1]{\textit{#1}}

\newcommand{\nhan}[1]{{#1}}

\newcommand{\update}[1]{{#1}}
\newcommand{\revise}[1]{#1}

\newcommand{\beforesec}{\vspace{-3.5ex}}
\newcommand{\aftersec}{\vspace{-2.25ex}}
\newcommand{\beforesubsec}{\vspace{-4ex}}
\newcommand{\aftersubsec}{\vspace{-2.5ex}}
\newcommand{\beforesubsubsec}{\vspace{-2.5ex}}
\newcommand{\aftersubsubsec}{\vspace{-2.5ex}}
\newcommand{\beforepara}{\vspace{-2.5ex}}

\begin{document}

\title{A Hybrid Stochastic Optimization Framework for Composite Nonconvex Optimization}

\titlerunning{A Hybrid Stochastic Optimization Framework for Composite Nonconvex Optimization}        

\author{Quoc Tran-Dinh$^{*}$ \and Nhan H. Pham \and Dzung T. Phan \and Lam M. Nguyen}
\authorrunning{Q. Tran-Dinh, N. H. Pham, D. T. Phan, \textit{and} L. M. Nguyen}


\institute{Quoc Tran-Dinh \and  Nhan H. Pham \at
		Department of Statistics and Operations Research\\
		The University of North Carolina at Chapel Hill, 318 Hanes Hall, UNC-Chapel Hill, NC 27599-3260.\\ 
		Email: \url{quoctd@email.unc.edu}, \url{nhanph@live.unc.edu}.
\and
Dzung T. Phan \and Lam M. Nguyen \at
IBM Research, Thomas J. Watson Research Center \\ Yorktown Heights, NY10598, USA.\\
Email: \url{phandu@us.ibm.com}, \url{lamnguyen.mltd@ibm.com}.
\and
$^{*}$Corresponding author.
}

\date{Version 4: The first manuscript  was posted on Arxiv on March 13, 2019}

\maketitle

\vspace{-3ex}
\begin{abstract}
We introduce a new approach to develop stochastic optimization algorithms for a class of stochastic composite and possibly nonconvex optimization problems.
The main idea is to combine two stochastic estimators to create a new hybrid one.  
We first introduce our hybrid estimator and then investigate its fundamental properties to form a foundational theory for algorithmic development.  
Next, we apply our theory to develop several variants of stochastic gradient methods to solve both expectation and finite-sum composite optimization problems. 
Our first algorithm can be viewed as a variant of proximal stochastic gradient methods with a single-loop, but can achieve $\BigO{\sigma^3\varepsilon^{-1} + \sigma \varepsilon^{-3}}$-oracle complexity bound, matching the best-known ones from state-of-the-art double-loop algorithms in the literature, where \revise{$\sigma > 0$} is the variance and $\varepsilon$ is a desired accuracy.
Then, we consider two different variants of our method: adaptive step-size and restarting schemes that have similar theoretical guarantees as in our first algorithm.
We also study two mini-batch variants of the proposed methods.  
In all cases, we achieve the best-known complexity bounds under standard assumptions.
We test our methods on several numerical examples with real datasets and compare them with state-of-the-arts.
Our numerical experiments show that the new methods are comparable and, in many cases, outperform their competitors.
\end{abstract}


\beforesec
\section{Introduction}\label{sec:intro}
\aftersec
In this paper, we consider the following composite and possibly nonconvex optimization problem, which is widely studied in the literature:
\myeq{eq:ncvx_prob}{
\min_{x\in\R^p}\Big\{ F(x) := f(x) + \psi(x) \equiv \Exps{\xi\sim\Probn}{f_{\xi}(x)} + \psi(x) \Big\},
}
where $f_{\xi}(\cdot) : \R^p\times \Omega \to \R$ is a stochastic function such that for each $x\in\R^p$, $f_{\xi}(x)$ is a random variable in a given probability space $(\Omega, \Probn)$,  while for each realization $\xi\in\Omega$,  $f_{\xi}(\cdot)$ is smooth on $\R^p$; $f(x) := \Exps{\xi\sim\Probn}{f_{\xi}(x)} = \int_{\Omega}f_{\xi}(x)d\Probn{(\xi)}$ is the expected value of the random function $f_{\xi}(x)$ over $\Omega$; and $\psi : \R^p\to\Rext$ is a proper, closed, and convex function.

In addition to \eqref{eq:ncvx_prob}, we also consider the following composite finite-sum problem:
\myeq{eq:finite_sum}{
\min_{x\in\R^p}\set{ F(x) := f(x) + \psi(x) \equiv \frac{1}{n}\sum_{i=1}^nf_i(x) + \psi(x) },
}
where $f_i : \R^p\to\R$ for $i=1,\cdots, n$ are all smooth functions.
Problem \eqref{eq:finite_sum} can be considered as a special case of \eqref{eq:ncvx_prob} when $\Omega$ is finite, i.e., $\Omega := \set{\xi_1, \xi_2, \cdots, \xi_n}$ and $f_i(x) := n\Prob{\xi=\xi_i}f_{\xi_i}(x)$.
Alternatively, \eqref{eq:finite_sum} can be viewed as a stochastic average approximation of \eqref{eq:ncvx_prob}.
If $n$ is extremely large such that evaluating the full gradient $\nabla{f}(x)$ and the function value $f(x)$ in \eqref{eq:finite_sum} is expensive, then, as usual, we refer to this setting as an online model.

If the regularizer $\psi$ is absent, then we obtain a smooth problem which has been widely studied in the literature.
As another special case, if $\psi$ is the indicator of a nonempty, closed, and convex set $\Xc$, i.e., $\psi := \delta_{\Xc}$, then \eqref{eq:ncvx_prob} also covers constrained nonconvex optimization problems.
In this paper, we do not make any assumption on $\psi$ except for convexity.

\beforesubsec
\subsection{\bf  Our goals, approach, and contribution}
\aftersubsec
\noindent\mytxtbi{Our goals:} Our goal is to develop a new approach to approximate a stationary point of \eqref{eq:ncvx_prob} and its finite-sum setting \eqref{eq:finite_sum} under standard assumptions  used in existing methods.
In this paper, we only focus on Stochastic Gradient Descent-type (SGD) algorithms.
We are also interested in both oracle complexity bounds and implementation aspects.
The ultimate goal is to design simple algorithms (e.g., with a single loop) that are easy to implement and require less parameter tuning effort.

\noindent\mytxtbi{Our approach:}
Our approach relies on a so-called ``hybrid'' idea which merges two existing stochastic estimators through a convex combination to design a ``hybrid'' offspring that inherits the advantages of its underlying estimators.
We will focus on the hybrid estimators formed from the SARAH (StochAstic Recursive grAdient algoritHm) estimator introduced in \cite{nguyen2017sarah} and any given unbiased estimator such as SGD \cite{RM1951}, SVRG (Stochastic Variance Reduced Gradient) \cite{johnson2013accelerating}, or SAGA (stochastic incremental gradient)\cite{Defazio2014}.
For the sake of presentation, we only focus on the SGD estimator in this paper, but our idea can be extended to any unbiased estimator.
We emphasize that our method is fundamentally different from momentum or  exponential moving average-type methods such as in \cite{Cutkosky2019,Kingma2014} where we use two independent estimators instead of a combination of the past  and the current estimators.

While our hybrid estimators are biased, fortunately, they possess some useful properties to develop new stochastic optimization algorithms.
One important attribute is the variance reduced property which often allows us to derive larger or constant step-size in stochastic methods.
Whereas a majority of stochastic algorithms rely on unbiased estimators such as SGD, SVRG, and SAGA, interestingly, recent evidence has shown that biased estimators such as SARAH, biased SAGA, or biased SVRG estimators  also provide comparable or even better algorithms in terms of oracle complexity bounds as well as empirical performance, see, e.g., \cite{driggs2019bias,fang2018spider,Nguyen2019_SARAH,Pham2019,wang2018spiderboost}.

Our approach, on the one hand, can be extended to study second-order methods such as cubic regularization and subsampled schemes as in \cite{bollapragada2016exact,erdogdu2015convergence,Roosta-Khorasani2016,wang2018stochastic,zhang2018adaptive,zhou2019stochastic}.
The main idea is to exploit hybrid estimators to approximate both gradient and Hessian of the objective function similar to \cite{wang2018stochastic,zhang2018adaptive,zhou2019stochastic}.
On the other hand, it can be applied to approximate a second-order stationary point of  \eqref{eq:ncvx_prob} and  \eqref{eq:finite_sum}.
The idea is to integrate our methods with a negative curvature search such as Oja's algorithm \cite{oja1982simplified} or Neon2 \cite{allen2018neon2}, or to employ perturbed/noise gradient techniques such as \cite{fang2019sharp,ge2019stabilized,Li2019a} to approximate a second-order stationary point.

\noindent\mytxtbi{Our contribution:}
To this end, our contribution can be summarized as follows:
\begin{compactitem}
\item[(a)] 
We first introduce a ``hybrid'' approach to merge two stochastic estimators in order to form a new one.
Such a new estimator can be viewed as a convex combination of a biased estimator and an unbiased one to inherit the advantages of its underlying estimators.
Although we only focus on a convex combination between SARAH \cite{nguyen2017sarah} and SGD \cite{RM1951}  estimator, our approach can be extended to cover other possibilities such as SVRG \cite{johnson2013accelerating} or SAGA \cite{Defazio2014}.
Given such a new hybrid estimator, we develop several fundamental properties that can be useful for developing new stochastic optimization algorithms.

\item[(b)] 
Next, we employ our new hybrid SARAH-SGD estimator to develop a novel stochastic proximal gradient algorithm, Algorithm~\ref{alg:A1}, to solve \eqref{eq:ncvx_prob}.
\revise{This algorithm has a single-loop, and if the variance $\sigma > 0$, then it achieves $\BigO{\sigma^3\varepsilon^{-1} + \sigma\varepsilon^{-3}}$-oracle complexity bound.
When $\sigma = 0$ (i.e., no randomness involved in \eqref{eq:ncvx_prob}), its complexity bound reduces to $\BigO{\varepsilon^{-2}}$ as in the deterministic setting.}
To the best of our knowledge, this is the \mytxtbi{first variant} of proximal SGD methods that achieves such an oracle complexity bound without using double loop or check-points as in SVRG or SARAH, or requiring an $n\times p$-table to store gradient components as in SAGA-type algorithms.

\item[(c)] 
Then, we derive two different variants of Algorithm~\ref{alg:A1}: adaptive step-size and restarting schemes.
Both variants have similar complexity bounds as of Algorithm~\ref{alg:A1}.
We also propose a mini-batch variant of Algorithm~\ref{alg:A1} and provide a trade-off analysis between mini-batch sizes and the choice of step-sizes to obtain better practical performance.
\end{compactitem}

Let us emphasize the following additional points of our contribution.
Firstly, the new algorithm, Algorithm~\ref{alg:A1}, is rather different from existing SGD methods.
It first forms a mini-batch stochastic gradient estimator at a given initial point to provide a good approximation to the initial gradient of $f$.
Then, it performs a single loop to update the iterate sequence which consists of two steps: proximal-gradient step and averaging step, where our hybrid estimator is used.
The algorithm therefore has two step-sizes to be updated.

Secondly, our methods work with both single-sample and mini-batch cases, and achieve the best-known complexity bounds in both cases. 
This is different from some existing methods such as SVRG-type  \cite{reddi2016proximal}, SpiderBoost  \cite{wang2018spiderboost}, \revise{and SNVRG \cite{zhou2018stochastic}} that only achieve the best complexity under certain choices of parameters.
Our methods  are also flexible to choose different mini-batch sizes for the hybrid components to achieve different complexity bounds and to adjust the performance.
For instance, in Algorithm~\ref{alg:A1}, we can choose single sample in the SARAH estimator while using a mini-batch in the SGD estimator or vice versa that leads to different trade-off on the choice of the weight as well as the step-sizes.

Finally, our theoretical results on hybrid estimators are also self-contained and independent.
As we have mentioned, they can be used to develop other stochastic algorithms such as second-order methods or perturbed SGD schemes.
We believe that they can also be used in other problems such as compositional and constrained optimization \cite{davis2017proximally,mokhtari2018escaping,wang2017stochastic}.

\beforesubsec
\subsection{\bf  Related work}
\aftersubsec
Both problems \eqref{eq:ncvx_prob} and \eqref{eq:finite_sum} have been widely studied in the literature for both convex and nonconvex models, see, e.g., \cite{bottou2010large,Bottou1998,Defazio2014,ghadimi2013stochastic,johnson2013accelerating,mairal2015incremental,Nemirovski2009,nguyen2017sarah,RM1951,schmidt2017minimizing}.
However, due to applications in deep learning, large-scale nonconvex optimization problems have attracted huge attention in recent years \cite{goodfellow2016deep,lecun2015deep,Unser2019}.
Numerical methods for solving these problems heavily rely on two approaches: deterministic and stochastic approaches, ranging from first-order to second-order methods.
Notable first-order methods include stochastic gradient-type, conditional gradient \cite{reddi2016stochastic}, incremental gradient \cite{Bertsekas2011}, and primal-dual schemes \cite{chambolle2017stochastic}.
In contrast, advanced second-order methods consist of quasi-Newton, trust-region, sketching Newton, subsampled Newton, and cubic regularized Newton-based methods, see, e.g., \cite{byrd2016stochastic,Nesterov2006a,pilanci2015newton,Roosta-Khorasani2016}.

In terms of stochastic first-order algorithms, there has been a tremendously increasing trend in stochastic gradient methods and their variants in the last fifteen years.
SGD-based algorithms can be classified into two categories: non-variance reduction and variance reduction schemes.
The classical SGD method was studied in the early work of Robbins and Monro \cite{RM1951}, and then, e.g., in \cite{Polyak1992} with an accelerated variant via averaging steps, but its convergence rate was then investigated in \cite{Nemirovski2009} under new robust variants.
Ghadimi and Lan extended SGD to nonconvex settings and  analyzed its complexity in \cite{ghadimi2013stochastic}.
Other extensions of SGD can be found, e.g., in \cite{allen2017natasha,davis2017proximally,AdaGrad,ge2015escaping,ghadimi2016accelerated,jofre2019variance,Kingma2014,Konecny2016,Moulines2011,Nguyen2018a,Polyak1992}.

Alternatively, variance reduction-based methods have been intensively studied in recent years for both convex and nonconvex settings.
Apart from mini-batch and importance sampling schemes \cite{ghadimi2016mini,zhao2015stochastic}, the following methods are the most notable.
The first class of algorithms is based on SAG estimator \cite{schmidt2017minimizing}, including SAGA-variants \cite{Defazio2014}.
The second one is SVRG \cite{johnson2013accelerating} and its variants such as Katyusha \cite{allen2016katyusha}, MiG \cite{zhou2018simple}, and many others \cite{li2018simple,reddi2016proximal}.
The third class relies on SARAH \cite{nguyen2017sarah} such as SPIDER \cite{fang2018spider}, SpiderBoost \cite{wang2018spiderboost}, ProxSARAH \cite{Pham2019}, and momentum variants  \cite{zhou2019momentum}.
\revise{The fourth idea is SNVRG \cite{zhou2018stochastic}, which combines different techniques such as nested variance reduction and sampling without replacement.}
Other approaches such as Catalyst \cite{lin2015universal},  SDCA \cite{shalev2013stochastic}, and SEGA \cite{hanzely2018sega} have also been proposed.
These algorithms often require stronger assumptions than SGDs.

In terms of theory, many researchers have focused on theoretical aspects of existing algorithms.
For example, \cite{ghadimi2013stochastic} appeared as one of the first remarkable works studying convergence rates of stochastic gradient-type methods for nonconvex and non-composite finite-sum problems.
They later extended it to the composite setting in \cite{ghadimi2016mini}.
The authors of \cite{wang2018spiderboost}  also investigated the gradient dominant case, and \cite{karimi2016linear} considered both finite-sum and composite finite-sum problems under different assumptions.
Whereas many researchers have been trying to improve complexity upper bounds of stochastic first-order methods using different techniques \cite{allen2017natasha2,allen2018neon2,Allen-Zhu2016,fang2018spider},
other works have attempted to  construct examples to establish lower-bound complexity barriers.
The upper oracle complexity bounds have been substantially improved among these works and some results have matched the lower bound complexity in both convex and nonconvex settings \cite{allen2017natasha,allen2017natasha2,arjevani2019lower,fang2018spider,ghadimi2013stochastic,li2018simple,lei2017non,Pham2019,reddi2016proximal,wang2018spiderboost,zhou2018stochastic}.
We refer to Table~\ref{tbl:SFO_compare} for some notable examples of stochastic gradient-type methods for solving \eqref{eq:ncvx_prob} and \eqref{eq:finite_sum} and their non-composite settings.
In fact, \cite{lei2017non} and \cite{zhou2018stochastic} only study the finite-sum problem with an additional bounded variance assumption, but allow the variance to go to infinite.

\begin{table}[hpt!]
\newcommand{\cellb}[1]{{\!\!}{\color{blue}#1}{\!\!}}
\newcommand{\cellr}[1]{{\!\!}{\color{red}#1}{\!\!}}
\newcommand{\cell}[1]{{\!\!\!}#1{\!\!\!}}
\begin{center}
\resizebox{\textwidth}{!}{
\begin{tabular}{l | c | c | c | c}\toprule
\cell{~~~\textbf{Algorithms}} & \cell{\bf Expectation} &  \cell{\bf Finite-sum} & \cell{\bf Composite} & \cell{\bf Type} \\ \toprule
\cell{GD \cite{Nesterov2004}} & NA & \cell{$\BigO{n\varepsilon^{-2}}$} & \cell{\done}   & \cell{Single} \\ \midrule
\cell{SGD \cite{ghadimi2013stochastic}} & \cell{$\BigO{\sigma^2\varepsilon^{-4}}$} & NA & \cell{\done}   & \cell{Single} \\ \midrule
\cell{SAGA \cite{reddi2016proximal}} &  NA & $\BigO{n + n^{2/3}\varepsilon^{-2}}$ & \cell{\done}    & \cell{Single$^{*}$}  \\ \midrule
\cell{SVRG  \cite{reddi2016proximal}} &  NA & $\BigO{n + n^{2/3}\varepsilon^{-2}}$ & \cell{\done}   & \cell{Double}  \\ \midrule
\cell{SVRG+ \cite{li2018simple}} & \cell{$\BigO{\sigma^2\varepsilon^{-10/3}}$}  & \cell{$\BigO{n+n^{2/3}\varepsilon^{-2}}$} & \cell{\done}    & \cell{Double}  \\ \midrule
\cell{SCSG \cite{lei2017non}} & \cell{N/A} & \cell{$\BigO{\left( \frac{\sigma^2}{\varepsilon^2} \wedge n \right) + \frac{1}{\varepsilon^2} \left( \frac{\sigma^2}{\varepsilon^2} \wedge n \right)^{2/3}}$} & \cell{\wontfix}    & \cell{Double}  \\ \midrule
\cell{SNVRG \cite{zhou2018stochastic}} & \cell{N/A} & \cell{$\BigO{\log^3\left( \frac{\sigma^2}{\varepsilon^2} \wedge n \right) \left[ \left( \frac{\sigma^2}{\varepsilon^2} \wedge n \right) + \frac{1}{\varepsilon^2} \left( \frac{\sigma^2}{\varepsilon^2} \wedge n \right)^{1/2} \right] }$} & \cell{\wontfix}  & \cell{Double}  \\ \midrule
\cell{SPIDER \cite{fang2018spider}} & \cell{$\BigO{\sigma^2\varepsilon^{-2}+\sigma\varepsilon^{-3}}$} &  \cell{$\BigO{n+n^{1/2}\varepsilon^{-2}}$} & \cell{\wontfix}   &  \cell{Double} \\ \midrule
\cell{SpiderBoost \cite{wang2018spiderboost}} &  \cell{$\BigO{\sigma^2\varepsilon^{-2}+\sigma\varepsilon^{-3}}$} & \cell{$\BigO{n+n^{1/2}\varepsilon^{-2}}$} & \cell{\done}   & \cell{Double} \\ \midrule
\cell{ProxSARAH \cite{Pham2019}} & \cell{$\BigO{\sigma^2\varepsilon^{-2} \vee \sigma\varepsilon^{-3}}$} &  \cell{$\BigO{n+n^{1/2}\varepsilon^{-2}}$} & \cell{\done}    & \cell{Double} \\ \midrule
\cellb{This paper} &  \cellb{$\BigO{\sigma^3\varepsilon^{-1} + \sigma\varepsilon^{-3}}$} & \cellb{$\BigO{n + \varepsilon^{-3}}$} & \cellb{\done}  &  \cellb{Single} \\ 
\bottomrule
\end{tabular}}
\caption{A comparison of stochastic first-order oracle complexity bounds and the type of algorithms for nonsmooth nonconvex optimization $($both non-composite and composite case$)$.
{\small
Here, $n$ is the number of data points and $\sigma$ is the variance in Assumption~\ref{as:A1b}, and ``single/double'' means that the algorithm uses single-loop or double-loop, respectively.
All the complexity bounds here must depend on the Lipschitz constant $L$ in Assumption~\ref{as:A1} and $F(x_0) - F^{\star}$, the difference between the initial objective value $F(x_0)$ and the lower-bound $F^{\star}$ in Assumption~\ref{as:A0}. 
We assume that $L = \BigO{1}$ and ignore these quantities in the complexity bounds.
\revise{In addition, we also assume that $\sigma > 0$ and dominates $L$ and $F(x_0) - F^{\star}$ to simplify the bounds.}
Note that SAGA is a single-loop method, but it requires a matrix of size $n\times p$ to store stochastic gradients~$(^{*})$.
}
}\label{tbl:SFO_compare}
\end{center}
\vspace{-6ex}
\end{table}

In the convex case, there exist numerous research papers including \cite{agarwal2010information,agarwal2015lower,carmon2017lower,foster2019complexity,Nemirovskii1983,Nesterov2004,woodworth2016tight} that study the lower bound complexity.
In \cite{fang2018spider,zhou2019lower}, the authors  constructed a lower-bound complexity for nonconvex finite-sum problems covered by \eqref{eq:finite_sum}.
They showed that the lower-bound complexity for any stochastic gradient method relied on only smoothness assumption to achieve an $\varepsilon$-stationary point in expectation  is $\Omega\left({n^{1/2}\varepsilon^{-2}}\right)$.
For the expectation problem \eqref{eq:ncvx_prob}, the best-known complexity bound to obtain an $\varepsilon$-stationary point in expectation is $\BigO{\sigma\varepsilon^{-3} + \sigma^2\varepsilon^{-2}}$ as shown in \cite{fang2018spider,wang2018spiderboost}, where \revise{$\sigma > 0$} is an upper bound of the variance (see Assumption~\ref{as:A1b}).
\revise{Very recently, \cite{allen2017natasha2} provides a study on lower-bound complexity for the non-composite and nonconvex setting of \eqref{eq:ncvx_prob} under different sets of assumptions.}

While stochastic algorithms for solving the non-composite setting, i.e., $\psi = 0$, are well-developed and have received considerable attention \cite{allen2017natasha2,allen2018neon2,Allen-Zhu2016,fang2018spider,lei2017non,Nguyen2018_iSARAH,Nguyen2019_SARAH,nguyen2017sarahnonconvex,reddi2016proximal,zhou2018stochastic}, methods for composite setting remain limited \cite{reddi2016proximal,wang2018spiderboost}.
In this paper, we will develop a novel approach to design stochastic optimization algorithms for solving the composite problems \eqref{eq:ncvx_prob} and \eqref{eq:finite_sum}.
Our approach is rather different from existing ones and we call it a ``hybrid'' approach.

\beforesubsec
\subsection{\bf Comparison}
\aftersubsec
Let us compare our algorithms and existing methods in the following aspects:
\beforepara
\paragraph{\mytxtbi{$\mathrm{(a)}$ Single-loop vs. multiple-loop:}}
As mentioned, we aim at developing practical methods that are easy to implement.
One major difference between our methods and existing state-of-the-arts is the algorithmic style: single-loop vs. multiple-loop.
As discussed in several works, including \cite{kovalev2019don}, single-loop methods have notable advantages over double-loop methods, including tuning parameters.
The single-loop style consists of SGD, SAGA, and their variants \cite{Defazio2014,defazio2014finito,ghadimi2013stochastic,Nemirovski2009,Polyak1992,RM1951,schmidt2017minimizing}, while the double-loop style comprises SVRG, SARAH, and their variants \cite{johnson2013accelerating,nguyen2017sarah}.
Other algorithms such as Natasha \cite{allen2017natasha} or Natasha1.5 \cite{allen2017natasha2} even have three loops.
Let us compare these methods in detail as follows:
\begin{compactitem}
\item[$\bullet$] SGD and SAGA-type methods have single-loop, but SAGA-type algorithms use an $n\times p$-matrix to maintain $n$ individual gradients which can be very large if $n$ and $p$ are large.
In addition, SAGA has not yet been applied to solve \eqref{eq:ncvx_prob}.
Our first algorithm, Algorithm~\ref{alg:A1}, has single-loop as SGD and SAGA, and does not require heavy memory storage.
However, to apply to \eqref{eq:finite_sum}, it still requires either an additional bounded variance condition (see \eqref{eq:L_smooth2}) compared to SAGA.
But if it solves \eqref{eq:ncvx_prob}, then it requires the same standard assumptions as used in many existing variance reduction methods.
In terms of complexity, Algorithm~\ref{alg:A1} is much better than SGD. 
But as a compensation, it uses a slightly stronger assumption: $L$-average smoothness in Assumption~\ref{as:A1} compared to SGD, which only requires the $L$-smoothness of the expected value function $f$. 
To the best of our knowledge, Algorithm~\ref{alg:A1} is the first single-loop SGD variant that achieves the best-known complexity.
Another related and concurrent work is \cite{Cutkosky2019}, which uses momentum approach, but requires additional bounded gradient assumption to achieve similar complexity as Algorithm~\ref{alg:A1}.

\item[$\bullet$] Algorithm~\ref{alg:A2} has double-loop similar to SVRG and SARAH-type methods.
While the double-loop in SVRG, SARAH, and their variants are required to achieve convergence, it is optional in Algorithm~\ref{alg:A2} as a restarting variant and no parameter tuning is needed.
Note that double-loop or multiple-loop methods require to tune more parameters such as epoch lengths and possibly the mini-batch size of the snapshot point.
In addition, the choice of the snapshot point also matters.
\end{compactitem}
\beforepara
\paragraph{\mytxtbi{$\mathrm{(b)}$ Single-sample and mini-batch:}}
Our methods work with both single sample and mini-batch, and in both cases, they achieve the best-known complexity bounds \cite{arjevani2019lower}.
This is different from some existing methods such as SVRG, SNVRG, or SARAH-based methods \cite{reddi2016proximal,wang2018spiderboost,zhou2018stochastic} where the best complexity is only obtained under an appropriate parameter selection.

\beforepara
\paragraph{\mytxtbi{$\mathrm{(c)}$ Oracle complexity bounds:}}
Algorithm~\ref{alg:A1} and its variants all achieve the best-known complexity bounds as in \cite{Pham2019,wang2018spiderboost} for solving \eqref{eq:ncvx_prob}.
In early works such as Natasha  \cite{allen2017natasha} and Natasha1.5 \cite{allen2017natasha2} which are based on the SVRG estimator, the best complexity is often $\BigO{\sigma^2\varepsilon^{-2}+\sigma\varepsilon^{-10/3}}$ for solving \eqref{eq:ncvx_prob} and $\BigO{n + n^{2/3}\varepsilon^{-2}}$ for solving \eqref{eq:finite_sum}.
By combining with additional sophisticated tricks, these complexity bounds are slightly improved.
For instance, Natasha \cite{allen2017natasha} or Natasha1.5 \cite{allen2017natasha2} can achieve  $\BigO{n + n^{2/3}\varepsilon^{-2}}$ in the finite-sum case and $\BigO{\varepsilon^{-3} + \sigma^{1/3}\varepsilon^{-10/3}}$ in the expectation case, but they require three loops with several parameter adjustment, which are difficult to tune in practice.
\revise{SNVRG \cite{zhou2018stochastic} exploits a nested variance reduction technique with dynamic epoch lengths as used in \cite{lei2017non} to improve its complexity bounds.
However,  \cite{zhou2018stochastic} only focuses on the non-composite finite-sum problems, and its final complexity bound is $\BigO{(n+n^{1/2}\varepsilon^{-2})\log^3(n)}$.}
Again, this method also requires complicated parameter selection procedure.
To achieve better complexity bounds, SARAH-based methods have been studied in \cite{fang2018spider,Nguyen2019_SARAH,Pham2019,wang2018spiderboost}.
Their oracle complexity meets  the lower-bound one (up to a constant factor), see \cite{arjevani2019lower,fang2018spider,Pham2019}.

\beforesubsec
\subsection{\bf Paper organization}
\aftersubsec
The rest of this paper is organized as follows.
Section~\ref{sec:basic_tools} discusses the main assumptions of our problems \eqref{eq:ncvx_prob} and \eqref{eq:finite_sum}, and their optimality conditions.
Section~\ref{sec:stochastic_estimators} develops new hybrid stochastic estimators and investigates their properties.
We consider both single-sample and mini-batch cases.
Section~\ref{sec:sgd_algs} studies a new class of hybrid gradient algorithms to solve both \eqref{eq:ncvx_prob} and \eqref{eq:finite_sum}.
We develop three different variants of hybrid algorithms and analyze their convergence and complexity estimates.
Section~\ref{sec:sgd_mini_batch} extends our algorithms to mini-batch cases.
Section~\ref{sec:num_experiments} gives several numerical examples and compares our methods with existing state-of-the-arts.
For the sake of presentation, many technical proofs are provided in the appendix.

\beforesec
\section{Basic Notations, Fundamental Assumptions, and Optimality Condition}\label{sec:basic_tools}
\aftersec
In this section, we first recall some basic notation and concepts.
Then, we state the fundamental assumptions imposed on  \eqref{eq:ncvx_prob} and \eqref{eq:finite_sum} and their optimality condition.
\beforesubsec
\subsection{\bf  Notations and basic concepts}
\aftersubsec
We work with the Euclidean spaces, $\R^p$ and $\R^n$, equipped with standard inner product $\iprods{\cdot,\cdot}$ and Euclidean norm $\norm{\cdot}$.
For any function $f$, $\dom{f} := \set{ x\in\R^p \mid f(x) < +\infty}$ denotes the effective domain of $f$. 
If $f$ is continuously differentiable, then $\nabla{f}$ denotes  its gradient.
If, in addition, $f$ is twice continuously differentiable, then $\nabla^2{f}$ denotes its Hessian.

For a stochastic function $f_{\xi}$ defined on a probability space $(\Omega, \mathbb{P})$, we use $\Exps{\xi}{f_{\xi}} := \Exps{\xi\sim\Probn}{f_{\xi}}$ to denote the expected value or vector of $f_{\xi}$ w.r.t. $\xi$ on $\Omega$.
We also overload the notation $\Exps{\xi_t}{\cdot}$ to express the expected value w.r.t. a realization $\xi_t$ in both single-sample and mini-batch cases.
Given a finite set $\Sc_m := \set{s_1, \cdots, s_m}$, we denote $s \sim \Unip{\pb}{\Sc_m}$ if $\Prob{s = s_i} = \pb_i$ for $\pb_i > 0$ and $\sum_{i=1}^m\pb_i = 1$.
If $\pb_i = \frac{1}{m}$ for $i=1,\cdots, m$, then we write $s \sim \Uni{\Sc_m}$ by dropping the probability distribution $\pb$.

Given a stochastic mapping $G : \R^p\times \Omega \to\R^q$ depending on a random vector $\xi\in\Omega$, we say that $G$ is $L$-average Lipschitz continuous if $\Exps{\xi}{\norms{G(x) - G(y)}^2} \leq L^2\norms{x - y}^2$ for all $x, y$, where $L \in (0, +\infty)$ is called the Lipschitz constant of $G$.
If $G$ is a deterministic function, then this condition becomes $\norms{G(x) - G(y)} \leq L\norms{x-y}$ stated that $G$ is $L$-Lipschitz continuous.
In particular, if this condition holds for $G = \nabla{f}$, then we say that $f$ is $L$-smooth.

For a proper, closed, and convex function $\psi : \R^p\to\Rext$,  $\partial{\psi}(x) := \{w\in\R^p \mid \psi(y) \geq \psi(x) + \iprods{w, y-x},~\forall y\in\dom{f} \}$ denotes its subdifferential at $x$, and $\prox_{\psi}(x) := \argmin_u\set{ \psi(x) + \tfrac{1}{2}\norms{u - x}^2}$ denotes its proximal operator. 
Note that $\prox_{\eta\psi}$ is nonexpansive, i.e., $\norms{\prox_{\eta\psi}(x) - \prox_{\eta\psi}(y)} \leq \norms{x - y}$ for all $x,y\in\dom{\psi}$.
If $\psi$ is the indicator $\delta_{\Xc}$ of a nonempty, closed, and convex set $\Xc$, then $\prox_{\delta_{\Xc}}$ reduces to the Euclidean projection $\proj_{\Xc}$ onto $\Xc$.

If $x$ is a matrix, then $\norm{x}$ is the spectral norm of $x$ and the inner product of two matrices $x$ and $y$ is defined as $\iprods{x,y} := \trace{x^{\top}y}$.
Also, $\mathbb{N}_{+}$ stands for the set of positive integer numbers, and $[n] := \set{1, 2, \cdots, n}$.
Given $a\in\R$, $\lfloor a\rfloor$ denotes the maximum integer number that is less than or equal to $a$.
We also use $\BigO{\cdot}$ to express complexity bounds of algorithms.

\beforesubsec
\subsection{\bf  Fundamental assumptions}
\aftersubsec
Our algorithms developed in the sequel rely on the following fundamental assumptions:

\begin{assumption}\label{as:A0}
Both problems \eqref{eq:ncvx_prob} and \eqref{eq:finite_sum} satisfy the following conditions:
\begin{compactitem}
\item[$\mathrm{(a)}$] $($\textbf{Convexity of the regularizer}$)$ $\psi : \R^p\to\Rext$ is a proper, closed, and convex function.
The domain $\dom{F} := \dom{f}\cap\dom{g}$ is nonempty.

\item[$\mathrm{(b)}$] $($\textbf{Boundedness from below}$)$ There exists a finite lower bound 
\myeq{eq:lower_bound}{
F^{\star} := \inf_{x\in\R^p}\Big\{ F(x) := f(x) + \psi(x) \Big\} > -\infty.
}
\end{compactitem}
\end{assumption}
Assumption~\ref{as:A0}(b) is fundamental and required for any algorithm.
Here, since $\psi$ is proper, closed, and convex, its proximal operator $\prox_{\eta\psi}(\cdot)$ is well-defined, single-valued, and non-expansive.
We assume that this proximal operator can be computed exactly.

\begin{assumption}[\textbf{$L$-average smoothness}]\label{as:A1} 
The expected value function $f(\cdot)$ in \eqref{eq:ncvx_prob}  is $L$-smooth on $\dom{F}$, i.e., there exists $L\in (0, +\infty)$ such that
\myeq{eq:L_smooth}{
\Exps{\xi}{\norm{\nabla{f}_{\xi}(x) - \nabla{f}_{\xi}(y)}^2} \leq L^2\norms{x - y}^2,~~\forall x, y\in\dom{F}.
}
In the finite sum setting \eqref{eq:finite_sum}, the $L$-smoothness condition \eqref{eq:L_smooth} can be expressed as the $L$-average smoothness of  all $f_i$ with the moduli $L$ as:
\myeq{eq:L_smooth2}{
\frac{1}{n}\sum_{i=1}^n\norm{\nabla{f_i}(x) - \nabla{f_i}(y)}^2 \leq L^2\norms{x - y}^2,~~\forall x, y\in\dom{F}.
}
\end{assumption}

\begin{assumption}[\textbf{Bounded variance}]\label{as:A1b}
There exists $\sigma \in [0, \infty)$ such that 
\myeq{eq:bounded_variance2}{
\Exps{\xi}{\norms{\nabla{f}_{\xi}(x) - \nabla{f}(x)}^2} \leq \sigma^2, ~~~\forall x\in\dom{F}.
}
The bounded variance condition for \eqref{eq:finite_sum} becomes
\myeq{eq:bounded_variance2}{
\Exps{i}{\norms{\nabla{f}_i(x) - \nabla{f}(x)}^2} \equiv \frac{1}{n}\sum_{i=1}^n\norms{\nabla{f_i}(x) - \nabla{f}(x)}^2 \leq \sigma^2, ~~~\forall x\in\dom{F}.
}
\end{assumption}
\revise{Assumptions \ref{as:A1} and \ref{as:A1b} are widely used in stochastic optimization methods. 
They or their stronger versions are required for all existing variance reduced-based stochastic gradient-based methods for solving \eqref{eq:ncvx_prob}.
The $L$-average smoothness in \eqref{eq:L_smooth2} is also called \mytxtit{mean-squared smoothness}, see, e.g., \cite{arjevani2019lower}.
In the finite-sum setting \eqref{eq:finite_sum}, if $f_i$ is $L_i$-smooth, then $f$ is also $L$-average smooth with $L = \frac{1}{\sqrt{n}}(\sum_{i=1}^nL_i^2)^{1/2}$.
Conversely, if $f$ is $L$-average smooth, then $f_i$ is also $L$-smooth with $L_i = \sqrt{n}L$.
Therefore, the $L$-average smoothness constant is generally smaller than the one derived from the individual smoothness constant of each $f_i$ \cite{Pham2019}.

The $L$-average smoothness is stronger than the $L$-smoothness of the expected value function $f$ used in standard SGD schemes \cite{ghadimi2013stochastic}.
Indeed, the $L$-average smoothness of $f$ implies the $L$-smoothness of the expected value function $f$ due to Jensen's inequality $\Vert \Exps{\xi}{\nabla{f}_{\xi}(x) - \nabla{f}_{\xi}(y)}\Vert^2 \leq \Exps{\xi}{\norm{\nabla{f}_{\xi}(x) - \nabla{f}_{\xi}(y)}^2}$, but not vice versa.
Fortunately, Assumption~\ref{as:A1} holds for many optimization problems in machine learning and statistics such as binary classification, linear regression, and neural networks.
If $\sigma = 0$, then \eqref{eq:ncvx_prob} reduces to a deterministic setting, while \eqref{eq:bounded_variance2} forces $f_i = f$ for all $i=1,\cdots, n$ in \eqref{eq:finite_sum}.
}

\beforesubsec
\subsection{\bf  First-order optimality condition}
\aftersubsec
The optimality condition of \eqref{eq:ncvx_prob} (or \eqref{eq:finite_sum}) can be written as
\myeq{eq:opt_cond}{
0 \in \nabla{f}(x^{\star}) + \partial{\psi}(x^{\star}).  
}
Any point $x^{\star}$ satisfying \eqref{eq:opt_cond} is called a stationary point of \eqref{eq:ncvx_prob} (or \eqref{eq:finite_sum}).
Note that \eqref{eq:opt_cond} can be written equivalently to 
\myeq{eq:grad_map}{
\Grad_{\eta}(x^{\star}) := \frac{1}{\eta}\left(x^{\star} - \prox_{\eta\psi}(x^{\star} - \eta\nabla{f}(x^{\star}))\right) = 0.
}
Here, $\Grad_{\eta}$ is called the gradient mapping of $F$ in \eqref{eq:ncvx_prob} for any $\eta > 0$.
It is obvious that if $\psi = 0$, then $\Grad_{\eta}(x) = \nabla{f}(x)$, the gradient of $f$ at $x$.
Our goal is to seek an $\varepsilon$-stationary point $\overline{x}_T$ of \eqref{eq:ncvx_prob} or \eqref{eq:finite_sum} defined as follows:

\begin{definition}\label{de:stationary_point}
\textit{
Given a desired acuracy $\varepsilon > 0$, a point $\overline{x}_T\in\dom{F}$ is said to be an $\varepsilon$-stationary point of \eqref{eq:ncvx_prob} or \eqref{eq:finite_sum} if
\myeq{eq:approx_opt_cond}{
\Exp{\norms{\Grad_{\eta}(\overline{x}_T)}^2} \leq \varepsilon^2.
}
Here, the expectation is taken over all the randomness rendered from both $\xi$ in \eqref{eq:ncvx_prob} and the algorithm that is solving the problem. 
}
\end{definition}

Let us clarify why $\overline{x}_T$ is an approximate stationary point of \eqref{eq:ncvx_prob} (or \eqref{eq:finite_sum}).
Indeed, if $\overline{x}_T^{+} := \prox_{\eta\psi}(\overline{x}_T - \eta\nabla{f}(\overline{x}_T))$, then $\Exp{\norms{\Grad_{\eta}(\overline{x}_T)}^2} \leq \varepsilon^2$ leads to $\Exp{\norms{\overline{x}_T^{+} - \overline{x}_T}^2} \leq \eta^2\varepsilon^2$.
On the other hand, $\overline{x}_T^{+} = \prox_{\eta\psi}(\overline{x}_T - \eta\nabla{f}(\overline{x}_T))$ is equivalent to $\frac{1}{\eta}(\overline{x}_T - \overline{x}_T^{+}) \in \nabla{f}(\overline{x}_T) + \partial{\psi}(\overline{x}_T^{+})$.
Therefore, $\norms{\nabla{f}(\overline{x}_T^{+}) + \nabla{\psi}(\overline{x}_T^{+})} \leq \norms{\nabla{f}(\overline{x}_T^{+}) - \nabla{f}(\overline{x}_T)} + \frac{1}{\eta}\norms{\overline{x}_T^{+} - \overline{x}_T}$ for some $\nabla{\psi}(\overline{x}_T^+) \in \partial{\psi}(\overline{x}_T^+)$.
Using the $L$-average smoothness of $f$, we have 
\myeqn{
\Exp{\norms{\nabla{f}(\overline{x}_T^{+}) + \nabla{\psi}(\overline{x}_T^{+})}^2} \leq 2\left(L^2 + \tfrac{1}{\eta^2}\right)\Exp{\norms{\overline{x}_T^{+} - \overline{x}_T}^2} \leq 2(1+L^2\eta^2)\varepsilon^2.
}
This shows that $\overline{x}_T^+$ is an approximate stationary point of \eqref{eq:ncvx_prob}  (or \eqref{eq:finite_sum}).

In practice, we often replace the condition \eqref{eq:approx_opt_cond} by $\min_{0\leq t\leq T} \norms{\Grad_{\eta}(x_t)} \leq \varepsilon$, which leads to the ``best'' iterate convergence in expectation.

\beforesec
\section{Hybrid Stochastic Estimators: Definition and Fundamental Properties}\label{sec:stochastic_estimators}
\aftersec
In this section, we propose new stochastic estimators for a generic function $G$ that can cover  function value, gradient, and Hessian of any expected value function $f$ in \eqref{eq:ncvx_prob}.

\beforesubsec
\subsection{\bf  The construction of hybrid stochastic estimators}\label{subsec:biased_estimator} 
\aftersubsec
Given a function $G(x) := \Exps{\xi}{G_{\xi}(x)}$, where $G_{\xi}$ is a (vector) stochastic function from $\R^p\times\Omega\to\R^q$.
We define the following stochastic estimator of $G$.
As concrete examples, $G$ can be the gradient mapping $\nabla{f}$ of $f$ or the Hessian mapping $\nabla^2{f}$ of $f$ in problem \eqref{eq:ncvx_prob} or \eqref{eq:finite_sum}.

\begin{definition}\label{de:hybrid_estimator}
\textit{
Let $u_t$ be an unbiased stochastic estimator of $G(x_t)$ formed by a realization $\zeta_t$ of $\xi$, i.e., $\Exps{\zeta_t}{u_t} = G(x_t)$ at a given $x_t$.
The following quantity:
\myeq{eq:v_estimator}{
v_t :=  \beta_{t-1}v_{t-1} + \beta_{t-1}\left[G_{\xi_t}(x_t) - G_{\xi_t}(x_{t-1})\right] + (1-\beta_{t-1})u_t,
}
is called a \textit{hybrid stochastic estimator} of $G$ at $x_t$, where $\xi_t$ and $\zeta_t$ are two independent realizations of $\xi$ on $\Omega$  and $\beta_{t-1} \in [0, 1]$ is a given weight.
}
\end{definition}
We observe from \eqref{eq:v_estimator} two extreme cases as follows:
\begin{compactitem}
\item[$\bullet$] If $\beta_t = 0$, then we obtain a simple unbiased stochastic estimator, i.e., $v_t = u_t$. 
\item[$\bullet$] If $\beta_t = 1$, then we obtain the SARAH-type estimator as studied in \cite{nguyen2017sarah} but for general function $G$ instead of just gradient mappings, i.e., $v_t = v_{t-1} + G_{\xi_t}(x_t) - G_{\xi_t}(x_{t-1})$.
\end{compactitem}
In this paper, we are interested in the case $\beta_t \in (0, 1)$, which can be referred to as a hybrid recursive stochastic estimator.

Note that we can rewrite $v_t$ as
\myeqn{
v_t :=   \beta_{t-1}G_{\xi_t}(x_t) + (1-\beta_{t-1})u_t + \beta_{t-1}\left[ v_{t-1} - G_{\xi_t}(x_{t-1})\right].
}
The first two terms are two stochastic estimators evaluated at $x_t$, while the third term is the difference $\delta_{t-1} := v_{t-1} - G_{\xi_t}(x_{t-1})$ of the previous estimator and a stochastic estimator at the previous iterate.
Here, since $\beta_{t-1} \in (0, 1)$, the estimator $v_t$ can be viewed as the way of emphasizing on the new information at $x_t$ than the old one evaluated at $x_{t-1}$.

In fact, if $G = \nabla{f}$, then the hybrid estimator $v_t$ covers many other estimators, including SGD, SVRG, and SARAH as follows:
\begin{compactitem}
\item[$\bullet$] \textbf{The classical stochastic estimator:} $u_t := G_{\zeta_t}(x_t)$. 
\item[$\bullet$] \textbf{The SVRG estimator:} $u_t := u_t^{\mathrm{svrg}} = \overline{G}(\widetilde{x}) + G_{\zeta_t}(x_t) - G_{\zeta_t}(\widetilde{x})$, where $\overline{G}(\widetilde{x})$ is a given unbiased snapshot evaluated at a given point $\widetilde{x}$.
\item[$\bullet$] \textbf{The SAGA estimator:} $u_t := u_t^{\textrm{saga}} = G_{j_t}(y_{t+1}^{j_t}) - G_{j_t}(y_t^{j_t}) + \frac{1}{n}\sum_{i=1}^nG_i(y_t^i)$, where $y_{t+1}^{j_t} = x_t$ if $i = j_t$ and $y_{t+1}^i = y_t^i$ if $i\neq j_t$.
\end{compactitem}
While the classical stochastic and SVRG estimators can be used in both expectation and finite-sum settings, this SAGA estimator is \nhan{only} applicable to the finite-sum setting \eqref{eq:finite_sum}.

\beforesubsec
\subsection{\bf Properties of hybrid stochastic estimators}
\aftersubsec
Let us first define 
\myeq{eq:sigma_field}{
\Fc_t := \sigma\left(x_0, \xi_0,\zeta_0, \cdots, \xi_{t-1}, \zeta_{t-1}\right)
}
the $\sigma$-field generated by the history of realizations $\set{x_0, \xi_0,\zeta_0, \cdots, \xi_{t-1}, \zeta_{t-1}}$ of $\xi$ up to the iteration $t$.
We first prove  the following property of  the hybrid stochastic estimator $v_t$.

\begin{lemma}\label{le:key_estimate10}
Let $v_t$ be defined by \eqref{eq:v_estimator}.
Then
\myeq{eq:biased_estimator}{
\Exps{(\xi_t,\zeta_t)}{v_t} = G(x_t) + \beta_{t-1}\left[ v_{t-1} - G(x_{t-1})\right].
}
If $\beta_{t-1} \neq 0$, then $v_t$ is a biased estimator of $G(x_t)$.
Moreover, we have
\myeq{eq:key_estimate10}{
\begin{array}{lcl}
\Exps{(\xi_t,\zeta_t)}{\norms{v_t - G(x_t)}^2} &=&  \beta_{t-1}^2\norms{v_{t-1}  - G(x_{t-1})}^2 - \beta_{t-1}^2\norms{G(x_t) - G(x_{t-1})}^2 \vspace{1ex}\\
&& + {~} \beta_{t-1}^2\Exps{\xi_t}{\norms{G_{\xi_t}(x_t) - G_{\xi_t}(x_{t-1})}^2} \vspace{1ex}\\
&& + {~} (1-\beta_{t-1})^2\Exps{\zeta_t}{\norms{u_t - G(x_t)}^2}.
\end{array}
}
\end{lemma}

\begin{proof}
By taking the expectation of both sides in \eqref{eq:v_estimator} w.r.t. $(\xi_t,\zeta_t)$ and using the fact that $\xi_t$ and $\zeta_t$ are independent, we can easily obtain \eqref{eq:biased_estimator}.

Let us first denote $\delta_t := v_t - G(x_t)$, $\hat{\delta}_t := u_t - G(x_t)$, $\Delta_{\xi_t} := G_{\xi_t}(x_t) - G_{\xi_t}(x_{t-1})$, and $\Delta_t := G(x_t) - G(x_{t-1})$.
Clearly, $\Exps{\xi_t}{\Delta_{\xi_t}} = \Delta_t$ and $\Exps{\zeta_t}{\hat{\delta}_t} = 0$.
Next, we write
\myeqn{ 
\begin{array}{lcl}
\delta_t &:=& v_t - G(x_t) =  \beta_{t-1}(v_{t-1}  - G(x_{t-1})) + \beta_{t-1}(\Delta_{\xi_t} - \Delta_t)  + (1-\beta_{t-1})\big[u_t - G(x_t)\big] \vspace{1ex}\\
&= & \beta_{t-1}\delta_{t-1} + \beta_{t-1}(\Delta_{\xi_t} - \Delta_t) + (1-\beta_{t-1})\hat{\delta}_t.
\end{array}
}
In this case, we have
\myeqn{
\begin{array}{lcl}
\norms{\delta_t}^2 &=& \beta_{t-1}^2\norms{\delta_{t-1}}^2 + \beta_{t-1}^2\norms{\Delta_{\xi_t} - \Delta_t}^2 + (1-\beta_{t-1})^2\norms{\hat{\delta}_t}^2 \vspace{1ex}\\
&&+ {~} 2\beta_{t-1}^2\iprods{\delta_{t-1}, \Delta_{\xi_t} - \Delta_t} + 2\beta_{t-1}(1-\beta_{t-1})\iprods{\delta_{t-1}, \hat{\delta}_t} + 2\beta_{t-1}(1-\beta_{t-1})\iprods{\Delta_{\xi_t} - \Delta_t,\hat{\delta}_t}.
\end{array}
}
Taking expectation w.r.t. $\xi_t$ conditioned on $\zeta_t$, and noting that $\Exps{\xi_t}{\Delta_{\xi_t}} = \Delta_t$, we obtain
\myeqn{
\begin{array}{lcl}
\Exps{\xi_t}{\norms{\delta_t}^2} &=&  \beta_{t-1}^2\norms{\delta_{t-1}}^2 + \beta_{t-1}^2\Exps{\xi_t}{\norms{\Delta_{\xi_t} - \Delta_t}^2} + (1-\beta_{t-1})^2\norms{\hat{\delta}_t}^2 \vspace{1ex}\\
&&+ {~} 2\beta_{t-1}(1-\beta_{t-1})\iprods{\delta_{t-1}, \hat{\delta}_t}.
\end{array}
}
Taking the expectation w.r.t. $\zeta_t$, and noting that $\Exps{(\xi_t,\zeta_t)}{\cdot} = \Exps{\zeta_t}{\Exps{\xi_t}{\cdot\mid \zeta_t}}$, $\Exps{\zeta_t}{\hat{\delta}_t} = 0$, and $\Exps{\xi_t}{\norms{\Delta_{\xi_t} - \Delta_t}^2} = \Exps{\xi_t}{\norms{\Delta_{\xi_t}}^2} - \norms{\Delta_t}^2$, we obtain
\myeqn{
\begin{array}{ll}
\Exps{(\xi_t,\zeta_t)}{\norms{\delta_t}^2} &=  \beta_{t-1}^2\norms{\delta_{t-1}}^2 + \beta_{t-1}^2\Exps{\xi_t}{\norms{\Delta_{\xi_t}}^2} - \norms{\Delta_t}^2 +  (1-\beta_{t-1})^2\Exps{\zeta_t}{\norms{\hat{\delta}_t}^2},
\end{array}
}
which is exactly \eqref{eq:key_estimate10} by substituting back the definitions of $\delta_t$, $\Delta_t$, $\Delta_{\xi_t}$, and $\hat{\delta}_t$ into it.  
\Eproof
\end{proof}

\begin{remark}\label{re:biased_variance_trade_off}
From \eqref{eq:v_estimator}, we can see that $v_t$ remains a biased estimator as long as $\beta_{t-1} \in (0, 1]$. 
Its biased term is 
\myeqn{
\mathrm{Bias}[v_t\mid \Fc_t] = \norms{\Exps{(\xi_t,\zeta_t)}{v_t - G(x_t) \mid \Fc_t}} = \beta_{t-1}\norms{v_{t-1} - G(x_{t-1})} \leq \norms{v_{t-1} - G(x_{t-1})}.
}
Clearly, the biased term of the estimator $v_t$ is  smaller than the one in the SARAH estimator $v_t^{\textrm{sarah}} := v_{t-1}^{\textrm{sarah}} + G_{\xi_t}(x_t) - G_{\xi_t}(x_{t-1})$ in \cite{nguyen2017sarah}, which is $\mathrm{Bias}[v^{\mathrm{sarah}}_t \mid \Fc_t] = \norms{v_{t-1}^{\mathrm{sarah}} - G(x_{t-1})}$.
\end{remark}

The following lemma bounds the variance $\Delta_t := v_t - \nabla{f}(x_t)$ of $v_t$ defined in \eqref{eq:v_estimator}. 

\begin{lemma}\label{le:upper_bound_new}
Assume that $G_{\xi}$ is $L$-average Lipschitz continuous  and $u_t := G_{\zeta_t}(x_t)$ is an unbiased stochastic estimator of $G$.
Then, we have the following upper bound:
\myeq{eq:vt_variance_bound_new}{
\Exp{\norms{v_t - G(x_t)}^2} \leq \omega_t\Exp{\norms{v_0 - G(x_0)}^2} + L^2\sum_{i=0}^{t-1}\omega_{i,t}\Exp{\norms{x_{i+1} - x_{i}}^2} + S_t,
}
where the expectation is taking over all the randomness $\Fc_{t+1} := \sigma(x_0,\xi_0, \zeta_0, \cdots, \xi_t, \zeta_t)$, and
\myeq{eq:other_quantities}{
\left\{\begin{array}{lcl}
\omega_{t} &:=& \prod_{i=1}^{t}\beta_{i-1}^2, \vspace{1ex}\\
\omega_{i, t} &:=& \prod_{j=i+1}^{t}\beta_{j-1}^2,~~~i=0,\cdots, t, \vspace{1ex}\\
S_{t} &:=& \sum_{i=0}^{t-1}\big(\prod_{j=i+2}^{t}\beta_{j-1}^2\big)(1-\beta_i)^2\Exp{\norms{u_{i+1} - G(x_{i+1})}^2}.
\end{array}\right.
}
\end{lemma}

\begin{proof}
We first upper bound \eqref{eq:key_estimate10} by using $\sigma_t^2 := \Exps{\zeta_t}{\norms{u_t - G(x_t)}^2}$  and then taking the full expectation over $\Fc_{t+1} := \sigma(x_0, \xi_0, \zeta_0, \cdots, \xi_t, \zeta_t)$ as
\myeqn{
\begin{array}{lcl}
\Exp{\norms{v_t - G(x_t)}^2} &\leq&  \beta_{t-1}^2\Exp{\norms{v_{t-1}  - G(x_{t-1})}^2} + \beta_{t-1}^2\Exp{\norms{G_{\xi_t}(x_t) - G_{\xi_t}(x_{t-1})}^2} \vspace{1ex}\\
&& +{~} (1-\beta_{t-1})^2\sigma_t^2 \vspace{1ex}\\
&\overset{\tiny\eqref{eq:L_smooth}}{\leq}&  \beta_{t-1}^2\Exp{\norms{v_{t-1}  - G(x_{t-1})}^2} + \beta_{t-1}^2L^2\Exp{\norms{x_t - x_{t-1}}^2} + (1-\beta_{t-1})^2\sigma_t^2.
\end{array}
}
If we define $A_t^2 := \Exp{\norms{ v_t - G(x_t)}^2}$ and $B_{t-1}^2 := \Exp{\norms{x_t-x_{t-1}}^2}$, then the above inequality can be rewritten as
\myeqn{
A_t^2 \leq \beta_{t-1}^2A_{t-1}^2 +  L^2\beta_{t-1}^2B_{t-1}^2 + (1-\beta_{t-1})^2\sigma^2_t.
}
By induction, the last inequality implies 
\myeqn{
\begin{array}{lcl}
A_t^2 &\leq& \beta_{t-1}^2A_{t-1}^2 +  L^2\beta_{t-1}^2B_{t-1}^2 + (1-\beta_{t-1})^2\sigma^2_t \vspace{1ex}\\
&\leq & \beta_{t-1}^2\beta_{t-2}^2\big[\beta_{t-3}^2A_{t-3}^2 +  L^2\beta_{t-3}^2B_{t-3}^2 + (1-\beta_{t-3})^2\sigma^2_{t-2}\big] \vspace{1ex}\\
&& + {~} L^2\beta_{t-1}^2\beta_{t-2}^2B_{t-2}^2 + L^2\beta_{t-1}^2B_{t-1}^2 + \big[(1-\beta_{t-1})^2\sigma^2_t + \beta_{t-1}^2(1-\beta_{t-2})^2\sigma^2_{t-1}\big] \vspace{1ex}\\
&= & \beta_{t-1}^2\beta_{t-2}^2\beta_{t-3}^2A_{t-3}^2 +  L^2\beta_{t-1}^2\beta_{t-2}^2\beta_{t-3}^2B_{t-3}^2 +   L^2\beta_{t-1}^2\beta_{t-2}^2B_{t-2}^2 \vspace{1ex}\\
&& + {~} L^2\beta_{t-1}^2B_{t-1}^2 + \big[(1-\beta_{t-1})^2\sigma^2_t + \beta_{t-1}^2(1-\beta_{t-2})^2\sigma^2_{t-1} + \beta_{t-1}^2\beta_{t-2}^2(1-\beta_{t-3})^2\sigma^2_{t-2}\big] \vspace{1ex}\\
& \cdots & \cdots\cdots \vspace{1ex}\\
&\leq & (\beta_{t-1}^2\cdots\beta_0^2)A_0^2 + L^2(\beta_{t-1}^2\cdots\beta_0^2)B_0^2 + L^2(\beta_{t-1}^2\cdots\beta_1^2)B_1^2 + \cdots + L^2\beta_{t-1}^2B_{t-1}^2 \vspace{1ex}\\
&& + {~} \big[(1-\beta_{t-1})^2\sigma^2_t + \beta_{t-1}^2(1-\beta_{t-2})^2\sigma^2_{t-1} + \beta_{t-1}^2\beta_{t-2}^2(1-\beta_{t-3})^2\sigma^2_{t-2} + \cdots \vspace{1ex}\\
&& + {~} \beta_{t-1}^2\beta_{t-2}^2\cdots\beta_1^2(1-\beta_0)^2\sigma^2_1\big].
\end{array}
}
Here, we use a convention that $\prod_{i=t+1}^t\beta_i^2 = 1$.
As a result, the last expression can be written in the following compact form:
\myeq{eq:key_estimate1d}{
A_t^2 \leq \Big(\prod_{i=1}^{t}\beta_{i-1}^2\Big)A_0^2 + L^2\sum_{i=0}^{t-1}\Big(\prod_{j=i+1}^{t}\beta_{j-1}^2\Big)B_i^2 + \sum_{i=0}^{t-1}\Big(\prod_{j=i+2}^{t}\beta_{j-1}^2\Big)(1-\beta_i)^2\sigma^2_{i+1}.
}
Define $\omega_{t} := \prod_{i=1}^{t}\beta_{i-1}^2$, $\omega_{i, t} := \prod_{j=i+1}^{t}\beta_{j-1}^2$, and $S_{t} := \sum_{i=0}^{t-1}s_i = \sum_{i=0}^{t-1}\big(\prod_{j=i+2}^{t}\beta_{j-1}^2\big)(1-\beta_i)^2\sigma^2_{i+1}$ with $s_i := (1-\beta_i)^2\sigma^2_{i+1}\big(\prod_{j=i+2}^{t}\beta_{j-1}^2\big)$.
Then, we can rewrite \eqref{eq:key_estimate1d} as
\myeqn{
A_t^2 \leq \omega_{t}A_0^2 + L^2\sum_{i=0}^{t-1}\omega_{i,t}B_i^2 + S_{t},
}
which is exactly \eqref{eq:vt_variance_bound_new} by using the definition of $A_t$ and $B_t$ above.
\Eproof
\end{proof}

\beforesubsec
\subsection{\bf Mini-batch hybrid stochastic estimators}\label{subsec:mini_batch_biased_estimator} 
\aftersubsec
We can also consider a mini-batch hybrid stochastic estimator $\hat{v}_t$ of $G(x_t)$ as:
\myeq{eq:vhat_t}{
\hat{v}_t := \beta_{t-1}\hat{v}_{t-1} + \frac{\beta_{t-1}}{b_t}\sum_{i\in\Bc_t}\left[ G_{\xi_i}(x_t) - G_{\xi_i}(x_{t-1})\right] +  (1-\beta_{t-1})u_t,
}
where $\beta_{t-1}\in [0, 1]$ and $\Bc_t$ is a proper mini-batch of size $b_t$ (i.e., for any $\xi_t\in\Bc_t$, $\Prob{\xi_t\in\Bc_t} > 0$) and independent of $u_t$.
Note that $u_t$ can also be a mini-batch unbiased estimator of $G(x_t)$, e.g., $u_t := \frac{1}{\hat{b}_t}\sum_{j\in\hat{\Bc}_t}G_{\zeta_j}(x_t)$, where $\hat{\Bc}_t$ is a mini-batch of size $\hat{b}_t$ and independent of $\Bc_t$.

For $\hat{v}_t$ defined by \eqref{eq:vhat_t}, we have the following property, whose proof is in Appendix~\ref{apdx:le:key_pro_of_vhat_t}.

\begin{lemma}\label{le:key_pro_of_vhat_t}
Let  $\hat{v}_t$ be the mini-batch stochastic estimator of $G(x_t)$ defined by \eqref{eq:vhat_t}, where $u_t$ is also a mini-batch unbiased stochastic estimator of $G(x_t)$ with $\Exps{\hat{\Bc_t}}{u_t} = G(x_t)$ such that $\Bc_t$ is independent of $\hat{\Bc}_t$.
Then, the following estimates hold:
\myeq{eq:key_pro_of_vhat_t}{
\begin{array}{lcl}
\Exps{(\Bc_t, \hat{\Bc}_t)}{\hat{v}_t} &=& G(x_t) + \beta_{t-1}(\hat{v}_{t-1} - G(x_{t-1})), \vspace{1.5ex}\\
\Exps{(\Bc_t, \hat{\Bc}_t)}{\norms{\hat{v}_t -  G(x_t)}^2} &=&  \beta_{t-1}^2\norms{\hat{v}_{t-1}  - G(x_{t-1})}^2 - \rho\beta_{t-1}^2\norms{G(x_t) - G(x_{t-1})}^2 \vspace{1ex}\\
&& + {~} \rho\beta_{t-1}^2\Exps{\xi}{\norms{G_{\xi}(x_t) - G_{\xi}(x_{t-1})}^2} \vspace{1ex}\\
&& + {~} (1-\beta_{t-1})^2\Exps{\hat{\Bc}_t}{\norms{u_t - G(x_t)}^2},
\end{array}
}
where $\rho := \frac{n - b_t}{(n-1)b_t}$ if $G(x) := \frac{1}{n}\sum_{i=1}^nG_i(x)$ is a finite-sum, and $\rho := \frac{1}{b_t}$, otherwise $($i.e., $G(x) := \Exps{\xi}{G_{\xi}(x)}$$)$.
\end{lemma}

Similar to Lemma \ref{le:upper_bound_new}, we can bound the variance $\Exp{\|\hat{v}_t - G(x_t)\|^2}$ of the mini-batch hybrid stochastic estimator $\hat{v}_t$ from \eqref{eq:vhat_t} in the following lemma, whose proof is in Appendix~\ref{apdx:le:upper_bound_new_batch}.
For simplicity of presentation, we choose $b_t := b \in\Nbb_{+}$ and $\hat{b}_t := \hat{b} \in\Nbb_{+}$.

\begin{lemma}\label{le:upper_bound_new_batch}
Assume that $G$ is $L$-average Lipschitz continuous.
Let $u_t := \frac{1}{\hat{b}_t}\sum_{j\in\hat{\Bc}_t}G_{\zeta_j}(x_t)$ be a mini-batch unbiased estimator of $G(x_t)$ and $\hat{v}_t$ be given by \eqref{eq:vhat_t} such that $\Bc_t$ and $\hat{\Bc}_t$ are independent mini-batches of sizes $b_t := b \in\Nbb_{+}$ and $\hat{b}_t := \hat{b} \in\Nbb_{+}$, respectively for all $t\geq 0$.
Then, we have the following upper bound on the variance $\Exp{\norms{\hat{v}_t - G(x_t)}^2}$:
\myeq{eq:vt_variance_bound_new_batch}{
\Exp{\norms{\hat{v}_t - G(x_t)}^2} \leq \omega_t \Exp{\norms{\hat{v}_0 - G(x_0)}^2} + \rho L^2\sum_{i=0}^{t-1}\omega_{i,t} \Exp{\norms{x_{i+1} - x_{i}}^2} + \hat{\rho} S_t,
}
where the expectation is taking over all the randomness $\Fc_{t+1} := \sigma(x_0,\Bc_0, \hat{\Bc}_0, \cdots, \Bc_t, \hat{\Bc}_t)$, and $\omega_t$, $\omega_{i,t}$, and $S_t$ are defined in \eqref{eq:other_quantities}.
Here, $\rho := \frac{n - b}{(n-1)b}$ if  $G(x) := \frac{1}{n}\sum_{i=1}^nG_i(x)$ is a finite-sum, and $\rho := \frac{1}{b}$, otherwise $($i.e., $G(x) := \Exps{\xi}{G_{\xi}(x)}$$)$.
\end{lemma}

The theoretical results developed in Section~\ref{sec:stochastic_estimators} are self-contained.
They can be specified to develop different stochastic optimization methods, including first-order and second-order schemes, for solving \eqref{eq:ncvx_prob} and \eqref{eq:finite_sum}, and other related problems.
In the \nhan{following} sections, we only exploit these properties for $G = \nabla{f}$, the first-order derivative of $f$, to develop stochastic gradient-type methods for solving both \eqref{eq:ncvx_prob} and \eqref{eq:finite_sum}.

\beforesec
\section{Proximal Hybrid SARAH-SGD Algorithms}\label{sec:sgd_algs}
\aftersec
In this section, we utilize our hybrid stochastic estimator above with $G_{\xi}(x) := \nabla{f}_{\xi}(x)$ to develop new stochastic gradient algorithms for solving \eqref{eq:ncvx_prob} and its finite-sum setting \eqref{eq:finite_sum}.

\beforesubsec
\subsection{\bf The single-loop stochastic proximal gradient algorithm}
\aftersubsec
Our first algorithm is a single-loop stochastic proximal gradient scheme for solving \eqref{eq:ncvx_prob}.
This algorithm is described in detail as in Algorithm~\ref{alg:A1}.

\begin{algorithm}[ht!]\caption{(\textbf{Prox}imal \textbf{H}ybrid  \textbf{S}tochastic \textbf{G}radient \textbf{D}escent (ProxHSGD) algorithm)}\label{alg:A1}
\normalsize
\begin{algorithmic}[1]
   \State{\bfseries Initialization:} An arbitrarily initial point $x^0 \in\dom{F}$.
   \vspace{0.5ex}  
   \State\hspace{0ex}\label{step:o1} Input the parameters $\tilde{b} \in \Nbb_{+}$, $\beta_t\in (0, 1)$, $\gamma_t \in (0, 1]$, and $\eta_t > 0$ (will be specified later).{\!\!\!}
   \vspace{0.65ex}
   \State\hspace{0ex}\label{step:o2} Generate an unbiased estimator $v_0 := \frac{1}{\tilde{b}}\sum_{\tilde{\xi}_i\in\widetilde{\Bc}}\nabla{f}_{\tilde{\xi}_i}(x_0)$ at $x_0$ using a mini-batch $\widetilde{\Bc}$.
   \vspace{0.65ex}   
   \State\hspace{0ex}\label{step:o3} Update $\widehat{x}_1 := \tprox{\eta_0\psi}{x_0 - \eta_0v_0}$ and $x_1 := (1-\gamma_0)x_0 + \gamma_0\widehat{x}_1$.
   \vspace{0.65ex}   
   \State\hspace{0ex}\label{step:o4}{\bfseries For $t := 1,\cdots,m$ do}
   \vspace{0.5ex}   
   \State\hspace{3ex}\label{step:i1} Generate a proper sample pair $(\xi_t, \zeta_t)$ independently (single sample or mini-batch).
   \vspace{0.65ex}   
   \State\hspace{3ex}\label{step:i2} Evaluate $v_{t}  := \beta_{t-1}v_{t-1}  + \beta_{t-1}\left[\nabla{f}_{\xi_t}(x_{t}) - \nabla{f}_{\xi_t}(x_{t-1})\right] + (1-\beta_{t-1})\nabla{f}_{\zeta_t}(x_{t})$.
   \vspace{0.65ex}  
   \State\hspace{3ex}\label{step:i3} Update $\widehat{x}_{t+1} := \tprox{\eta_t\psi}{x_{t} - \eta_t v_{t}}$ and $x_{t+1} := (1-\gamma_t)x_t + \gamma_t\widehat{x}_{t+1}$.
   \vspace{0.5ex}   
   \State\hspace{0ex}{\bfseries EndFor}
   \vspace{0.5ex}   
   \State\hspace{0ex}\label{step:o5}Choose $\overline{x}_m$ from $\set{x_0, x_1, \cdots, x_m}$ (at random or deterministic, specified later). 
\end{algorithmic}
\end{algorithm}

\noindent Let us discuss the differences between Algorithm~\ref{alg:A1} and existing SGD methods:
\begin{compactitem}
\item[$\bullet$] 
Firstly, Algorithm~\ref{alg:A1} starts with a relatively large mini-batch $\widetilde{\Bc}$ to compute an initial estimate for the initial gradient $\nabla{f}(x_0)$ at $x_0$.
This is quite different from existing methods where they often use single-sample, mini-batch, or increasing mini-batch sizes for the whole algorithms (e.g., \cite{ghadimi2016mini}), and do not separate into two phases as in Algorithm~\ref{alg:A1}: 
\begin{compactitem}
\item[$\diamond$] Phase 1: Step~\ref{step:o2}: Find an appropriate initial search direction $v_0$.
\item[$\diamond$] Phase 2: Step~\ref{step:o3} to Step~\ref{step:i3}: Update the iterate sequence $\set{x_t}$.
\end{compactitem}
The idea is to find a good stochastic approximation $v_0$ for $\nabla{f}(x_0)$ as an initial search direction to guide the algorithm moving into a good direction.

\item[$\bullet$] 
Secondly, Algorithm~\ref{alg:A1} adopts the idea of ProxSARAH in \cite{Pham2019} with two steps in $\widehat{x}_t$ and $x_t$ to handle the composite setting.
This is different from existing methods as well as methods for non-composite problems where we use two step-sizes $\eta_t$ and $\gamma_t$.
While the first update on $\widehat{x}_t$ is a standard proximal gradient step, the second one on $x_t$ is an averaging step.
If $\psi = 0$, i.e., in the non-composite problems, then Steps~\ref{step:o3} and~\ref{step:i3} become{\!\!}
\myeqn{
x_{t+1} := x_t - \hat{\eta}_tv_t,~~~~\text{where}~~~\hat{\eta}_t := \gamma_t\eta_t.
}
Therefore, the product $\gamma_t\eta_t$ can be viewed as a combined step-size of Algorithm~\ref{alg:A1}.
Note that by using $\widetilde{\Grad}_{\eta_t}(x_t)$ to approximate the gradient mapping $\Grad_{\eta}$ defined by \eqref{eq:grad_map}, we can rewrite the main step of Algorithm~\ref{alg:A1} as
\myeqn{
x_{t+1} := x_t - \hat{\eta}_t \widetilde{\Grad}_{\eta_t}(x_t),~~\text{where}~~\widetilde{\Grad}_{\eta_t}(x_t) := \tfrac{1}{\eta_t}\left(x_t -  \tprox{\eta_t\psi}{x_{t} - \eta_t v_{t}}\right)~~\text{and}~~\hat{\eta}_t := \gamma_t\eta_t.
}

\item[$\bullet$] 
Thirdly, another main difference between Algorithm~\ref{alg:A1} and existing methods is at Step~\ref{step:i2}, where we use our hybrid stochastic gradient estimator $v_t$.
In addition, we will show in the sequel that by using different step-sizes, Algorithm~\ref{alg:A1} leads to different variants.
\end{compactitem}
Note that Algorithm~\ref{alg:A1} has only one loop as standard SGD or SAGA.
Hitherto, SAGA has been developed to solve the finite-sum setting \eqref{eq:finite_sum}, and there has existed no variant for solving \eqref{eq:ncvx_prob} yet.
Algorithm~\ref{alg:A1} can solve both \eqref{eq:ncvx_prob} and \eqref{eq:finite_sum}.
Moreover, it does not use an $n\times p$-table to store past gradient components as in SAGA so that it almost has the same memory requirement as in SGD.
However, at each iteration, it requires three stochastic gradient evaluations instead of one as in SGD for the single-sample case.
Therefore, its per-iteration cost can be viewed as a mini-batch SGD scheme of the batch size $3$.

\beforesubsec
\subsection{\bf One-iteration analysis}
\aftersubsec
We first prove the following two lemmas to provide key estimates for convergence analysis of Algorithm~\ref{alg:A1}.
For the sake of presentation, the proof of these lemmas is moved to Appendices~\ref{apdx:le:upper_bound_new} and \ref{apdx:le:descent_pro}, respectively.

\begin{lemma}\label{le:upper_bound_new2}
Assume that Assumptions~\ref{as:A0}, \ref{as:A1}, and \ref{as:A1b} hold.
Let $\sets{(x_t, \widehat{x}_{t})}$ be the sequence generated by Algorithm \ref{alg:A1} and  $\Grad_{\eta_t}$ be the gradient mapping of \eqref{eq:ncvx_prob} defined by \eqref{eq:grad_map}. 
Then
\myeq{eq:upper_bound_new0}{
\begin{array}{lcl}
\Exp{F(x_{t+1})} &\leq& \Exp{F(x_t)} - \frac{q_t\eta_t^2}{2}\Exp{\norms{ \Grad_{\eta_t}(x_t)}^2}  +  \frac{\theta_t}{2} \Exp{\norms{\nabla{f}(x_t) - v_t}^2}  \vspace{1.25ex}\\
&& - {~}  \frac{\kappa_t}{2}\Exp{\norms{ \widehat{x}_{t+1} - x_t }^2}  - \frac{1}{2}\Exp{\tilde{\sigma}_t^2}.
\end{array}{\!\!\!\!}
}
where $\sets{c_t}$, $\sets{r_t}$, and $\sets{q_t}$ are any given positive sequences,  
$\tilde{\sigma}_t^2 := \frac{\gamma_t}{c_t}\norms{\nabla f(x_t) - v_t  - c_t(\widehat{x}_{t+1} - x_t)}^2 \geq 0$, $S_t$ is defined in \eqref{eq:other_quantities}, and $\theta_t$ and $\kappa_t$ are given by
\myeq{eq:theta_kappa}{
\theta_t:= \frac{\gamma_t}{c_t} + (1+r_t)q_t\eta_t^2~~~~\text{and}~~~~\kappa_t := \frac{2\gamma_t}{\eta_t} - L\gamma_t^2 - \gamma_tc_t - q_t\left(1 + \frac{1}{r_t}\right).
}
\end{lemma}

\begin{lemma}\label{le:descent_pro}
Assume that Assumptions~\ref{as:A0}, \ref{as:A1}, and \ref{as:A1b} hold.
Let $\sets{(x_t, \widehat{x}_{t})}$ be the sequence generated by Algorithm \ref{alg:A1} and $\Grad_{\eta_t}$ be the gradient mapping  of \eqref{eq:ncvx_prob} defined by \eqref{eq:grad_map}. 
Given $\alpha_t > 0$, let $V$ be a Lyapunov function defined by
\myeq{eq:Lyapunov_func}{
V(x_t) := \Exp{F(x_t)} + \frac{\alpha_t}{2}\Exp{\norms{v_t - \nabla{f}(x_t)}^2}.
}
Assume that
\myeq{eq:key_cond1001}{
\alpha_t - \beta_t^2\alpha_{t+1} - \theta_t \geq 0 ~~~\text{and}~~\kappa_t - \alpha_{t+1}\beta_t^2\gamma_t^2L^2 \geq 0.
}
Then, the following estimate holds
\myeq{eq:key_bound1001a}{
V(x_{t+1}) \leq V(x_t) -  \frac{q_t\eta_t^2}{2}\Exp{\norms{ \Grad_{\eta_t}(x_t)}^2}  +  \frac{1}{2}\alpha_{t+1}(1-\beta_t)^2\sigma_{t+1}^2,
}
where $\sigma_t^2 := \Exps{\zeta_t}{\norms{\nabla{f}_{\zeta_t}(x_t) - \nabla{f}(x_t)}^2}$.
As a consequence,  for any $m\geq 0$, we also have
\myeq{eq:key_bound1001}{
\begin{array}{ll}
\displaystyle\frac{1}{2}\sum_{t=0}^mq_t\eta_t^2\Exp{\norms{ \Grad_{\eta_t}(x_t)}^2}  &\leq F(x_0)  - F^{\star} + \displaystyle\frac{\alpha_0}{2}\Exp{\norms{v_0 - \nabla{f}(x_0)}^2} \vspace{1ex}\\
& + {~} \displaystyle \frac{1}{2}\sum_{t=0}^m\alpha_{t+1}(1-\beta_t)^2\sigma_{t+1}^2.
\end{array}
}
\end{lemma}
Note that if $\beta_t = 1$ for all $t\geq 0$, then our hybrid stochastic estimator $v_t$ reduces to the SARAH estimator \cite{nguyen2017sarah}.
In this case, the estimate~\eqref{eq:key_bound1001a} becomes $V(x_{t+1}) \leq V(x_t) -  \frac{q_t\eta_t^2}{2}\Exp{\norms{ \Grad_{\eta_t}(x_t)}^2}$, which shows a monotonic non-increase of $\set{V(x_t)}$.
This estimate can be used to analyze the convergence of the double-loop SARAH-based algorithms in \cite{Nguyen2019_SARAH,Pham2019}.

\beforesubsec
\subsection{\bf Convergence analysis of Algorithm~\ref{alg:A1}}
\aftersubsec
We consider two variants of Algorithm \ref{alg:A1}: constant step-sizes and adaptive step-sizes.

\beforesubsubsec
\subsubsection{\mytxtbi{The constant step-size case}}
\aftersubsubsec
The following theorem shows the convergence of Algorithm~\ref{alg:A1}  with constant step-sizes and its oracle complexity bounds.

\revise{
\begin{theorem}\label{th:constant_stepsize_convergence}
Assume that Assumptions~\ref{as:A0}, \ref{as:A1}, and \ref{as:A1b} hold.
Let $\set{x_t}_{t=0}^m$ be generated by Algorithm~\ref{alg:A1} to solve \eqref{eq:ncvx_prob} using the following constant weight $\beta_t$ and step-sizes $\gamma_t$ and $\eta_t$:
\myeq{eq:constant_step_sizes_A}{
\left\{\begin{array}{lclcl}
\beta_t &=& \beta &:=& 1 - \frac{1}{\sqrt{\tilde{b}(m+1)}}, \vspace{1ex}\\
\gamma_t &=& \gamma &:=&  \frac{3}{\sqrt{13}[\tilde{b}(m+1)]^{1/4}}, \vspace{1ex}\\
\eta_t &=& \eta &:=& \frac{2}{(3+\gamma)L}.
\end{array}\right.
}
Then the following statements hold:
\begin{compactitem}
\vspace{0.5ex}
\item[$\mathrm{(a)}$] The parameters $\beta$, $\gamma$, and $\eta$ satisfy $\beta \in (0, 1)$, $\gamma \in (0, 1)$, and $\frac{1}{2L} \leq \eta \leq \frac{2}{3L}$.

\vspace{1ex}
\item[$\mathrm{(b)}$] Let $\overline{x}_m \sim \Uni{\set{x_t}_{t=0}^m}$ be the output of Algorithm~\ref{alg:A1}.
Then, we have
\myeq{eq:bounds_of_stepsizes}{
\Exp{\norms{ \Grad_{\eta}(\bar{x}_m)}^2} \leq \frac{16\sqrt{13}L\tilde{b}^{1/4}}{3(m+1)^{3/4}}\left[F(x_0)  - F^{\star}\right] + \frac{208\sigma^2}{9\sqrt{\tilde{b}(m+1)}}.
}
\item[$\mathrm{(c)}$]
If we choose $\tilde{b} := c_1^2(m+1)^{1/3}$ for some $c_1 \geq \frac{1}{(m+1)^{2/3}}$, then \eqref{eq:bounds_of_stepsizes} reduces to
\myeq{eq:bounds_of_stepsizes2}{
\Exp{\norms{ \Grad_{\eta}(\bar{x}_m)}^2} \leq \frac{\Delta_0}{(m+1)^{2/3}},
}
where $\Delta_0 :=  \frac{16\sqrt{13c_1}L}{3}\left[ F(x_0)  - F^{\star}\right] + \frac{208\sigma^2}{9c_1}$.
Therefore, for any tolerance $\varepsilon > 0$, the number of iterations $m$ to obtain $\overline{x}_m$ such that $\Exp{\norms{\Grad_{\eta}(\overline{x}_m)}^2} \leq \varepsilon^2$ is at most
\myeqn{
m := \left\lfloor \frac{\Delta_0^{3/2}}{\varepsilon^3}\right\rfloor.  
}
This is also the total number of proximal operations $\prox_{\eta\psi}$.
In addition, the total number $\Tc_m$ of stochastic gradient evaluations $\nabla{f_{\xi}}(x_t)$ is at most
\myeqn{
\Tc_m := \left\lfloor \frac{c_1^2\Delta_0^{1/2}}{\varepsilon} + \frac{3\Delta_0^{3/2}}{\varepsilon^3}\right\rfloor. 
}
\end{compactitem}
\end{theorem}
}

\beforepara
\paragraph{\mytxtbi{Oracle complexity comparison:}}
\revise{
Before proving Theorem~\ref{th:constant_stepsize_convergence}, let us discuss the oracle complexity of Algorithm~\ref{alg:A1} derived from Theorem~\ref{th:constant_stepsize_convergence} and compare it with existing results.
\begin{compactitem}
\item[$\bullet$] The bound  \eqref{eq:bounds_of_stepsizes2} shows that the convergence rate of Algorithm~\ref{alg:A1} is $\BigO{\frac{1}{m^{2/3}}}$, which is better than $\BigO{\frac{1}{m^{1/2}}}$ in standard SGD methods \cite{ghadimi2013stochastic}, but our $L$-average smoothness assumption is stronger than the $L$-smoothness of the expected value function $f$ in \cite{ghadimi2013stochastic}.

\item[$\bullet$] In Statement (c), although, we require the constant $c_1$ to satisfy $c_1 \geq \frac{1}{(m+1)^{2/3}}$, but it is independent of $m$. Since $m\geq 0$, we can have $c_1\geq 1$.

\item[$\bullet$]  If $\sigma = 0$, i.e., no stochasticity in our model \eqref{eq:ncvx_prob}, then from \eqref{eq:bounds_of_stepsizes}, we have $\Exp{\norms{\Grad_{\eta}(\overline{x}_m)}^2} \leq \frac{16\sqrt{13}L\tilde{b}^{1/4}}{(m+1)^{3/4}} \left[ F(x_0)  - F^{\star}\right]$.
Moreover, \eqref{eq:th1_proof1} still holds for any $\tilde{b} \geq \frac{1}{m+1}$, which is not necessary integer.
In this  case, we choose the lower bound $\tilde{b} := \frac{1}{m+1}$ to obtain the well-known bound in the deterministic case (up to a constant factor):
\myeqn{
\Exp{\norms{\Grad_{\eta}(\overline{x}_m)}^2} \leq \frac{16\sqrt{13}L}{(m+1)} \left[ F(x_0)  - F^{\star}\right].
}
Here, the expectation is taking over the remaining randomness (e.g., the random choice of $\bar{x}_m$).
This bound leads to the oracle complexity of $\BigO{\varepsilon^{-2}}$ as often seen in gradient-based methods for non-convex optimization.

\item[$\bullet$]  
If $\sigma > 0$, then one can minimize the right-hand side of  \eqref{eq:bounds_of_stepsizes} over $\tilde{b}$ to get 
\myeqn{
\tilde{b} := \frac{13^{2/3}\sigma^{8/3}(m+1)^{1/3}}{3^{4/3}L^{4/3}\Delta_F^{4/3}}, ~~\text{where}~~\Delta_F := F(x_0)  - F^{\star} > 0.
}
With this choice of $\tilde{b}$, the number of iterations $m$ and the total number $\Tc_m$ of stochastic gradient evaluations in Theorem~\ref{th:constant_stepsize_convergence} become
\myeqn{
m =  \BigO{\frac{  \sigma L \Delta_F}{\varepsilon^3}}~~~\text{and}~~~\Tc_m = \BigO{ \frac{\sigma^3}{L \Delta_F\varepsilon} + \frac{ \sigma L\Delta_F }{\varepsilon^3} }.
}
This bound shows the linear dependence on $\sigma$, $L$, $\Delta_F$ of $m$.
If $\Delta_F$ or $L$ is large compared to $\sigma$ or $\sigma$ is too small, we can rescale $\tilde{b}$ to trade-off $\Delta_F$, $L$, and $\sigma$ in \eqref{eq:bounds_of_stepsizes}, leading to different bounds of $m$ and $\Tc_m$ (up to a constant).
\item[$\bullet$] 
If $0 < \sigma \leq \BigO{\varepsilon^{-1}}$ and $\sigma$ dominates $L$ and $\Delta_F$, then the oracle complexity of Algorithm~\ref{alg:A1} is $\BigO{\sigma^3\varepsilon^{-1} + \sigma\varepsilon^{-3}}$, which is the same as the best-known stochastic oracle complexity $\BigO{\sigma^2\varepsilon^{-2}+\sigma\varepsilon^{-3}}$ of SPIDER \cite{fang2018spider}, SpiderBoost \cite{wang2018spiderboost}, or  ProxSARAH \cite{Pham2019}. 
\end{compactitem}
}

\begin{remark}\label{re:dependence_of_sigma}
\revise{We also make the following remarks:
\begin{compactitem}
\item[$\bullet$]
The weight $\beta$ and the step-sizes $\eta$ and $\gamma$ in Theorem~\ref{th:constant_stepsize_convergence} is not unique.
As shown in the proof of Theorem~\ref{th:constant_stepsize_convergence}, the configuration \eqref{eq:constant_step_sizes_A} is obtained by choosing $c_t := L$, $r_t := 1$, and $q_t := \frac{L\gamma_t}{2}$ in Lemma~\ref{le:upper_bound_new}. 
Under different choice of these parameters, we can obtain different configuration than \eqref{eq:constant_step_sizes_A}.

\item[$\bullet$]
Note that if we choose $\eta$ such that $0 < \eta < \frac{2}{(3+\gamma)L}$, then our results in Theorem~\ref{th:constant_stepsize_convergence} still hold, but the right-hand side of \eqref{eq:bounds_of_stepsizes} will be scaled up by a factor proportional to $\frac{1}{\eta^2}$.
\end{compactitem}
}
\end{remark}

\begin{proof}[\mytxtbi{Proof of Theorem \ref{th:constant_stepsize_convergence}}]
First, let us choose $c_t := L$, $r_t := 1$, and $q_t := \frac{L\gamma_t}{2}$ in Lemma~\ref{le:upper_bound_new}.
We also fix $\beta_t := \beta \in (0, 1)$, $\eta_t := \eta > 0$, and $\gamma_t := \gamma \in (0, 1)$.
From \eqref{eq:theta_kappa}, we have
\myeqn{
\theta_t = \theta = \left(\tfrac{1+L^2\eta^2}{L}\right)\gamma ~~~~\text{and}~~~\kappa_t = \kappa = \gamma\left(\tfrac{2}{\eta} - L\gamma - 2L\right).
}
Next, since $v_0$ is computed by Step~\ref{step:o2} with a mini-batch size $\tilde{b}$, by  \cite[Lemma 2]{Pham2019}, we have 
\myeq{eq:th1_proof1}{
\Exp{\norms{v_0 - \nabla{f}(x_0)}^2} \leq \frac{1}{\tilde{b}}\Exps{\xi}{\norms{\nabla{f}_{\xi}(x_0) - \nabla{f}(x_0)}^2} \leq \frac{\sigma^2}{\tilde{b}}.
}
Let us also fix $\alpha_t := \alpha > 0$ for $t\geq 0$ in Lemma~\ref{le:descent_pro}. 
Then, by utilizing \eqref{eq:th1_proof1} and $q_t := \frac{L\gamma_t}{2}$, we can derive from \eqref{eq:key_bound1001} that
\myeq{eq:th41_est2}{
{\!\!\!\!}\frac{1}{m+1}\sum_{t=0}^m\Exp{\norms{ \Grad_{\eta}(x_t)}^2} \leq \frac{4}{L\eta^2\gamma(m+1)}\left[F(x_0)  - F^{\star}\right] + \frac{2\alpha\sigma^2}{L\gamma\eta^2}\left[\frac{1}{\tilde{b}(m+1)} + (1-\beta)^2\right].{\!\!\!\!}
}
By minimizing $\frac{1}{\tilde{b}(m+1)} + (1-\beta)^2$ over $\beta \in [0, 1]$, we obtain $\beta := 1- \frac{1}{[\tilde{b}(m+1)]^{1/2}}$ as in \eqref{eq:constant_step_sizes_A}.
Moreover, the two conditions in \eqref{eq:key_cond1001} can be simplified as
\myeq{eq:key_cond1001b}{
(1 + L^2\eta^2)\gamma  \leq (1-\beta^2)\alpha L ~~~~~\text{and}~~~~~\frac{2}{\eta} - L\gamma - 2L \geq \alpha\gamma\beta^2 L^2.
}
\noindent{(a)}~Let us update $\eta := \frac{2}{L(3 + \gamma)}$ as \eqref{eq:constant_step_sizes_A}.
Since $\gamma \in [0,1]$, we have $\frac{1}{2L} \leq \eta \leq \frac{2}{3L}$.
Moreover, by the update of $\beta := 1- \frac{1}{[\tilde{b}(m+1)]^{1/2}}$, we also have $\beta \in (0, 1)$ since $m\geq 0$ and $\tilde{b}\geq 1$.

Now, since $\eta\leq \frac{2}{3L}$, we have $1 + L^2\eta^2 \leq \frac{13}{9}$.
In addition, it is obvious that $1-\beta^2 \geq 1-\beta = \frac{1}{[\tilde{b}(m+1)]^{1/2}}$.
Therefore, the first condition of \eqref{eq:key_cond1001b}  holds if we choose 
\myeqn{
0 < \gamma \leq \bar{\gamma} := \frac{9L \alpha }{13[\tilde{b}(m+1)]^{1/2}}.
}
Alternatively, since $\frac{2}{\eta} - L\gamma - 2L = L$ and $\beta \in [0, 1]$, the second condition of  \eqref{eq:key_cond1001b}  holds if we choose $0 < \gamma \leq \bar{\gamma} := \frac{1}{L\alpha}$.
Combining both conditions on $\gamma$, we obtain $\bar{\gamma} := \frac{1}{L\alpha} = \frac{9L \alpha }{13[\tilde{b}(m+1)]^{1/2}}$.
Hence, we obtain $\alpha := \frac{\sqrt{13}[\tilde{b}(m+1)]^{1/4}}{3L}$, which implies that $\bar{\gamma} = \frac{3}{\sqrt{13}[\tilde{b}(m+1)]^{1/4}}$.
Since $\tilde{b} \geq 1$ and $m\geq 0$, we have $ 0 < \bar{\gamma} < 1$.
Therefore, we can update 
\myeqn{
\gamma := \frac{3}{\sqrt{13}[\tilde{b}(m+1)]^{1/4}} \in (0, 1)
}
as shown in \eqref{eq:constant_step_sizes_A}.

\vspace{1ex}
\noindent{(b)}~Next, using the update of $\gamma$, the choice of $\alpha$, and the fact that $\frac{1}{2L} \leq \eta \leq \frac{2}{3L}$, we can further simplify \eqref{eq:th41_est2} as
\myeqn{
\frac{1}{m+1}\sum_{t=0}^m\Exp{\norms{ \Grad_{\eta}(x_t)}^2} \leq \frac{16\sqrt{13}L\tilde{b}^{1/4}}{3(m+1)^{3/4}}\left[F(x_0)  - F^{\star}\right] + \frac{208\sigma^2}{9[\tilde{b}(m+1)]^{1/2}}.{\!\!\!\!}
}
Since $\overline{x}_m \sim \Uni{\set{x_t}_{t=0}^m}$, we have $\Exp{\norms{\Grad_{\eta}(\overline{x}_m)}^2} = \frac{1}{m+1}\sum_{t=0}^m\Exp{\norms{ \Grad_{\eta}(x_t)}^2}$.
Combining this relation and the last inequality, we obtain \eqref{eq:bounds_of_stepsizes}.

\vspace{1ex}
\noindent{(c)}~
If we choose $\tilde{b} := c_1^2(m+1)^{1/3}$ for some $c_1 > 0$, then the bound \eqref{eq:bounds_of_stepsizes} reduces to \eqref{eq:bounds_of_stepsizes2}, where $\Delta_0 := \frac{16\sqrt{13c_1}L}{3}\left[ F(x_0)  - F^{\star}\right] + \frac{208\sigma^2}{9c_1}$. 
Moreover, since $\beta = 1 - \frac{1}{[\tilde{b}(m+1)]^{1/2}}$, to guarantee $\beta \in (0,1]$, we need to choose $c_1 \geq \frac{1}{(m+1)^{2/3}}$.

For any tolerance $\varepsilon > 0$, the number of iterations $m$ to achieve $\Exp{\norms{\Grad_{\eta}(\overline{x}_m)}^2} \leq \varepsilon^2$ can be estimated from \eqref{eq:bounds_of_stepsizes2} by letting:
\myeqn{
\frac{1}{(m+1)^{2/3}}\left[ \frac{16\sqrt{13c_1}L}{3}\left[ F(x_0)  - F^{\star}\right] + \frac{208\sigma^{2}}{9c_1}\right] = \frac{\Delta_0}{(m+1)^{2/3}} \leq \varepsilon^2.
}
This implies that $m+1 \geq \frac{\Delta_0^{3/2}}{\varepsilon^3}$.
Therefore, we can choose $m :=  \left\lfloor \frac{\Delta_0^{3/2}}{\varepsilon^3}\right\rfloor$.
This is also the number of proximal operations $\prox_{\eta\psi}$.
The total number $\Tc_m$ of stochastic gradient evaluations \nhan{$\nabla{f}_{\xi}(x_t)$}  is estimated as
\myeqn{
\Tc_m := \tilde{b} + 3(m+1) = c_1^2(m+1)^{1/3} +   \frac{3\Delta_0^{3/2}}{\varepsilon^3} = \frac{c_1^2\Delta_0^{1/2}}{\varepsilon} + \frac{3\Delta_0^{3/2}}{\varepsilon^3}.
}
Hence, we can choose $\Tc_m := \left\lfloor \frac{c_1^2\Delta_0^{1/2}}{\varepsilon} + \frac{3\Delta_0^{3/2}}{\varepsilon^3} \right\rfloor$ as its upper bound, which proves  (c).
\Eproof
\end{proof}

\beforesubsubsec
\subsubsection{\mytxtbi{The adaptive step-size case}}
\aftersubsubsec
\revise{Theorem~\ref{th:constant_stepsize_convergence} states the convergence and complexity estimate of Algorithm~\ref{alg:A1} with constant step-sizes.
However, when the number of iterations $m$ is large, this constant step-size $\gamma$ is small. 
We instead develop an adaptive rule to update the step-size $\gamma_t$ as follows:
\begin{compactitem}
\item[$\bullet$] Let us first fix $\beta := 1 - \frac{1}{[\tilde{b}(m+1)]^{1/2}} \in (0, 1)$ as in Theorem~\ref{th:constant_stepsize_convergence}.
\item[$\bullet$] Next, we also fix $\eta_t := \eta \in (0, \frac{1}{L})$ and define $\delta := \frac{2}{\eta} - 2L > 0$.
\item[$\bullet$] Then, we can update $\gamma_t$ adaptively as 
\myeq{eq:update_of_eta_t}{
\gamma_m := \frac{\delta}{L}~~~~\text{and}~~~~\gamma_t := \frac{\delta}{L + L(1+L^2\eta^2)\big[\beta^2\gamma_{t+1} + \beta^4\gamma_{t+2} + \cdots + \beta^{2(m-t)}\gamma_m\big]},
}
for $t=0,\cdots, m-1$.
\end{compactitem}
Applying Lemma~\ref{le:adaptive_step_size}, it is obvious to show that  $0 < \gamma_0 < \gamma_1 < \cdots < \gamma_m$.
Interestingly, our step-size is updated in an increasing manner instead of diminishing as in existing SGD-type methods.
Here, $\gamma_t$ becomes larger as $t$ increases.
Moreover, given $m$, we can pre-compute the sequence of these step-sizes $\set{\gamma_t}_{t=0}^m$ in advance within $\BigO{m}$ basic operations.
Therefore, it does not significantly incur the computational cost of our method.
}

The following theorem states the convergence of Algorithm~\ref{alg:A1} under the adaptive update rule \eqref{eq:update_of_eta_t}, whose proof can be found in Appendix~\ref{apdx:th:singe_loop_adapt_step}.

\revise{
\begin{theorem}\label{th:singe_loop_adapt_step}
Assume that Assumptions~\ref{as:A0}, \ref{as:A1}, and \ref{as:A1b} hold.
Let $\sets{x_t}_{t=0}^m$ be the sequence generated by Algorithm~\ref{alg:A1} to solve  \eqref{eq:ncvx_prob} using the parameters $\beta$, $\eta$, and step-size $\gamma_t$ defined by \eqref{eq:update_of_eta_t}.
Then, the following statements hold:
\begin{compactitem}
\vspace{0.5ex}
\item[$(\mathrm{a})$] 
If $\Sigma_m := \sum_{t=0}^m\gamma_t$ and $\overline{x}_m\sim \Unip{\pb}{\sets{x_t}_{t=0}^m}$ with $\pb_t := \Prob{\overline{x}_m = x_t} = \frac{\gamma_t}{\Sigma_m}$, then we have
\myeq{eq:adaptive_key_est}{
{\!\!\!\!}
\Exp{\norms{\Grad_{\eta}(\overline{x}_m)}^2} \leq \frac{8\sqrt{2}\big(L + \sqrt{\delta L}[\tilde{b}(m+1)]^{1/4}\big)}{L\eta^2\delta(m+1)}\left[F(x_0) -  F^{\star}\right]  +  \frac{8\sigma^2}{L^2\eta^2[\tilde{b}(m+1)]^{1/2}}.
{\!\!\!\!}
}

\vspace{0.65ex}
\item[$(\mathrm{b})$] 
Let us choose the initial mini-batch size $\tilde{b} :=  c_1^2 (m + 1)^{1/3}$ for $c_1 \geq \frac{1}{(m+1)^{2/3}}$.
Then, for any $\varepsilon > 0$, the number of iterations $m$ to guarantee $\Exp{\norms{\Grad_{\eta}(\overline{x}_m)}^2} \leq \varepsilon^2$ does not exceed
\myeqn{
m := \left\lfloor  \frac{\Delta_0^{3/2}}{\varepsilon^3} \right\rfloor, ~\text{where}~\Delta_0 := \frac{8}{L^2\eta^2}\left[\frac{\sqrt{2}L( L + \sqrt{c_1L\delta})}{\delta}\big[F(x_0) - F^{\star}\big] + \frac{\sigma^2}{c_1}\right].
}
This is also the total number of proximal operations $\prox_{\eta\psi}$.
Consequently, the number $\Tc_m$ of stochastic gradient evaluations is at most $\Tc_m := \left\lfloor \frac{c_1^2\Delta_0^{1/2}}{\varepsilon} + \frac{3\Delta_0^{3/2}}{\varepsilon^3}\right\rfloor$. 
\end{compactitem}
\end{theorem}
}

While the proof of Theorem~\ref{th:constant_stepsize_convergence} relies on the Lyapunov function $V$ defined by \eqref{eq:Lyapunov_func} that has an asymptotically  monotone property, the proof of Theorem~\ref{th:singe_loop_adapt_step} is completely different by adopting the techniques in \cite{Pham2019} and does not use any Lyapunov function.
\revise{The oracle complexity of Theorem~\ref{th:singe_loop_adapt_step} remains the same as in Theorem~\ref{th:constant_stepsize_convergence}.
When $\sigma > 0$, we can also adjust $\tilde{b}$ in \eqref{eq:adaptive_key_est} to obtain the final bounds for $m$ and $\Tc_m$ that depend on the variance $\sigma$.}

\begin{remark}[\textbf{Without initial batch}]\label{re:compare_sgd}
\revise{If we choose the initial batch size $\tilde{b} := 1$ (i.e., single sample) to compute $v_0$ at Step~\ref{step:o2} of Algorithm~\ref{alg:A1}, then  \eqref{eq:adaptive_key_est} becomes
\myeqn{ 
\Exp{\norms{\Grad_{\eta}(\overline{x}_m)}^2} \leq \frac{8\sqrt{2}\left[F(x_0) -  F^{\star}\right] }{\eta^2\delta(m+1)} +  \frac{8\sqrt{2L\delta}\left[F(x_0) -  F^{\star}\right]}{L\eta^2\delta(m+1)^{3/4}}  +  \frac{8\sigma^2}{L^2\eta^2(m+1)^{1/2}}.
}
Hence, the oracle complexity of Algorithm~\ref{alg:A1} reduces to $\BigO{\max\set{\frac{(L\Delta_F)^{4/3}}{\varepsilon^{8/3}}, \frac{\sigma^2}{\varepsilon^4}}}$, where $\Delta_F := F(x_0) -  F^{\star}$.
This complexity is similar to proximal SGD methods, see, e.g., \cite{ghadimi2013stochastic} if the second term dominates the first one.
Therefore, the choice of the initial mini-batch $\widetilde{\Bc}_t$ for $v_0$ is crucial in Algorithm~\ref{alg:A1} to achieve better complexity bounds than SGD.}
\end{remark}

\begin{remark}[\textbf{The effect of $m$ on $\gamma_t$}]\label{re:adaptive_stepsize}
Due to the update \eqref{eq:update_of_eta_t}, we have $\gamma_m > \gamma_{m-1} > \cdots > \gamma_0 > 0$.
Clearly, if $m$ is large, then $\set{\gamma_t}$ is getting smaller and smaller as $t$ is decreasing, which leads to a slow convergence.
This suggests that we should restart Algorithm~\ref{alg:A1} after a relatively small number of iterations $m$ to avoid small step-sizes $\set{\gamma_t}$.
This algorithmic variant becomes more efficient if we combine it with a \nhan{double-loop} as described in Algorithm~\ref{alg:A2}.
\end{remark}

\beforesubsec
\subsection{\bf Restarting proximal hybrid stochastic gradient descent algorithm}\label{sec:sgd_algs2}
\aftersubsec
\revise{
\noindent\mytxtbi{Motivation:} We observe from Theorems~\ref{th:constant_stepsize_convergence} and~\ref{th:singe_loop_adapt_step} that:
\begin{compactitem}
\item[$\bullet$] 
If $m$ is large, then from \eqref{eq:constant_step_sizes_A}, we can see that the step-size $\gamma$ is small and $\beta$ is very close to $1$.
While a small step-size $\gamma$ leads to slow progress in Algorithm~\ref{alg:A1}, $\beta\approx 1$ shows that  the unbiased term does not significantly compensate the biaseness of the estimator $v_t$.

\item[$\bullet$]
Similarly, as can be seen from Remark~\ref{re:adaptive_stepsize} that if $m$ is large, then the step-size $\gamma_t$ in Theorem~\ref{th:singe_loop_adapt_step} is also small as $t$ decreases, which also makes Algorithm~\ref{alg:A1} slow.
\end{compactitem}
To circumvent this issue, we can restart Algorithm~\ref{alg:A1} after running it for a certain number of iterations $m$ by adding an outer-loop.
In this case, we obtain a double-loop variant as in  SVRG or SARAH variants.
However, unlike SVRG and SARAH-based methods where their double-loop is mandatory to guarantee convergence, we use it as a restarting loop.
Without the outer loop, Algorithm~\ref{alg:A1} still has convergence guarantee as shown in Theorems~\ref{th:constant_stepsize_convergence} and \ref{th:singe_loop_adapt_step}.
According to a very recent work \cite{carmon2017lower}, our algorithm achieves the optimal oracle complexity (up to a constant) under Assumptions~\ref{as:A0}, \ref{as:A1}, and \ref{as:A1b}. 
}

The complete restarting variant of Algorithm~\ref{alg:A1} is described in Algorithm~\ref{alg:A2}.

\begin{algorithm}[ht!]\caption{(Restarting Proximal Hybrid SGD algorithm (ProxHSGD-RS))}\label{alg:A2}
\normalsize
\begin{algorithmic}[1]
   \State{\bfseries Initialization:} An initial point $\overline{x}^{(0)}$  and parameters $\tilde{b}$, $m$, $\beta_t$, and $\eta_t$ (will be specified).{\!\!\!}
   \vspace{0.5ex}
   \State{\bfseries Restarting stage:}~{\bfseries For $s := 1, 2, \cdots, S$ do}
   \vspace{0.5ex}   
   \State\hspace{3ex}\label{step:a2o2} Run Algorithm~\ref{alg:A1} with an initial point $x_0^{(s)} := \overline{x}^{(s-1)}$.
   \vspace{0.5ex}   
   \State\hspace{3ex}\label{step:a2o5} Set $\overline{x}^{(s)} := x_{m+1}^{(s)}$ as the last iterate of Algorithm~\ref{alg:A1}. 
   \vspace{0.5ex}   
   \State{\bfseries EndFor}
\end{algorithmic}
\end{algorithm}

To analyze Algorithm~\ref{alg:A2}, we use $x^{(s)}_t$ to represent the iterate of Algorithm~\ref{alg:A1} at  the $t$-th inner iteration within each stage $s$.
As we can see,  Algorithm~\ref{alg:A2} calls Algorithm~\ref{alg:A1} as a subroutine \nhan{for every} iteration, called the $s$-th stage and \nhan{retrieves} the output $\overline{x}^{(s)} := x_{m+1}^{(s)}$ as the last iterate of Algorithm~\ref{alg:A1} instead of taking it randomly from $\sets{x_t^{(s)}}_{t=0}^m$.
Here, we assume that we fix the step-size \revise{$\eta_t = \eta \in (0, \frac{1}{L})$}, fix the mini-batch $\tilde{b}_s = \tilde{b} \in\Nbb_{+}$, and choose $\beta := 1 - \frac{1}{[\tilde{b}(m+1)]^{1/2}} \in (0, 1)$ for simplicity of our analysis.

Now, we can derive the convergence of Algorithm~\ref{alg:A2} in the following theorem whose proof is deferred to Appendix~\ref{apdx:th:double_loop_convergence}.

\revise{
\begin{theorem}\label{th:double_loop_convergence}
Assume that we choose $\tilde{b}_s := \tilde{b} \in\Nbb_{+}$, $\beta := 1 - \frac{1}{[\tilde{b}(m+1)]^{1/2}} \in (0, 1)$, and $\eta \in (0, \frac{1}{L})$, and update the  step-size $\gamma_t$ for Algorithm~\ref{alg:A2} as in \eqref{eq:update_of_eta_t}.
Let $\sets{x^{(s)}_t}_{t=0\to m}^{s=1\to S}$ be generated by Algorithm~\ref{alg:A2} to solve \eqref{eq:ncvx_prob} and $\overline{x}_T \sim \Unip{\pb}{\sets{x^{(s)}_t}_{t=0\to m}^{s=1\to S}}$ with $\Prob{\overline{x}_T = x_t^{(s)}} = \frac{\gamma_t}{S\Sigma_m}$ be the output of Algorithm~\ref{alg:A2}.
Then, the following statement holds:
\begin{compactitem}
\item[$\mathrm{(a)}$] 
The following estimate holds:
\myeq{eq:double_loop_est}{
\Exp{\norms{\Grad_{\eta}(\overline{x}_T)}^2} \leq  \frac{8\sqrt{2}\tilde{b}^{1/4}\big(L + \sqrt{L\delta}\big)}{L\delta\eta^2S (m+1)^{3/4}} \big[F(\overline{x}^{(0)}) - F^{\star}\big] + \frac{8\sigma^2}{L^2\eta^2[\tilde{b}(m+1)]^{1/2}}.
}
\item[$\mathrm{(b)}$] 
For some constant $c_1 \geq \frac{1}{(m+1)}$ and  for any tolerance $\varepsilon > 0$, let us choose
\myeqn{
\tilde{b} :=  \frac{16c_1}{L^2\eta^2}\cdot\frac{\max\set{1,\sigma^2}}{\varepsilon^2}~~~\text{and}~~~m + 1 :=  \frac{16}{c_1L^2\eta^2}\cdot\frac{\max\set{1,\sigma^2}}{\varepsilon^2}.
}
Then,  to guarantee $\Exp{\norms{\Grad_{\eta}(\overline{x}_T)}^2} \leq \varepsilon^2$, we need at most $S$ outer iterations as
\myeq{eq:S_iterations}{
S := \left\lfloor \frac{4\sqrt{2}c_1\big(L + \sqrt{L\delta}\big)}{\delta\eta\varepsilon}\big[F(\overline{x}^{(0)}) - F^{\star}\big] \right\rfloor.
}
Consequently,  the total number  $\Tc_{\nabla{f}}$ of stochastic gradient evaluations and the total number $\Tc_{\prox}$ of proximal operations $\prox_{\eta\psi}$, respectively do not exceed 
\myeq{eq:Toc_double_loop}{
\hspace{-1ex}\begin{array}{llclcl}
&\Tc_{\nabla{f}} &:= &   \dfrac{64\sqrt{2}(c_1^2 + 3)(L + \sqrt{L\delta})\max\set{1,\sigma}}{L^2\eta^3\delta \varepsilon^3}\big[F(\overline{x}^{(0)}) - F^{\star}\big]
& = &  \BigO{ \dfrac{\max\set{\sigma, 1}}{ \varepsilon^{3}}}, \vspace{1ex}\\
&\Tc_{\prox} &:=& \dfrac{64\sqrt{2}(L + \sqrt{L\delta})\max\set{1,\sigma}}{L^2\eta^3\delta \varepsilon^3} \big[F(\overline{x}^{(0)}) - F^{\star}\big]
& = &  \BigO{ \dfrac{\max\set{\sigma, 1}}{\varepsilon^{3}}}.
\end{array}
\hspace{-1ex}
}
\end{compactitem}
\end{theorem}
}

\revise{If $\sigma = 0$, i.e., no stochasticity involved in \eqref{eq:ncvx_prob} and $\tilde{b} := c_1^2(m+1)$, then the bound \eqref{eq:double_loop_est} reduces to $\Exp{\norms{\Grad_{\eta}(\overline{x}_T)}^2} \leq \frac{8\sqrt{2c_1}\big(L + \sqrt{L\delta}\big)}{L\delta\eta^2S (m+1)^{1/2}} \big[F(\overline{x}^{(0)}) - F^{\star}\big]$.
However, since $c_1 \geq \frac{1}{m+1}$, we can choose its lower bound as $c_1 := \frac{1}{m+1}$.
In this case, the last inequality becomes
\myeqn{
\Exp{\norms{\Grad_{\eta}(\overline{x}_T)}^2} \leq \frac{8\sqrt{2}\big(L + \sqrt{L\delta}\big)}{L\delta\eta^2S (m+1)} \big[F(\overline{x}^{(0)}) - F^{\star}\big] =  \BigO{\frac{L[F(\overline{x}^{(0)}) - F^{\star}]}{S(m+1)}}.
}
This bound is the same as in gradient-based methods.
If $\sigma > 0$, then the total number of stochastic gradient evaluations is at most $\BigO{\frac{\sigma^2}{\varepsilon^2} + \frac{\sigma}{\varepsilon^3}}$, which is optimal  (up to a constant factor) according to \cite{carmon2017lower} under Assumptions~\ref{as:A0}, \ref{as:A1}, and \ref{as:A1b}.
Practically, if $\beta$ is very close to $1$, one can remove the unbiased SGD term to save one stochastic gradient evaluation. 
In this case, our estimator reduces to SARAH but using different step-size.
Our empirical results show that when $\beta\approx 0.999$, the performance of our methods is not affected. 
}

\begin{remark}\label{re:choice_of_m}
We have not tried to optimize all the constant factors in the bounds of Theorems~\ref{th:constant_stepsize_convergence}, \ref{th:singe_loop_adapt_step}, and \ref{th:double_loop_convergence}.
Therefore, our bounds can be different from existing results up to a constant factor as we can see in the case $\sigma = 0$ (i.e., no stochasticity in \eqref{eq:ncvx_prob}).
\end{remark}

\beforesubsec
\subsection{\textbf{Linear convergence under gradient dominant condition}}\label{subsec:grad_dominance}
\aftersubsec
If the composite function $F$ satisfies the following $\tau$-gradient dominant condition \cite{wang2018spiderboost}:
\myeq{eq:grad_dominance}{
F(x) - F^{\star} \leq \frac{\tau}{2}\norms{\Grad_{\eta}(x)}^2,
}
for any $x\in\dom{F}$ and $\eta > 0$, where $\tau > 0$ (see, e.g., \cite{wang2018spiderboost}), then we can modify Algorithm~\ref{alg:A2} by setting $\overline{x}^{(s)} := \overline{x}_m^{(s)}$, where $ \overline{x}_m^{(s)}\sim\Unip{\pb}{\sets{x_t^{(s)}}_{t=0}^m}$, to obtain an $\varepsilon$-linear convergence rate.
Note that if $\psi = 0$, then the gradient dominant condition above reduces to the standard one $f(x) - f(x^{\star}) \leq \frac{\tau}{2}\norms{\nabla{f}(x)}^2$ for any $x\in\dom{f}$, which is widely used in the literature. 

\revise{
\begin{corollary}\label{re:grad_dominant}
Suppose that the assumptions of Theorem~\ref{th:double_loop_convergence} and the gradient dominant condition \eqref{eq:grad_dominance} holds.
Let $\sets{x_t^{(s)}}_{t=0}^m$ be generated by Algorithm~\ref{alg:A2} to solve \eqref{eq:ncvx_prob}, where $\overline{x}^{(s)} := \overline{x}_m^{(s)}$ with $\overline{x}_m^{(s)}\sim\Unip{\pb}{\sets{x_t^{(s)}}_{t=0}^m}$, and $m$ and $\tilde{b}$ are chosen as
\myeqn{
m + 1 := \frac{32(L + \sqrt{L\delta})\tau^{3/2}\sigma}{L^2\eta^3\delta \sqrt{\varepsilon}}~~~~\text{and}~~~~\tilde{b} := \frac{2\delta\tau^{1/2}\sigma^3}{L^2(L+\sqrt{L\delta})\eta\varepsilon^{3/2}},
}
for a given tolerance $\varepsilon > 0$.
Then, the following inequalities hold:
\myeq{eq:eps_linear_rate}{
\Exp{F(\overline{x}^{(s)}) - F^{\star}} \leq \frac{1}{2}\Exp{F(\overline{x}^{(s-1)}) - F^{\star}} + \frac{\varepsilon}{2} \leq \frac{1}{2^S}\left(\Exp{F(\overline{x}^{(0)}) - F^{\star}} -\varepsilon\right) + \varepsilon.
}
Consequently, $\set{\Exp{F(\overline{x}^{(s)}) - F^{\star}}}$ converges linearly to an $\varepsilon$-ball around zero.
\end{corollary}
}

\begin{proof}
Similar to the proof of \eqref{eq:est5d_1}, using the choice of $\overline{x}^{(s)}$ we can show that
\myeqn{
\Exp{\norms{\Grad_{\eta}(\overline{x}^{(s)})}^2}  \leq \dfrac{8\sqrt{2} \tilde{b}^{1/4}(L + \sqrt{L\delta})}{L\eta^2\delta(m+1)^{3/4}}\Exp{F(\overline{x}^{(s-1)}) - F^{\star}} + \dfrac{8\sigma^2}{L^2\eta^2\tilde{b}^{1/2}(m+1)^{1/2}}.
}
Multiplying this inequality by $\frac{\tau}{2}$ and then using \eqref{eq:grad_dominance}, we can show that
\myeqn{
\Exp{F(\overline{x}^{(s)}) - F^{\star}} \leq \dfrac{4\sqrt{2}\tau \tilde{b}^{1/4}(L + \sqrt{L\delta})}{L\eta^2\delta(m+1)^{3/4}} \Exp{F(\overline{x}^{(s-1)}) - F^{\star}} + \dfrac{4\tau\sigma^2}{L^2\eta^2\tilde{b}^{1/2}(m+1)^{1/2}}.
}
Assume that $\frac{4\tau \sigma^2}{L^2\eta^2\tilde{b}^{1/2}(m+1)^{1/2}} = \frac{\varepsilon}{2}$ and $\frac{4\sqrt{2}\tau\tilde{b}^{1/4}(L + \sqrt{L\delta})}{L\eta^2\delta(m+1)^{3/4}} = \frac{1}{2}$.
These relations lead to $m+1 := \frac{32(L + \sqrt{L\delta})\tau^{3/2}\sigma}{L^2\eta^3\delta \sqrt{\varepsilon}}$ and $\tilde{b} := \frac{2\delta\tau^{1/2}\sigma^3}{L^2(L+\sqrt{2L\delta})\eta\varepsilon^{3/2}}$.
Hence, the last inequality can be simplified as 
\myeqn{
\Exp{F(\overline{x}^{(s)}) - F^{\star}} \leq \frac{1}{2}\Exp{F(\overline{x}^{(s-1)}) - F^{\star}} + \frac{\varepsilon}{2}.
}
Denote $\Delta_s := \Exp{F(\overline{x}^{(s)}) - F^{\star}}$.
Then, the last inequality becomes $\Delta_s - \varepsilon \leq \frac{1}{2}\left(\Delta_{s-1} - \varepsilon\right)$.
Whenever $\Delta_s \geq \varepsilon$, by induction, we have $\Delta_S - \varepsilon \leq \frac{1}{2^S}\left(\Delta_0 - \varepsilon\right)$.
This implies \eqref{eq:eps_linear_rate}.
Therefore, $\set{\Exp{F(\overline{x}^{(s)}) - F^{\star}}}$ converges linearly to an $\varepsilon$-ball around zero.
\Eproof
\end{proof}

\beforesubsec
\subsection{\bf Applications to the finite-sum and non-composite settings}\label{subsec:finite_sum}
\aftersubsec
We first consider the finite-sum setting \eqref{eq:finite_sum} and then discuss the non-composite form of \eqref{eq:ncvx_prob}.

\beforesubsubsec
\subsubsection{\mytxtbi{The finite-sum case}}
\aftersubsubsec
We can apply both Algorithm~\ref{alg:A1} and Algorithm~\ref{alg:A2} to solve the finite-sum problem \eqref{eq:finite_sum}.
We can use a mini-batch $\widetilde{\Bc}_t$ of the size $\tilde{b}\in [n]$ to approximate $v_0$.
However, we make the following changes in Algorithm~\ref{alg:A1} to solve \eqref{eq:finite_sum}:
\begin{compactitem}
\item[$\diamond$] We compute $v_0 := \nabla{f}(x_0) = \frac{1}{n}\sum_{i=1}^n\nabla{f}_i(x_0)$, the full gradient of $f$ at $x_0$.
\item[$\diamond$] We evaluate $v_{t}  := \beta v_{t-1}  + \beta\big(\nabla{f}_{i_t}(x_{t}) - \nabla{f}_{i_t}(x_{t-1})\big) + (1-\beta)\nabla{f}_{j_t}(x_{t})$, where $i_t, j_t \in [n]$ are two independent random indices \nhan{uniformly generated from} $[n]$. 
\end{compactitem}
Since we set $\tilde{b} = n$, we need to change the weight $\beta$ in Theorems~\ref{th:constant_stepsize_convergence}, \ref{th:singe_loop_adapt_step}, and \ref{th:double_loop_convergence} to $\beta := 1 - \frac{c_1}{(m+1)^{2/3}}$ for some $0 < c_1 \leq (m+1)^{2/3}$.
With this choice of $\tilde{b}$ and $\beta$, the conclusions of Theorems~\ref{th:constant_stepsize_convergence}, \ref{th:singe_loop_adapt_step}, and \ref{th:double_loop_convergence} remain true.
But the number of stochastic gradient evaluations is at most, e.g., $\Tc_m := \BigO{n + \frac{\Delta_0^{3/2}}{\varepsilon^3}}$ in Theorem~\ref{th:constant_stepsize_convergence}.
To avoid overloading this paper, we omit the analysis here.

In terms of assumptions, apart from Assumptions~\ref{as:A0} and \ref{as:A1}, we still require Assumption~\ref{as:A1b} (i.e., \eqref{eq:bounded_variance2}) to hold for \eqref{eq:finite_sum}.
Hence, Algorithm~\ref{alg:A1} can solve \eqref{eq:finite_sum}, but it requires stronger assumptions (Assumptions~\ref{as:A0}, \ref{as:A1}, and \ref{as:A1b}) than ProxSVRG \cite{reddi2016proximal}, SpiderBoost \cite{wang2018spiderboost}, and ProxSARAH \cite{Pham2019}.
However, as a compensation, Algorithm~\ref{alg:A1}  uses a single-loop.

\beforesubsubsec
\subsubsection{\mytxtbi{The non-composite settings}}
\aftersubsubsec
If $\psi = 0$, then we obtain a non-composite setting of \eqref{eq:ncvx_prob} and \eqref{eq:finite_sum}, respectively.
The analysis in Theorems~\ref{th:constant_stepsize_convergence}, \ref{th:singe_loop_adapt_step}, and \ref{th:double_loop_convergence} can be modified to cover the non-composite setting of \eqref{eq:ncvx_prob}:
\begin{compactitem}
\item[$\bullet$] Step~\ref{step:o3} of Algorithm~\ref{alg:A1} becomes $x_1 := x_0 - \hat{\eta}_0v_0$, where $\hat{\eta}_0$ is a new step-size.
\item[$\bullet$] Step~\ref{step:i3} of Algorithm~\ref{alg:A1} reduces to $x_{t+1} := x_t - \hat{\eta}_tv_t$, where $\hat{\eta}_t$ is a new step-size.
\item[$\bullet$] The step-size $\hat{\eta}_t := \frac{2}{L\left[1 + (1+4\alpha_m^2)^{1/2}\right]}$ which combines both $\eta_t$ and $\gamma_t$ in Theorem~\ref{th:constant_stepsize_convergence}, where $\beta := 1 - \frac{1}{[\tilde{b}(m+1)]^{1/2}}$ and $\alpha_m := \frac{\beta^2(1-\beta^{2m})}{1-\beta^2}$.
\end{compactitem}
For clarity of exposition, we \nhan{omit} the analysis of this variant here.

\beforesec
\section{Extensions to Mini-batch Variants}\label{sec:sgd_mini_batch}
\aftersec
We consider the mini-batch variants of Algorithm~\ref{alg:A1} and Algorithm~\ref{alg:A2} for solving \eqref{eq:ncvx_prob}.
More precisely, the mini-batch SARAH-SGD estimator $\hat{v}_t$ for $\nabla{f}(x_t)$ is defined as
\myeq{eq:mini_batch_grad}{
\hat{v}_t := \beta \hat{v}_{t-1} + \frac{\beta}{b}\sum_{\xi_t\in\Bc_t}\left(\nabla{f}_{\xi_t}(x_t) - \nabla{f}_{\xi_t}(x_{t-1})\right) + \frac{1-\beta}{\hat{b}}\sum_{\zeta_t\in\hat{\Bc}_t}\nabla{f}_{\zeta_t}(x_t),
}
where $\Bc_t$ is a mini-batch of size $b$ and $\hat{\Bc}_t$ is a mini-batch of size $\hat{b}$ and independent of $\Bc_t$.
Here, we fix $\beta \in (0, 1)$ and the mini-batch sizes $b\in\Nbb_{+}$ and $\hat{b}\in\Nbb_{+}$ for all $t\geq 0$.
Note that the estimator \eqref{eq:mini_batch_grad} is an instance of \eqref{eq:vhat_t} when $G = \nabla{f}$.
For the sake of presentation, we only consider the constant step-size variant as a consequence of Theorem~\ref{th:constant_stepsize_convergence}.
We state our first result in the following theorem, whose proof can be found in Appendix~\ref{apdx:th:mini_batch_constant_stepsize_convergence}.

\revise{
\begin{theorem}\label{th:mini_batch_constant_stepsize_convergence}
Assume that Assumptions~\ref{as:A0}, \ref{as:A1}, and \ref{as:A1b} hold.
Let $\sets{x_t}_{t=0}^m$ be the sequence generated by Algorithm~\ref{alg:A1} to solve \eqref{eq:ncvx_prob} using the mini-batch update for $\hat{v}_t$ as in \eqref{eq:mini_batch_grad} at Step~\ref{step:i2} instead of $v_t$,
and the following parameter configuration:
\myeq{eq:constant_step_sizes}{
\left\{\begin{array}{ll}
\beta_t &= \beta := 1 - \frac{ \hat{b}^{1/2}}{[\tilde{b}(m+1)]^{1/2}} \vspace{1ex}\\
\gamma_t &= \gamma := \frac{3c_0 \hat{b}^{1/4}b^{1/2}}{\sqrt{13} [\tilde{b}(m+1)]^{1/4}}  \vspace{1ex}\\
\eta_t & = \eta := \frac{2}{L(3 + \gamma)},
\end{array}\right.
}
where $1\leq \hat{b} \leq \tilde{b}(m+1)$ and $0 < c_0 \leq  \frac{\sqrt{13}}{3b^{1/2}}$ is given.
Then, the following statements hold:
\begin{compactitem}
\vspace{0.5ex}
\item[$\mathrm{(a)}$] The parameters $\beta$, $\gamma$, and $\eta$ satisfy $\beta \in [0, 1)$, $\gamma \in (0, 1]$, and $\frac{1}{2L} < \eta \leq \frac{2}{3L}$.

\vspace{0.65ex}
\item[$\mathrm{(b)}$] Let $\overline{x}_m \sim \Uni{\set{x_t}_{t=0}^m}$ be the output of Algorithm~\ref{alg:A1}.
Then, we have
\myeq{eq:th61_main_estimate}{
\Exp{\norms{ \Grad_{\eta}(\overline{x}_m)}^2} \leq \frac{16\sqrt{13}L \tilde{b}^{1/4} \left[F(x_0) - F^{\star}\right]}{3c_0 \hat{b}^{1/4}b^{1/2}(m+1)^{3/4}} {~} + {~} \frac{208\sigma^2}{9\big[\hat{b}\tilde{b}(m+1)\big]^{1/2}}.
}
\item[$\mathrm{(c)}$]
Let us choose $\hat{b} = b \in\Nbb_{+}$ and $\tilde{b}  := c_1^2[b(m+1)]^{1/3}$ for some $c_1 \geq \frac{b^{1/3}}{(m+1)^{2/3}}$.
Then, the step-size $\gamma$ becomes $\gamma := \frac{3c_0b^{2/3}}{\sqrt{13c_1}(m+1)^{1/3}}$.
Moreover, for any $\varepsilon > 0$, the number of iterations $m$ to obtain $\Exp{\norms{\Grad_{\eta}(\overline{x}_m)}^2} \leq \varepsilon^2$ is at most
\myeqn{
m := \left\lfloor \frac{\Delta_0^{3/2}}{b\varepsilon^3}\right\rfloor,~~~~\text{where}~~\Delta_0 := \frac{16}{3}\left[\frac{\sqrt{13 c_1}L\left[F(x_0) - F^{\star}\right]}{c_0} + \frac{13\sigma^2}{3c_1}\right].
}
This is also the total number $\Tc_{\prox}$ of proximal operations $\prox_{\eta\psi}$, i.e., $\Tc_{\prox} = m$.
The total number of stochastic gradient evaluations $\nabla{f_{\xi}}(x_t)$ is at most
\myeqn{
\Tc_m := \left\lfloor \frac{c_1^2\Delta_0^{1/2}}{\varepsilon} + \frac{3\Delta_0^{3/2}}{\varepsilon^3} \right\rfloor.
}
\end{compactitem}
\end{theorem}
}

Theorem~\ref{th:mini_batch_constant_stepsize_convergence} states that using the mini-batch estimator $\hat{v}_t$ from \eqref{eq:mini_batch_grad}, the total number of stochastic gradient evaluations $\Tc_m$ in Algorithm~\ref{alg:A1} remains the same as in Theorem~\ref{th:constant_stepsize_convergence}.
However, the total number of proximal operations $\Tc_{\prox}$ reduces from $\BigO{\frac{\sigma}{\varepsilon^3}}$ to $\BigO{\frac{\sigma}{b\varepsilon^3}}$.

We can also modify Algorithm~\ref{alg:A2} to obtain a mini-batch variant.
The following theorem shows the convergence of this mini-batch variant whose proof can be found in Appendix~\ref{apdx:th:double_loop_convergence_minibatch}.

\revise{
\begin{theorem}\label{th:double_loop_convergence_minibatch}
Let $\sets{x^{(s)}_t}_{t=0\to m}^{s=1\to S}$ be the sequence generated by Algorithm~\ref{alg:A2} to solve \eqref{eq:ncvx_prob} using the mini-batch update for $\hat{v}_t$ as in \eqref{eq:mini_batch_grad} at Step~\ref{step:i2} instead of $v_t$, $\eta \in (0, \frac{1}{L})$, and 
\myeq{eq:step_size_gamma_t}{
\gamma_m := \frac{\delta}{L}~~\text{and}~~\gamma_t := \frac{\delta b}{Lb + L(1+L^2\eta^2)\big[\beta^2\gamma_{t+1} + \beta^4\gamma_{t+2} + \cdots + \beta^{2(m-t)}\gamma_m\big]},
}
where  $\delta := \frac{2}{\eta} - 2L > 0$ and $\beta := 1 - \frac{ \hat{b}^{1/2}}{[\tilde{b}(m+1)]^{1/2}}$.
Then, the following statements hold:
\begin{compactitem}
\item[$\mathrm{(a)}$]
If we choose the output of Algorithm~\ref{alg:A2} as $\overline{x}_T \sim \Unip{\pb}{\sets{x^{(s)}_t}_{t=0\to m}^{s=1\to S}}$ with the probability $\pb_t  := \Prob{\overline{x}_T = x_t^{(s)}} = \frac{\gamma_t}{S\Sigma_m}$, then the following estimate holds
\myeq{eq:double_loop_est2}{
\begin{array}{lcl}
\Exp{\Vert \Grad_{\eta}(\overline{x}_T)\Vert^2} & \leq &  \dfrac{8\sqrt{2}\left[ \hat{b}^{1/4}b^{1/2}L {~} + {~} [\tilde{b}(m+1)]^{1/4} \sqrt{L\delta}\right]}{L\eta^2\delta S(m+1)\hat{b}^{1/4}b^{1/2}}\big[F(\overline{x}^{(0)}) - F^{\star}\big] \vspace{1ex}\\
&& + {~}  \dfrac{8\sigma^2}{L^2\eta^2[\hat{b}\tilde{b}(m+1)]^{1/2}}.
\end{array}
}
\item[$\mathrm{(b)}$]
Given  $\varepsilon > 0$ and $c_1 > \frac{1}{m+1}$, let us choose   $\hat{b} = b \in \Nbb_{+}$ such that $1\leq b \leq \frac{2\sqrt{6\delta}}{L\sqrt{L}\eta\varepsilon}$, and
\myeqn{
\tilde{b} :=  \frac{24c_1}{L^2\eta^2 \varepsilon^2}\max\set{\sigma^2,1}~~~~\text{and}~~m +1 := \frac{24}{c_1L^2\eta^2 b\varepsilon^2}\max\set{\sigma^2,1}.
}
Then, the total number of iterations $T$ to achieve $\Exp{\norms{\Grad_{\eta}(\overline{x}_T)}^2} \leq \varepsilon^2$ is at most
\myeq{eq:S_iterations2}{
T :=  (m+1)S =  \left\lfloor \frac{96\sqrt{3}\max\set{1,\sigma}}{\eta^3L\sqrt{L\delta} b\varepsilon^3}\big[F(\overline{x}^{(0)}) - F^{\star}\big] \right\rfloor.
}
This is also the total number of proximal operations $\prox_{\eta\psi}$.
The total number of stochastic gradient evaluations $\nabla{f}_{\xi}(x_t)$ is at most 
\myeqn{
\Tc_{\nabla{f}} := \left\lfloor (c_1^2 + 3)\frac{96\sqrt{3}\max\set{1,\sigma}}{\eta^3L\sqrt{L\delta} \varepsilon^3}\big[F(\overline{x}^{(0)}) - F^{\star}\big] \right\rfloor.
}
\end{compactitem}
\end{theorem}
}

\begin{remark}[\textbf{Mini-batch and step-size trade-off}]\label{re:batch_step_size_trade_off}
\revise{
From Theorem~\ref{th:mini_batch_constant_stepsize_convergence}(c), we can see that $\gamma := \frac{3c_0b^{2/3}}{\sqrt{13c_1}(m+1)^{1/3}}$.
Clearly, if we use a large mini-batch size $b$ for $\hat{v}_t$, then we obtain a large value of the step-size $\gamma$.
Assume that $\gamma \approx 1$, which  is equivalent to $\frac{3c_0b^{2/3}}{\sqrt{13c_1}(m+1)^{1/3}} \approx 1$.
Moreover, from Theorem~\ref{th:mini_batch_constant_stepsize_convergence}(c), we also have $b(m+1) = \frac{\Delta_0^{3/2}}{\varepsilon^3}$.
Combining both conditions, we can roughly set $b \approx \frac{\sqrt{13c_1}\Delta_0^{1/2}}{3c_0\varepsilon}$ and $m + 1 \approx \frac{3c_0\Delta_0}{\sqrt{13c_1}\varepsilon^2}$.
Empirical evidence in Section~\ref{sec:num_experiments} will show that large step-size $\gamma$ leads to better performance.

For the mini-batch double-loop variant stated in Theorem~\ref{th:double_loop_convergence_minibatch}, the update \eqref{eq:step_size_gamma_t} of $\gamma_t$ hints that if $m$ is small, then the sequence of step-sizes $\set{\gamma_t}_{t=0}^m$ is large.
From Theorem~\ref{th:double_loop_convergence_minibatch}(b), we have $m  := \left\lfloor\frac{24}{c_1L^2\eta^2 b\varepsilon^2}\max\set{\sigma^2,1}\right\rfloor$.
Therefore, to obtain a small $m$, we need to choose $b$ large.
In this case, the total number of proximal operations $\prox_{\eta\psi}$ decreases, but the total number of stochastic gradient evaluations remains unchanged.}
\end{remark}

\begin{remark}[\textbf{Practical termination condition}]\label{re:last_iterate}
\revise{
In Algorithm~\ref{alg:A1}, Algorithm~\ref{alg:A2}, and their variants, we need to choose $\overline{x}_m$ or   $\overline{x}_T$ randomly among the iterate sequence generated up to the iteration $m$ or $T := S(m+1)$, respectively.
We can choose one of the two options:
\begin{compactitem}
\item[$\bullet$] 
We randomly generate an index $T_{*} \in \set{0,1,\cdots, m}$ (or $T_{*} \in \set{0,1,\cdots, T}$) using the probability distribution $\pb_t = \frac{\gamma_t}{\Sigma_t}$ (or $\pb_t = \frac{\gamma_t}{S\Sigma_t}$).
Then, we run Algorithm~\ref{alg:A1} or Algorithm~\ref{alg:A2} up to $T_{*}$ iterations.
The corresponding iterate at the $T_{*}$-th is $\overline{x}_m$ (or $\overline{x}_T$).

\item[$\bullet$]
We can choose the best-so-far iterate $\overline{x}_T$ based on the following guarantee:
\myeqn{
\min_{0\leq t\leq m}\norms{\Grad_{\eta}(x_t)} \leq \varepsilon ~~~\text{or}~~~\min_{0\leq t\leq m,1\leq s\leq S}\norms{\Grad_{\eta}(x_t^{(s)})} \leq \varepsilon.
}
\end{compactitem}
However, in practice, we often take the last iterate $x_m$ or $\overline{x}^{(S)}$ as the output of the algorithm which unfortunately does not have a theoretical guarantee in this paper.
}
\end{remark}

\beforesec
\section{Numerical Experiments}\label{sec:num_experiments}
\aftersec
In this section, we provide three examples to illustrate the performance of our algorithms and compare them with several  state-of-the-art methods.
We use different configurations of parameters to investigate the \nhan{practical} advantages and disadvantages of our methods.

\beforesubsec
\subsection{\bf Implementation details and configuration}\label{subsec:config}
\aftersubsec
\paragraph{\mytxtbi{Algorithms and competitors:}}
We implement the following variants of our algorithms:
\begin{compactitem}
\item[$\bullet$] 
Algorithm~\ref{alg:A1} with constant stepsizes stated in Theorem~\ref{th:constant_stepsize_convergence}. 
We denote it by \texttt{ProxHSGD-SL}.
The parameters are set as suggested by Theorem \ref{th:constant_stepsize_convergence}.
For the mini-batch variant stated in Theorem \ref{th:mini_batch_constant_stepsize_convergence}, we also fix \update{$\gamma := 0.95$}, and choose mini-batch sizes as suggested in Remark~\ref{re:batch_step_size_trade_off}.

\item[$\bullet$]
Algorithm~\ref{alg:A2} with constant and adaptive step-sizes as stated in Theorem~\ref{th:double_loop_convergence}.
We call these two variants by \texttt{ProxHSGD-RS1} for constant step-size, and \texttt{ProxHSGD-RS2} for adaptive step-size.
We set $\eta$ and $\beta$ as suggested by Theorem~\ref{th:double_loop_convergence} for constant step-size or by \eqref{eq:update_of_eta_t} for adaptive step-size.
For the mini-batch case, we choose \update{$\gamma := 0.95$ and compute the mini-batch accordingly as guided by Remark~\ref{re:batch_step_size_trade_off}.}
\end{compactitem}
We normalize datasets so that the average-Lipschitz constant $L$ in our experiment is $L = L_{\ell}$, the Lipschitz constant of the outer loss function $\ell$ specified in the sequel.
For comparison, we also implement the following algorithms from the literature:
\begin{compactitem}
\item[$\bullet$]
The proximal stochastic gradient methods, e.g., from \cite{ghadimi2016mini} with constant  and scheduled diminishing step-sizes $\eta := \eta_0 > 0$ and $\eta_t :=  \frac{\eta_0}{1 + \eta'\lfloor t/n\rfloor}$, respectively, where $\eta_0 > 0$ and $\eta' \geq 0$ that will be tuned in each experiment. 
We denote the SGD variant with constant step-size by \texttt{ProxSGD1} using $\eta' = 0$, and the SGD scheme with scheduled diminishing step-size by \texttt{ProxSGD2} using $\eta' > 0$.
\update{Without further specification, we will set $\eta' := 1.0$ and $\eta_0 := 0.01$ when the minibatch size $\hat{b} = 1$, whereas $\eta_0 := 0.05$ when $\hat{b} > 1$, which allows us  to obtain consistent performance.
But in many test cases below, we carefully tune these parameters to obtain the best results for fair comparison.}

\item[$\bullet$]
We also implement the proximal SpiderBoost method in \cite{wang2018spiderboost}, which works well in several examples, see \cite{Pham2019}.
Note that this algorithm can be viewed as an instance of ProxSARAH in \cite{Pham2019}, where we skip comparing with other variants here. 
We denote it by \texttt{ProxSpiderBoost}.
In this algorithm, we choose the constant step-size $\eta := \frac{1}{2L}$ and choose the optimal mini-batch size \update{$b := \lfloor\sqrt{n}\rfloor$ and epoch length $m := \lfloor\sqrt{n}\rfloor$ as suggested by the authors}.
This is a large constant step-size and it works really well in most cases.

\item[$\bullet$]
Another algorithm is the proximal SVRG scheme in \cite{li2018simple,reddi2016proximal} denoted by \texttt{ProxSVRG}.
\revise{While the learning rate $\eta$ suggested in \cite{reddi2016proximal} is $\eta := \frac{1}{3nL}$ for single sample and $\eta := \frac{1}{3L}$ for the mini-batch of size $b := \lfloor n^{2/3}\rfloor$, we find it not perform well in our experiments. 
Therefore, we also tune  \texttt{ProxSVRG} to find the best learning rate  in our experiments.}
\end{compactitem}
All the algorithms are implemented in Python running on a single node of a Linux server (called \textit{Longleaf}) with configuration: 3.40GHz Intel processors, 30M cache, and 256GB RAM.
For the last example, we also use \href{https://www.tensorflow.org}{\textbf{TensorFlow} (https://www.tensorflow.org)} to create networks and run simulation on a GPU system.
Since each algorithm uses different values of the step-size $\eta$, we  pick a fixed value $\eta := 0.5$ to compute the norm of gradient mapping $\Vert \Grad_{\eta}(x^{(s)}_t)\Vert$ for visualization and report in all methods.
We run the first and second examples up to $40$ epochs, whereas we increase up to $60$ epochs in the last example. 

\beforepara
\paragraph{\mytxtbi{Datasets:}}
Several datasets used in this paper are from \cite{CC01a}, which are available online at \href{https://www.csie.ntu.edu.tw/~cjlin/libsvm/}{https://www.csie.ntu.edu.tw/{$\sim$}cjlin/libsvm/}.
We use the following $6$ representative datasets:
\begin{compactitem}
\item[$\bullet$]
\textbf{Small and medium datasets:} 
Three different datasets:  \texttt{w8a} ($n = 49,749$, $p = 300$), \texttt{rcv1.binary} ($n=20242$, $p=47236$), and \texttt{real-sim} ($n=72309$, $p=20958$) are widely used in the literature. 

\item[$\bullet$]
\textbf{Large datasets:}
We also test the above algorithms on larger datasets: \texttt{url\_combined} ($n = 2,396,130; p = 3,231,961$), \texttt{epsilon} ($n = 400,000; p = 2,000$), and \texttt{news20.binary} ($n=19,996; p=1,355,191$). 
\end{compactitem}
Another well-known dataset is \texttt{mnist} \nhan{available at} \href{http://yann.lecun.com/exdb/mnist/}{http://yann.lecun.com/exdb/mnist/}.

\beforesubsec
\subsection{\bf Nonnegative principal component analysis}\label{subsec:exam1}
\aftersubsec
The first example is a non-negative principal component analysis (NN-PCA) model studied in \cite{reddi2016proximal}, which can be described as follows:
\myeq{eq:nn_pca}{
f^{\star} := \min_{x\in\R^p}\Big\{ f(x) := -\frac{1}{2n}\sum_{i=1}^nx^{\top}(z_iz_i^{\top})x ~~~\text{s.t.}~~~ \norms{x} \leq 1, ~x \geq 0 \Big\}.
}
Here, $\set{z_i}_{i=1}^n$ in $\R^p$ is a given set of samples.
By defining $f_i(x) := -\frac{1}{2}x^{\top}(z_iz_i^{\top})x$ for $i=1,\cdots, n$, and $\psi(x) := \delta_{\Xc}(x)$, the indicator of $\Xc := \set{x\in\R^p \mid \norms{x} \leq 1, x \geq 0}$, we can formulate \eqref{eq:nn_pca} into \eqref{eq:finite_sum}.
Moreover, since $z_i$ is normalized, the Lipschitz constant of $\nabla{f}_i$ is $L = 1$ for $i=1,\cdots, n$.

\beforepara
\paragraph{\mytxtbi{$\mathrm{(a)}$~Single-loop methods with single sample:}}
Our first experiment is to compare Algorithm~\ref{alg:A1} using constant step-sizes with the two different variants of SGD.
\revise{We use different values $c_1$ as suggested by Theorem~\ref{th:constant_stepsize_convergence} to form an initial mini-batch $\tilde{b}$.
More specifically, after some experiments, we set $c_1 := 10$ for \texttt{w8a}, $c_1 := 40$ for \texttt{rcv1\_train.binary}, and $c_1 := 50$ for \texttt{real-sim}.
To have a fair comparison, we also carefully tune the learning rate for both \texttt{ProxSGD1} and \texttt{ProxSGD2} to achieve their best performance.}
The performance of these methods is plotted in Figure~\ref{fig:NN_single_loop_new} for three different datasets.

\begin{figure}[htp!]
\begin{center}
\includegraphics[width = 1\textwidth]{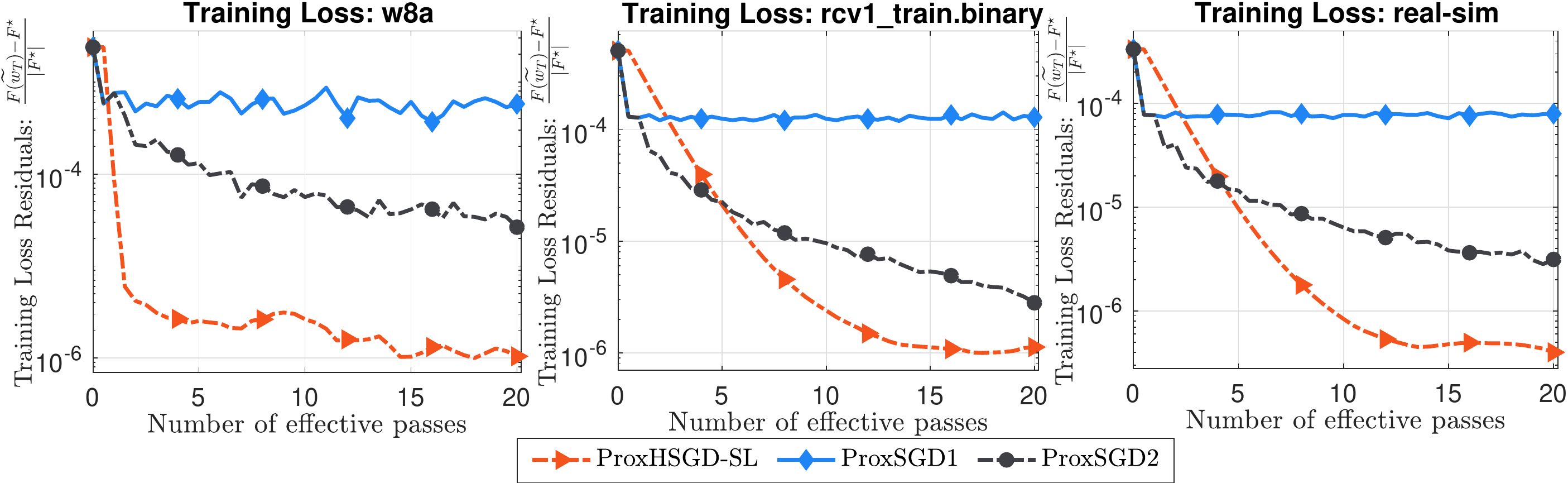}
\vspace{0.5ex}
\includegraphics[width = 1\textwidth]{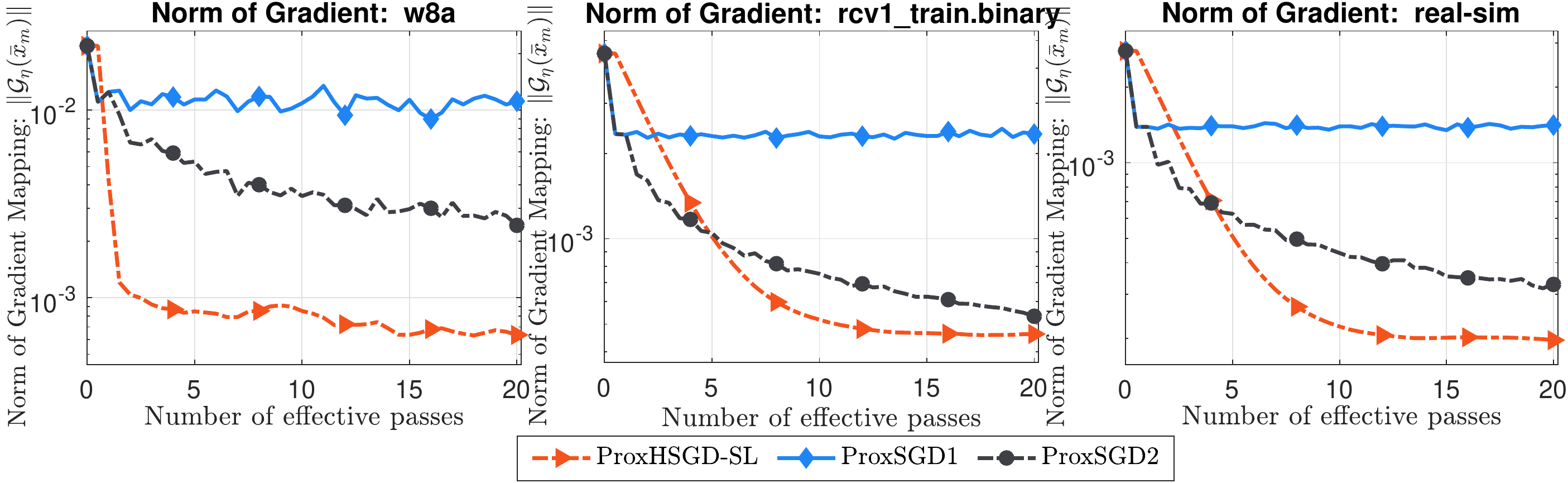}
\vspace{-2ex}
\caption{The relative training loss residuals and the absolute gradient mapping norms of \eqref{eq:nn_pca} on three small datasets: \mytxtbi{Single-loop with single-sample}.}\label{fig:NN_single_loop_new}
\end{center}
\vspace{-5ex}
\end{figure}

As we can observe from Figure~\ref{fig:NN_single_loop_new} that our \texttt{ProxHSGD-SL} variant works relatively well and outperforms both \texttt{ProxSGD1} and \texttt{ProxSGD2}.
However, it then slows down or is saturated at a certain value of the loss function, probably due to the effect of the SGD term $u_t$ in our estimator $v_t$.
It is not surprise that \texttt{ProxSGD2} works better than \texttt{ProxSGD1} due to the use of a scheduled diminishing learning rate. 

\beforepara
\paragraph{\mytxtbi{$\mathrm{(b)}$~Single-loop methods with mini-batch:}}
Next, we test \texttt{ProxHSGD-SL},  \texttt{ProxSGD1}, and  \texttt{ProxSGD2} with mini-batches on six different datasets.
We use the mini-batch size $\hat{b} := 50$ for \texttt{w8a}, \texttt{rcv1-binary}, \texttt{real-sim}, and \texttt{news20.binary}, $\hat{b} := 300$ for \texttt{epsilon}, and $\hat{b} := 500$ for \texttt{url\_combined}. 
For \texttt{ProxSpiderBoost} and \texttt{ProxSVRG}, we set the mini-batch sizes as stated in Subsection~\ref{subsec:config}. 

\revise{
For three small datasets, as suggested by Theorem~\ref{th:mini_batch_constant_stepsize_convergence}, after some simple experiments, we find that $c_0 := 9$ and $c_1 := 10$ for \texttt{w8a}, $c_0 := 18$ and $c_1 := 9$  for \texttt{rcv1\_train.binary}, and $c_0 := 30$ and $c_1 := 15$ for \texttt{real-sim} work well.
While we choose the same mini-batch size in \texttt{ProxSGD1}, and  \texttt{ProxSGD2} as described above, we again tune their learning rates to have the best performance.
}
The results are shown in Figure~\ref{fig:NN_single_loop_mini_batch_new}.

\begin{figure}[htp!]
\begin{center}
\includegraphics[width = 1\textwidth]{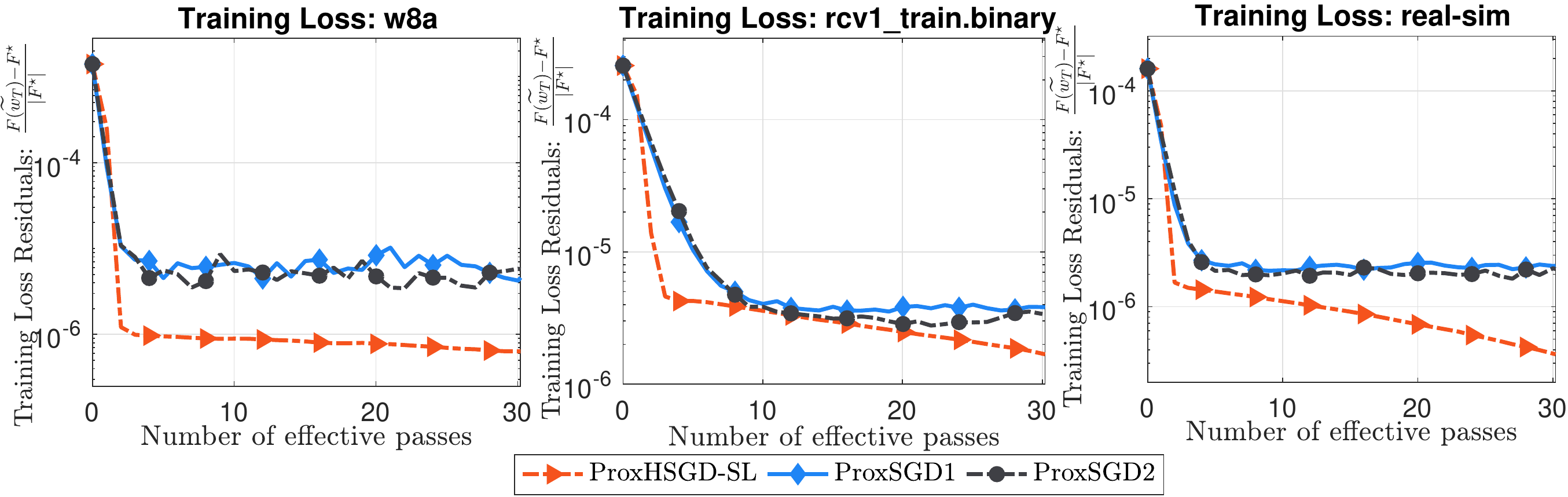}
\vspace{0.5ex}
\includegraphics[width = 1\textwidth]{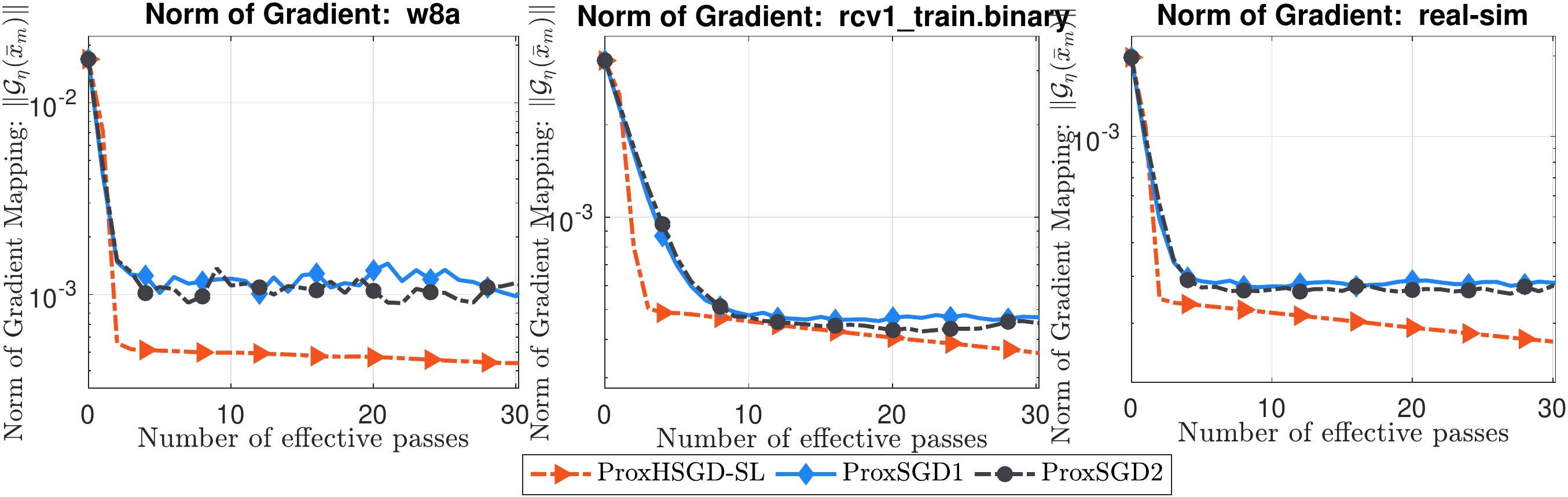}
\vspace{-2ex}
\caption{The relative training loss residuals and the absolute gradient mapping norms of \eqref{eq:nn_pca} on three small datasets: \mytxtbi{Single-loop  with mini-batch}.}
\label{fig:NN_single_loop_mini_batch_new}
\end{center}
\vspace{-5ex}
\end{figure}

In this case, both \texttt{ProxSGD1} and \texttt{ProxSGD2} perform much better  than the single sample case, and are comparable with \texttt{ProxHSGD-SL}.
However, \texttt{ProxHSGD-SL} still outperforms \texttt{ProxSGD1} and \texttt{ProxSGD2} in the first and the last datasets, while it is slightly better than  \texttt{ProxSGD1} and \texttt{ProxSGD2} in the second one.

Finally, we again evaluate three single-loop algorithms with mini-batches on three larger datasets: \texttt{url\_combined}, \texttt{epsilon}, and \texttt{news20.binary}.
Here, based on Theorem~\ref{th:mini_batch_constant_stepsize_convergence}, we find that $c_0 := 40\times 10^3$ and $c_1 := 20$ for \texttt{url\_combined}, $c_0 := 15\times 10^3$ and $c_1 := 30$ for \texttt{epsilon}, and 
$c_0 := 10\times 10^3$ and $c_1 := 20$ for \texttt{news20.binary} work well for our method.
Figure~\ref{fig:NN_single_loop_mini_batch2} shows the convergence behavior of three algorithms.

\begin{figure}[htp!]
\begin{center}
\includegraphics[width = 1\textwidth]{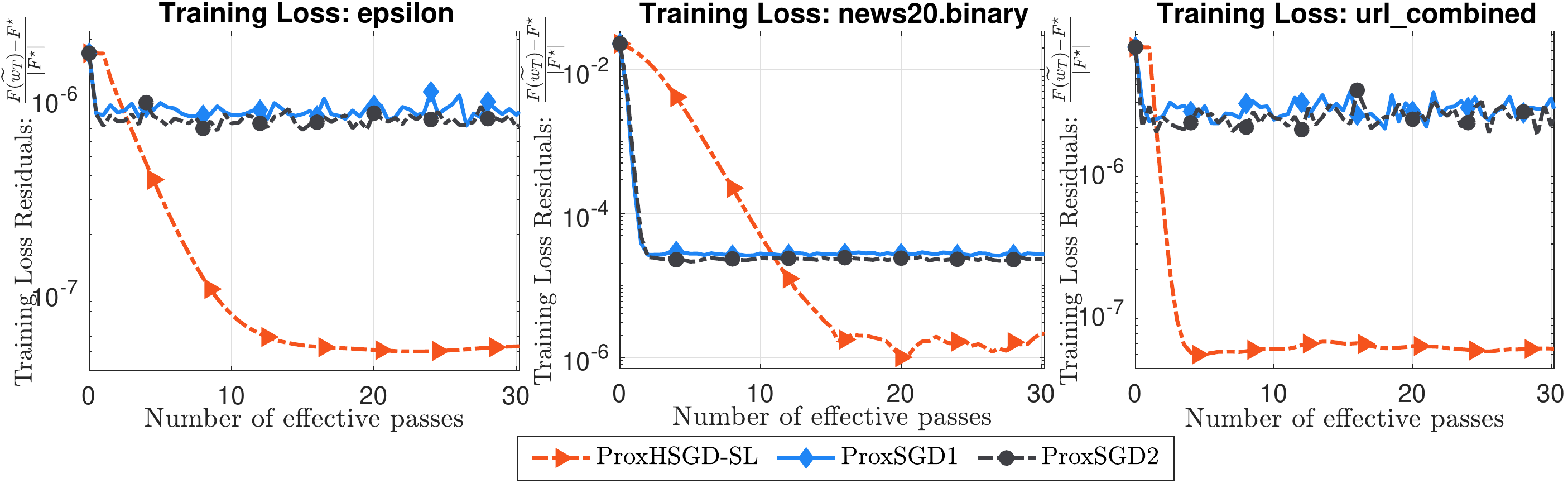}
\vspace{0.5ex}
\includegraphics[width = 1\textwidth]{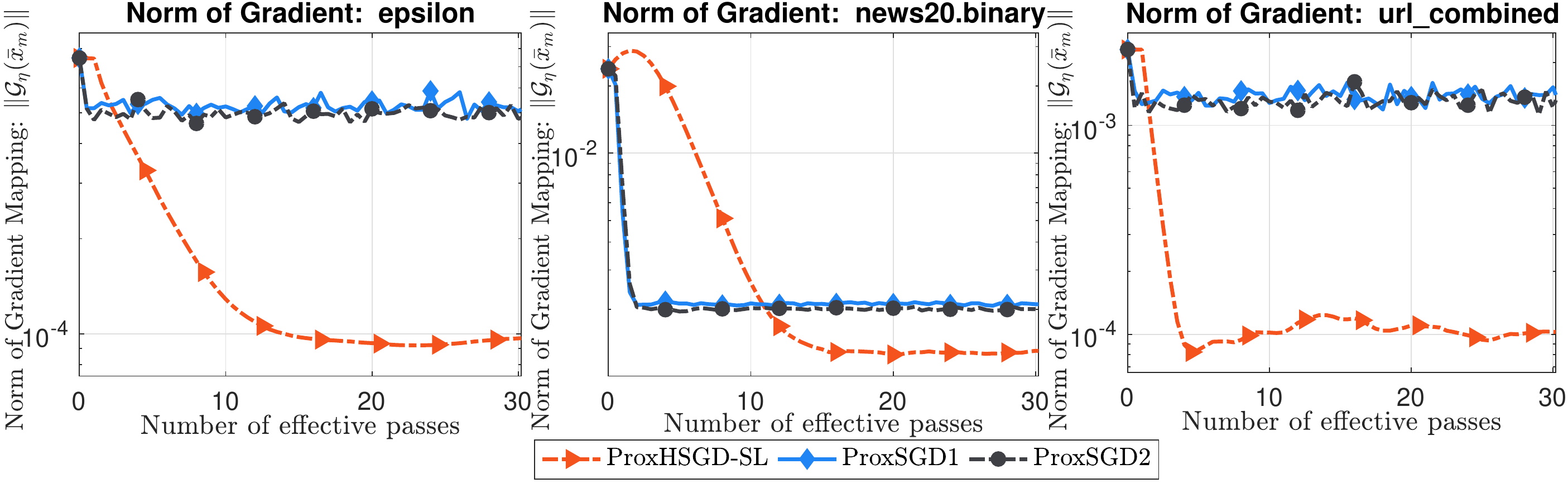}
\vspace{-3ex}
\caption{The relative training loss residuals and the absolute gradient mapping norms of \eqref{eq:nn_pca} on three large datasets: \mytxtbi{Single-loop  with mini-batch}.}
\label{fig:NN_single_loop_mini_batch2}
\end{center}
\vspace{-1ex}
\end{figure}

We obverse that  from Figure~\ref{fig:NN_single_loop_mini_batch2}  \texttt{ProxHSGD-SL} performs much better than \texttt{ProxSGD1} and \texttt{ProxSGD2}, especially in the first and the third datasets. 
From the experiments (a) and (b), we believe that ProxHSGD-SL generally outperforms  \texttt{ProxSGD} in these tests.

\beforepara
\paragraph{\mytxtbi{$\mathrm{(c)}$~Double-loop methods with mini-batch:}}
Our next test is on double-loop methods.
We compare two different double-loop methods: \texttt{ProxSVRG} and \texttt{ProxSpiderBoost} with two restarting variants of Algorithm~\ref{alg:A2}:  \texttt{ProxHSGD-RS1} and  \texttt{ProxHSGD-RS2}.
We use mini-batch for this test since  \texttt{ProxSpiderBoost} only has guarantee  for mini-batch variants.

First, let us test there algorithms using recommended learning rates and mini-batch sizes that have convergence guarantee.
Note that  both \texttt{ProxSpiderBoost} and \texttt{ProxSVRG} use a large learning rate $\eta := \frac{1}{2L}$ and $\eta := \frac{1}{3L}$, respectively. 
As observed in \cite{Pham2019}, these learning rates work well.
We use our step-sizes and mini-batch sizes as suggested in Remark~\ref{re:batch_step_size_trade_off}.
\revise{To have a fair assessment, we also add a tuning variant (\texttt{Tuned}) of both  \texttt{ProxSpiderBoost} and \texttt{ProxSVRG}, where we heuristically tune their learning rate to get the best performance.}
The results of this test are reported in Figure~\ref{fig:NN_double_loop2_notune}.

\begin{figure}[hpt!]
\vspace{-2ex}
\begin{center}
\includegraphics[width = 1\textwidth]{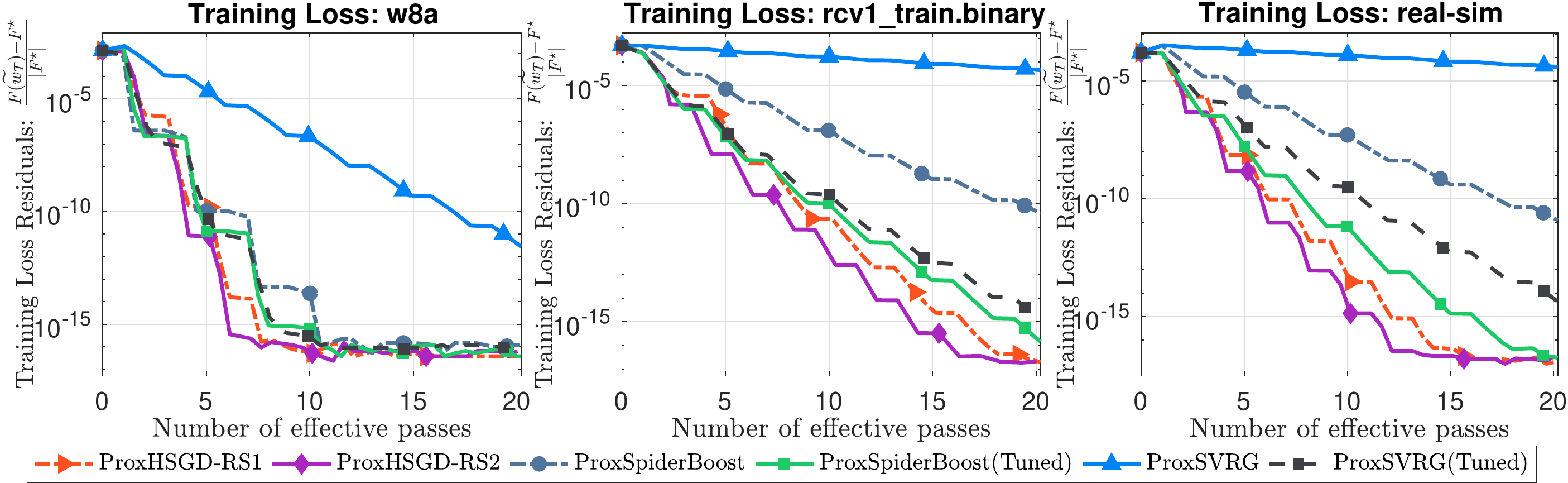}
\vspace{0.5ex}
\includegraphics[width = 1\textwidth]{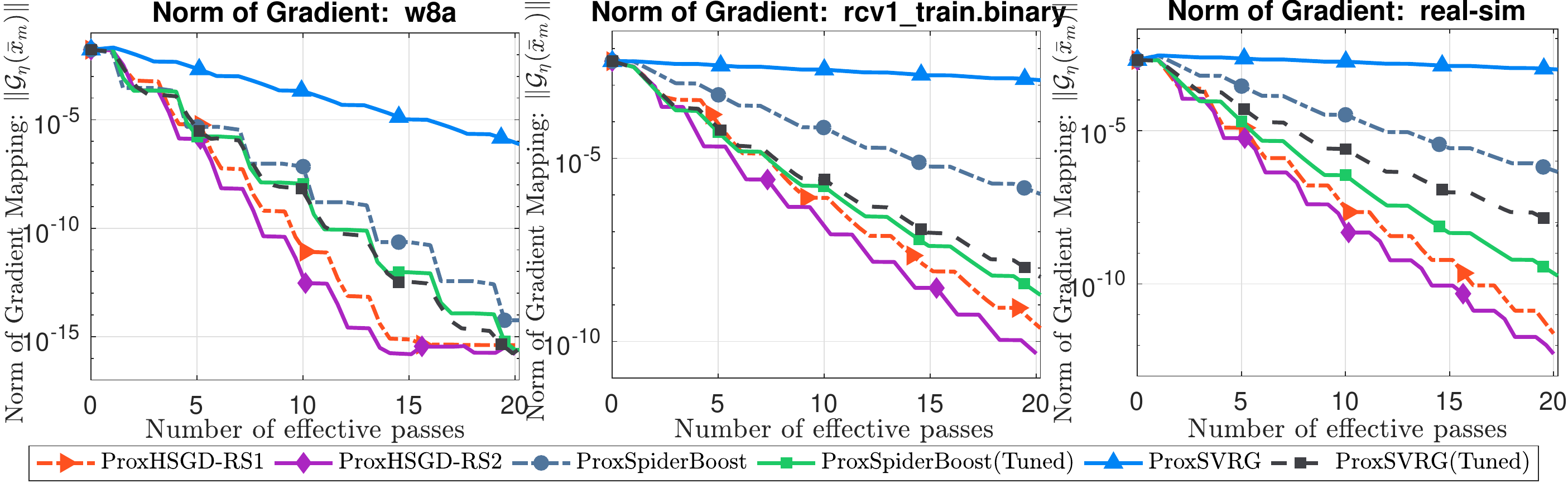}
\vspace{-3ex}
\caption{The relative training loss residuals and the absolute gradient mapping norms of \eqref{eq:nn_pca} on three small datasets using both recommended (theoretical) and tuned step-sizes: \mytxtbi{Double-loop  with mini-batch}.}
\label{fig:NN_double_loop2_notune}
\end{center}
\vspace{-4ex}
\end{figure}

As we can see from Figure~\ref{fig:NN_double_loop2_notune} that  both \texttt{ProxHSGD-RS1} and \texttt{ProxHSGD-RS2} highly outperform \texttt{ProxSVRG} and \texttt{ProxSpiderBoost}, especially in the last two datasets.
Due to the use of adaptive step-size, \texttt{ProxHSGD-RS2} appears to be better than \texttt{ProxHSGD-RS1}.
\revise{For the tuned learning rates, both  \texttt{ProxSVRG(Tuned)} and  \texttt{ProxSpiderBoost(Tuned)} are relatively comparable with \texttt{ProxHSGD-RS1} and \texttt{ProxHSGD-RS2} on these datasets.
More precisely, \texttt{ProxHSGD-RS1} and \texttt{ProxHSGD-RS2} are slightly better than both  \texttt{ProxSVRG(Tuned)} and  \texttt{ProxSpiderBoost(Tuned)}, especially in the last dataset.
But, \texttt{ProxSpiderBoost(Tuned)} is still slightly better than \texttt{ProxSVRG(Tuned)} in the second and third datasets.}

Now, we test  these variants on three larger datasets: \texttt{url\_combined}, \texttt{epsilon}, and \texttt{news20.binary}. 
Their performance is shown in Figure~\ref{fig:NN_double_loop_mini_batch2}.

\begin{figure}[htp!]
\begin{center}
\includegraphics[width = 1\textwidth]{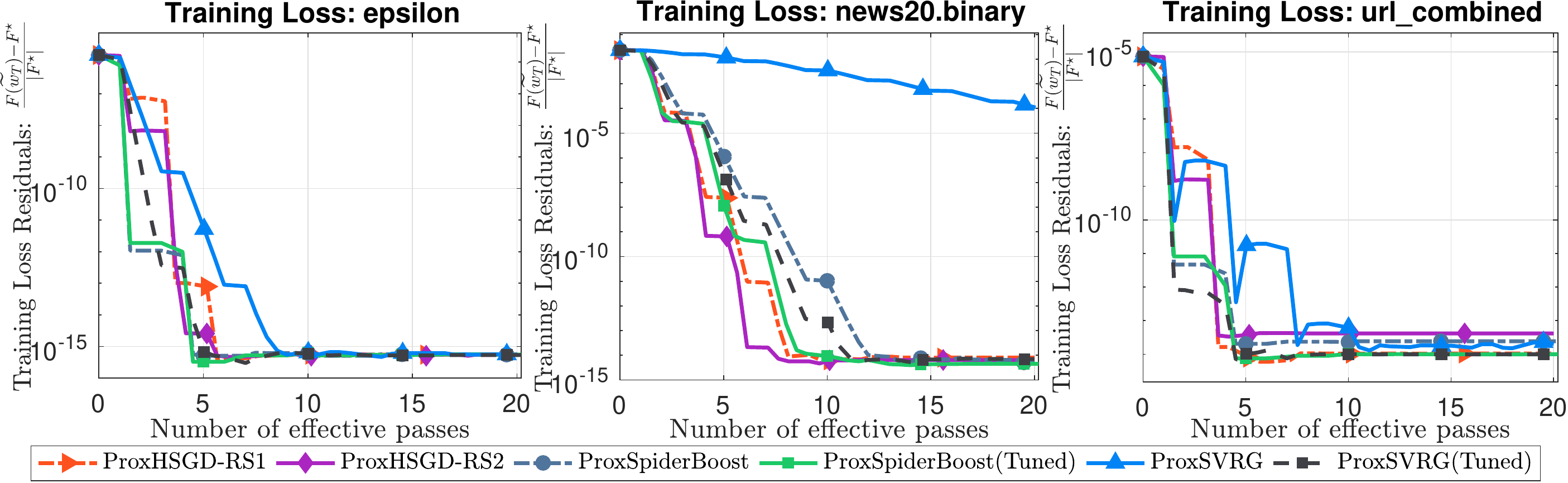}
\vspace{0.5ex}
\includegraphics[width = 1\textwidth]{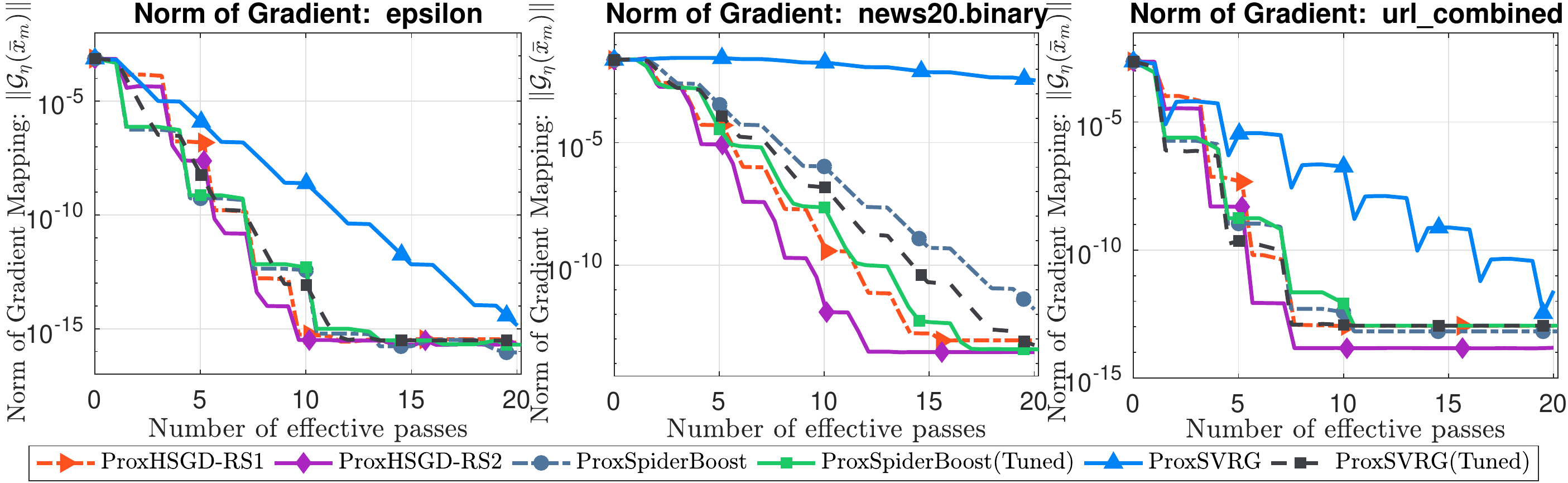}
\vspace{-3ex}
\caption{The relative training loss residuals and the absolute gradient mapping norms of \eqref{eq:nn_pca} on three large datasets: \mytxtbi{Double-loop  with mini-batch}.}
\label{fig:NN_double_loop_mini_batch2}
\end{center}
\vspace{-4ex}
\end{figure}

\revise{
Without tuning learning rates, we observe similar behavior as in Figure~\ref{fig:NN_double_loop_mini_batch2}, where our methods, \texttt{ProxHSGD-RS1} and \texttt{ProxHSGD-RS2}, highly outperform  \texttt{ProxSVRG} and are comparable with   \texttt{ProxSpiderBoost}.
Again, with tuned learning rates, \texttt{ProxSVRG(Tuned)} and \texttt{ProxSpiderBoost(Tuned)}  are more comparable with \texttt{ProxHSGD-RS1} and \texttt{ProxHSGD-RS2}, but our methods are still slightly better than their competitors in the second dataset.
However, we do not observe significant difference between the the adaptive step-size variant, \texttt{ProxHSGD-RS2} and the constant one, \texttt{ProxHSGD-RS1}, in this test.}

\beforesubsec 
\subsection{\bf Binary classification with nonconvex models}
\aftersubsec
In this example, we consider the following binary classification model involving a nonconvex objective function and a convex regularizer broadly studied in the literature:
\vspace{-0.5ex}
\myeq{eq:sparse_BinClass_ncvx}{
\min_{x\in\R^p}\set{ F(x) := \frac{1}{n}\sum_{i=1}^n\ell(a_i^{\top}x, b_i) + \psi(x)},
\vspace{-0.5ex}
}
where $\set{(a_i, b_i)}_{i=1}^n \subset\R^p\times \set{-1,1}^n$ is a given training dataset, $\psi$ is a convex regularizer, and $\ell(\cdot,\cdot)$ is a given smooth and nonconvex loss function.
By setting $f_i(w) := \ell(a_i^{\top}w, b_i)$ and  choosing a convex regularizer $\psi$, we obtain the form \eqref{eq:finite_sum} that satisfies Assumptions~\ref{as:A0} and \ref{as:A1}.

We consider the following settings for the choice of $\ell$ and $\psi$, where \nhan{the first} three models were studied in \cite{zhao2010convex}, and the last one has been used in \cite{metel2019simple}:
\begin{compactitem}
\item[$\diamond$]
\textbf{Normalized sigmoid loss:}
$\ell_1(s, \tau) := 1 - \tanh(\tau s)$ for a given $\omega > 0$ and $\psi(x) := \lambda\norms{x}_1$.
Here, $\ell_1(\cdot,\tau)$ is $L$-smooth with respect to $s$, where $L : \approx 0.7698$.

\item[$\diamond$]
 \textbf{Nonconvex loss in 2-layer neural networks:}
$\ell_2(s, \tau) := \left(1 - \frac{1}{1 + \exp(-\tau s)}\right)^2$ and $\psi(x) := \lambda\norms{x}_1$.
This function is also $L$-smooth with $L = 0.15405$.

\item[$\diamond$]
\textbf{Logistic difference loss:} $\ell_3(s,\tau) := \ln(1 + \exp(-\tau s)) - \ln(1 + \exp(-\tau s - 1))$ and $\psi(x) := \lambda\norms{x}_1$.
This function is $L$-smooth with $L = 0.092372$.

\item[$\diamond$]
\textbf{Lorenz loss:}
$\ell_4(s,\tau) := \ln(1 +(\tau s - 1)^2)$ if $\tau s \le 1$, and $\ell_4(s,\tau) = 0$, otherwise, and $\psi(x) = \lambda \|x\|_1$. 
This function is $L$-smooth with $L = 4$.
\end{compactitem}
We set the regularization parameter $\lambda := \frac{1}{n}$ in all the tests which gives us relatively sparse solutions.
We test the above algorithms on different scenarios ranging from small to large datasets using different algorithms.

\beforepara
\paragraph{\mytxtbi{$(a)$~Single-loop schemes with single-sample:}}
Our first experiment for \eqref{eq:sparse_BinClass_ncvx} is on single-loop variants with single sample.
For this test, we only choose the losses $\ell_3$ and $\ell_4$ with two well-known datasets: \texttt{rcv1-binary} and \texttt{real-sim} to avoid overloading the paper.
\revise{The result of our test on the $\ell_3$-loss is shown in Figures~\ref{fig:binary_loss3_single_loop_single_sample} for three algorithms using the same setting as in Subsection \ref{subsec:exam1}, where we attempt to tune  the learning rates for three competitors: \texttt{ProxSGD1}, \texttt{ProxSGD2}, and \texttt{ProxSVRG}.}

\begin{figure}[hpt!]
\vspace{-2ex}
\begin{center}
\includegraphics[width = 1\textwidth]{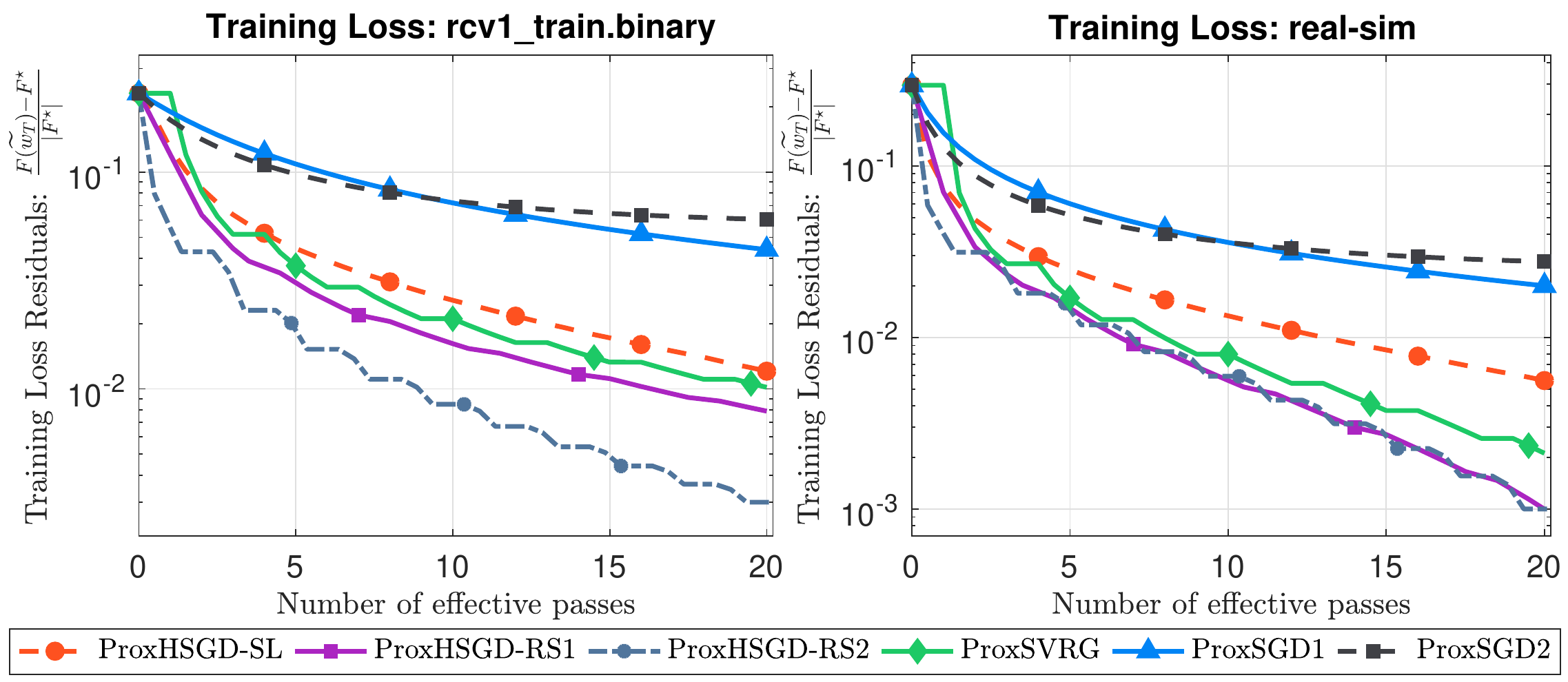}
\vspace{0.5ex}
\includegraphics[width = 1\textwidth]{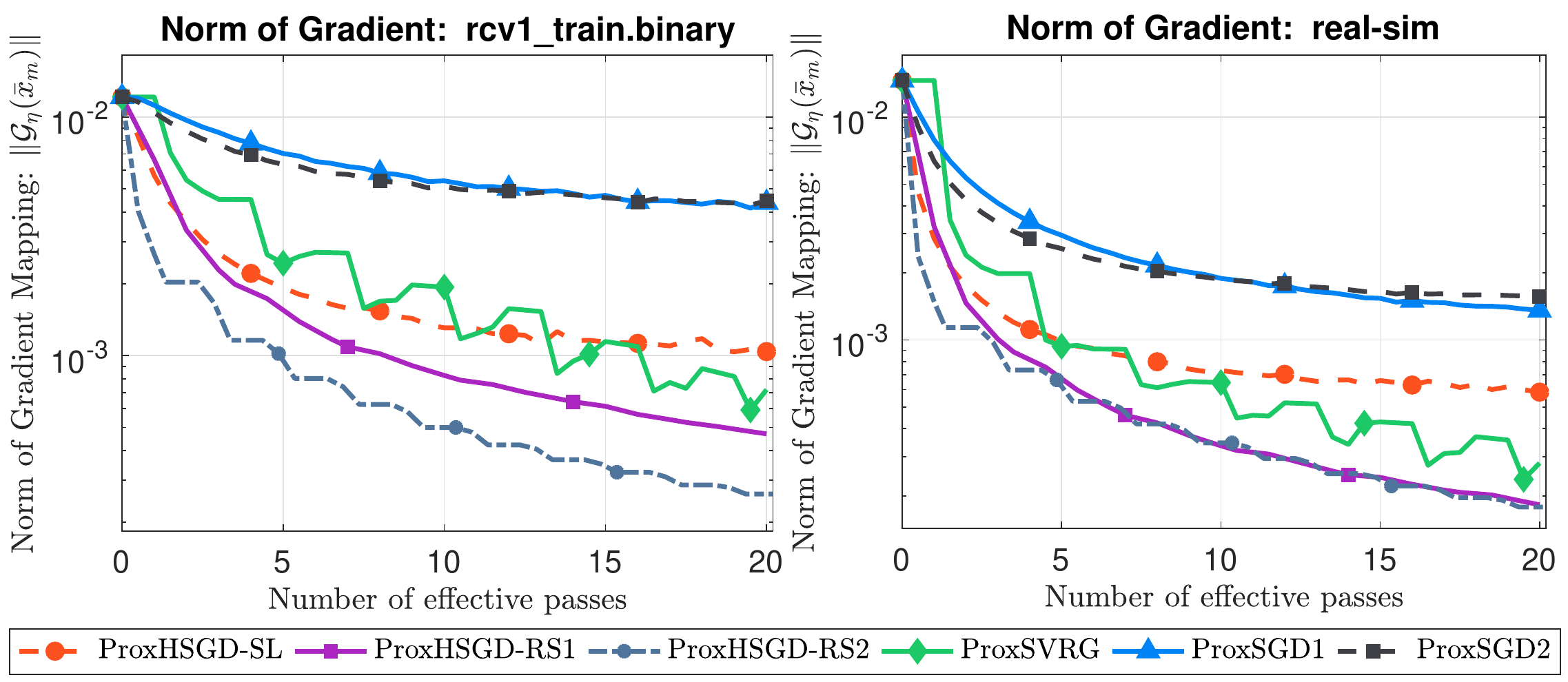}
\vspace{-2ex}
\caption{The relative training loss residuals and the absolute gradient mapping norms of \eqref{eq:sparse_BinClass_ncvx} with the nonconvex training loss $\ell_3$ on the two datasets: \mytxtbi{Single-loop  with single-sample}.}\label{fig:binary_loss3_single_loop_single_sample}
\end{center}
\vspace{-5ex}
\end{figure}

\revise{For the loss $\ell_3$ with $20$ epochs, \texttt{ProxSGD1} and \texttt{ProxSGD2} show their less competitive  performance than our methods and \texttt{ProxSVRG}.
While \texttt{ProxSGD2} is not better than \texttt{ProxSGD1} as observed in the previous example, \texttt{ProxSVRG} with tuned learning rate works really well and beats our \texttt{ProxHSGD-SL}.
The restarting variants of our methods highly outperform \texttt{ProxSGD1} and \texttt{ProxSGD2} and is slightly better than \texttt{ProxSVRG}.}

Additionally, the results of these six algorithms on the $\ell_4$-loss using the same setting are shown in Figure~\ref{fig:binary_loss4_single_loop_single1}.
\begin{figure}[hpt!]
\vspace{-2ex}
\begin{center}
\includegraphics[width = 1\textwidth]{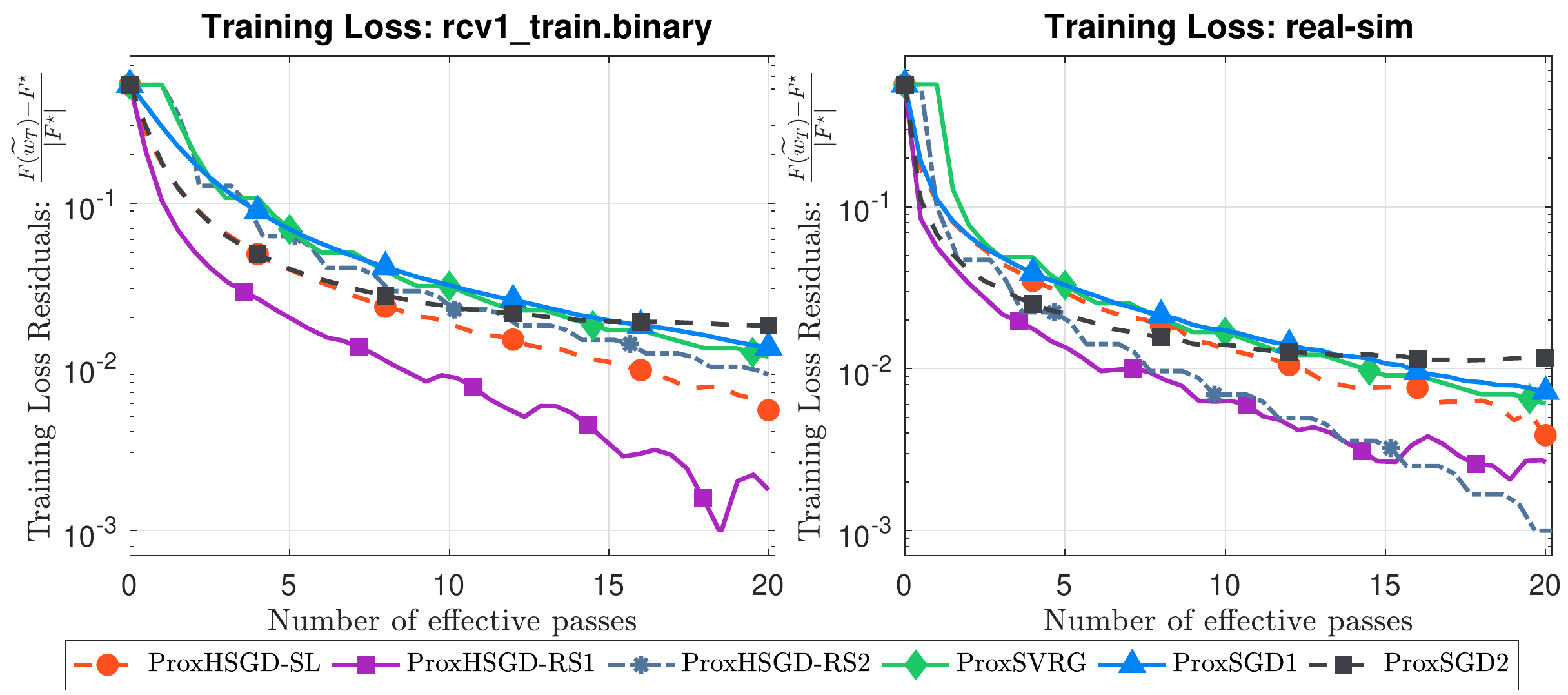}
\vspace{0.5ex}
\includegraphics[width = 1\textwidth]{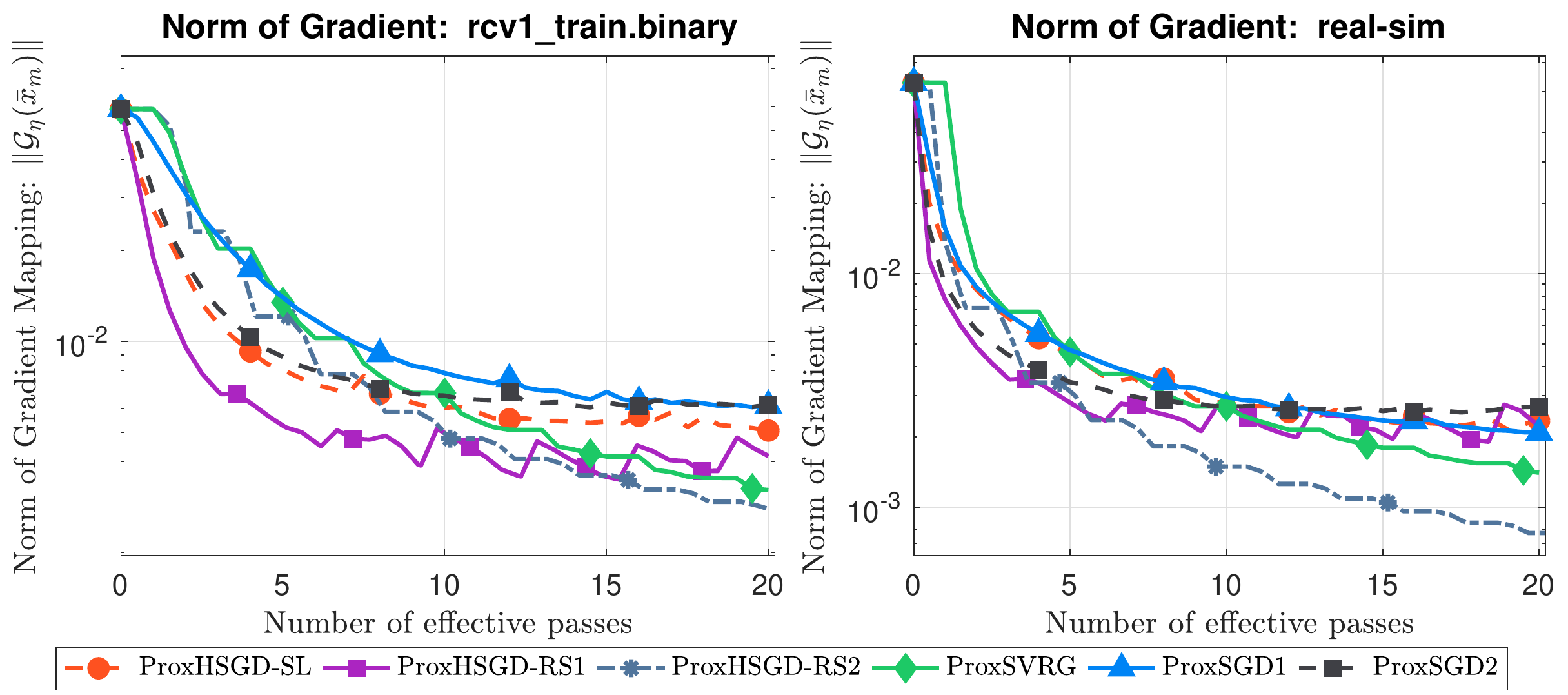}
\vspace{-2ex}
\caption{The relative training loss residuals and the absolute gradient mapping norms of \eqref{eq:sparse_BinClass_ncvx} with the nonconvex training loss $\ell_4$ on the two datasets: \mytxtbi{Single-loop  with single-sample}.}\label{fig:binary_loss4_single_loop_single1}
\end{center}
\vspace{-5ex}
\end{figure}

\revise{Figure~\ref{fig:binary_loss4_single_loop_single1} shows improved performance of both \texttt{ProxSGD1}, and \texttt{ProxSGD2}.
In this case, these methods are comparable with our \texttt{ProxHSGD-SL} and \texttt{ProxSVRG}.
However, our restarting variants,  \texttt{ProxHSGD-RS1} and  \texttt{ProxHSGD-RS2} are still slightly better than their competitors.
}

\beforepara
\paragraph{\mytxtbi{$\mathrm{(b)}$~\nhan{Complete test} on the mini-batch case:}}
Finally, we carry out a \nhan{more thorough} test on four different losses using two small datasets and two large datasets.
We compare five different methods as shown in Tables~\ref{tbl:small_datasets} and \ref{tbl:large_datasets}.
We report the relative training loss residuals, the absolute norms of gradient mapping, the training accuracy, and the test accuracy.
\revise{Since \texttt{ProxSVRG} with the theoretical learning rate $1/3L$ performs quite poorly, we carry out a grid search between $[1/(15L), 5/(3L)]$ to find a good learning rate for this experiment.}

Table~\ref{tbl:small_datasets} reports the results on the two small datasets: \texttt{rcv1\_train.binary} and \texttt{real-sim} after $20$ and $40$ epochs, respectively.

\begin{table}[hpt!]
\newcommand{\cellbest}[1]{{\!\!}{\color{blue}\textbf{#1}}{\!\!}}
\newcommand{\cellbetter}[1]{{\!\!}{\color{blue}#1}{\!\!}}
\newcommand{\cellr}[1]{{}{\color{red}#1}{}}
\newcommand{\cell}[1]{{\!\!}#1{\!\!}}
\begin{center}
\resizebox{\textwidth}{!}{
\begin{tabular}{|l||c|c||c|c||c|c||c|c||}
\hline
 \multicolumn{9}{|c|}{The loss function $\ell_1$}\\\hline
& \multicolumn{2}{|c|}{Training Loss Residual}& \multicolumn{2}{|c||}{$\norms{G_{\eta}(w_T)}$}& \multicolumn{2}{|c|}{Training Accuracy}& \multicolumn{2}{|c|}{Test Accuracy}\\\cline{2-9}
& \cell{20th ep.} & \cell{40th ep.} & \cell{20th ep.} & \cell{40th ep.} &  \cell{20th ep.} & \cell{40th ep.} & \cell{20th ep.} & \cell{40th ep.}\\\hline
\multirow{1}{*}{\text{~~~Algorithms}} & \multicolumn{8}{|c||}{\textbf{rcv1\_train.binary ($\boldsymbol{n=20,242,p=47,236}$)}}\\\hline
\cell{\textbf{ProxHSGD-SL}}  & \cellbetter{8.964e-02} & \cellbetter{3.409e-02} & \cellbetter{9.806e-05} & \cellbetter{3.043e-05} & \cellbetter{0.949} & \cellbetter{0.954} & \cellbetter{0.938} & \cellbetter{0.947}\\\hline
\cell{\textbf{ProxHSGD-RS1}}  & \cellbest{3.478e-02} & \cellbest{1.888e-04} & \cellbest{3.541e-05} & \cellbest{1.586e-05} & \cellbest{0.955} & \cellbest{0.960} & \cellbest{0.947} & \cellbest{0.956}\\\hline
\cell{ProxSpiderBoost}  & \cell{1.649e-01} & \cell{8.281e-02} & \cell{2.664e-04} & \cell{8.734e-05} & \cell{0.944} & \cell{0.950} & \cell{0.934} & \cell{0.938}\\\hline
\cell{ProxSVRG}  & \cell{2.918e-01} & \cell{1.574e-01} & \cell{1.578e-02} & \cell{1.212e-02} & \cell{0.626} & \cell{0.825} & \cell{0.004} & \cell{0.561}\\\hline
\cell{ProxSGD2}  & \cell{1.505e-01} & \cell{7.110e-02} & \cell{2.446e-04} & \cell{8.607e-05} & \cell{0.945} & \cell{0.951} & \cell{0.936} & \cell{0.941}\\\hline 
\multirow{1}{*}{\text{~~~Algorithms}} & \multicolumn{8}{|c||}{\textbf{\textbf{real-sim ($\boldsymbol{n=72,309,p=20,958}$)}}}\\\hline
\cell{\textbf{ProxHSGD-SL}}  & \cellbetter{4.554e-02} & \cellbetter{1.742e-02} & \cellbetter{1.377e-05} & \cellbetter{4.317e-06} & \cellbetter{0.977} & \cellbetter{0.981} & \cell{0.659} & \cell{0.647}\\\hline
\cell{\textbf{ProxHSGD-RS1}}  & \cellbest{1.756e-02} & \cellbest{1.006e-04} & \cellbest{4.699e-06} & \cellbest{1.837e-06} & \cellbest{0.981} & \cellbest{0.984} & \cell{0.646} & \cell{0.634}\\\hline
\cell{ProxSpiderBoost}  & \cell{1.459e-01} & \cell{8.004e-02} & \cell{1.124e-04} & \cell{3.537e-05} & \cell{0.966} & \cell{0.973} & \cellbetter{0.695} & \cellbetter{0.670}\\\hline
\cell{ProxSVRG}  & \cell{4.150e-02} & \cell{2.443e-02} & \cell{3.377e-03} & \cell{1.783e-03} & \cell{0.962} & \cell{0.964} &  \cellbest{0.963} & \cellbest{0.964}\\\hline
\cell{ProxSGD2}  & \cell{7.619e-02} & \cell{3.613e-02} & \cell{3.333e-05} & \cell{1.087e-05} & \cell{0.974} & \cell{0.978} & \cell{0.669} & \cell{0.653}\\\hline\hline
\multicolumn{9}{|c|}{The loss function $\ell_2$}\\\hline
& \multicolumn{2}{|c|}{Training Loss Residual}& \multicolumn{2}{|c|}{$\norms{G_{\eta}(w_T)}$}& \multicolumn{2}{|c|}{Training Accuracy}& \multicolumn{2}{|c|}{Test Accuracy}\\\cline{2-9}
& \cell{20th ep.} & \cell{40th ep.} & \cell{20th ep.} & \cell{40th ep.} &  \cell{20th ep.} & \cell{40th ep.} & \cell{20th ep.} & \cell{40th ep.}\\\hline
\multirow{1}{*}{\text{~~~Algorithms}} & \multicolumn{8}{|c||}{\textbf{rcv1\_train.binary ($\boldsymbol{n=20,242,p=47,236}$)}}\\\hline
\cell{\textbf{ProxHSGD-SL}}  & \cellbest{9.319e-03} & \cellbest{1.229e-04} & \cellbest{1.701e-03} & \cellbest{9.506e-04} & \cellbest{0.944} & \cellbest{0.949} & \cellbest{0.937} & \cellbest{0.943}\\\hline
\cell{\textbf{ProxHSGD-RS1}}  & \cellbetter{1.294e-02} & \cellbetter{2.546e-03} & \cellbetter{1.999e-03} & \cellbetter{1.107e-03} & \cellbest{0.944} & \cellbetter{0.948} & \cellbetter{0.936} & \cellbetter{0.941}\\\hline
\cell{ProxSpiderBoost}  & \cell{2.593e-02} & \cell{1.072e-02} & \cell{3.146e-03} & \cell{1.793e-03} & \cell{0.940} & \cell{0.944} & \cell{0.931} & \cell{0.936}\\\hline
\cell{ProxSVRG}  & \cell{2.927e-02} & \cell{1.232e-02} & \cell{3.445e-03} & \cell{1.934e-03} & \cell{0.939} & \cell{0.944} & \cell{0.929} & \cell{0.934}\\\hline
\cell{ProxSGD2}  & \cell{4.545e-02} & \cell{2.508e-02} & \cell{6.103e-03} & \cell{4.481e-03} & \cell{0.935} & \cell{0.940} & \cell{0.926} & \cell{0.930}\\\hline 
\multirow{1}{*}{\text{~~~Algorithms}} & \multicolumn{8}{|c||}{\textbf{\textbf{real-sim ($\boldsymbol{n=72,309,p=20,958}$)}}}\\\hline
\cell{\textbf{ProxHSGD-SL}}  & \cellbest{8.850e-03} & \cellbest{1.103e-04} & \cellbest{1.339e-03} & \cellbest{7.503e-04} & \cellbest{0.968} & \cellbest{0.971} & \cell{0.665} & \cell{0.650}\\\hline
\cell{\textbf{ProxHSGD-RS1}}  & \cellbetter{1.294e-02} & \cellbetter{2.787e-03} & \cellbetter{1.626e-03} & \cellbetter{9.099e-04} & \cellbetter{0.966} & \cellbetter{0.970} & \cell{0.673} & \cell{0.655}\\\hline
\cell{ProxSpiderBoost}  & \cell{2.003e-02} & \cell{6.730e-03} & \cell{2.156e-03} & \cell{1.179e-03} & \cell{0.963} & \cell{0.969} & \cell{0.686} & \cell{0.662}\\\hline
\cell{ProxSVRG}  & \cell{2.644e-02} & \cell{1.150e-02} & \cell{2.654e-03} & \cell{1.521e-03} & \cell{0.960} & \cell{0.967} & \cellbetter{0.697} & \cellbetter{0.670}\\\hline
\cell{ProxSGD2}  & \cell{4.401e-02} & \cell{2.249e-02} & \cell{4.208e-03} & \cell{2.562e-03} & \cell{0.949} & \cell{0.962} & \cellbest{0.726} & \cellbest{0.690}\\\hline\hline
\multicolumn{9}{|c|}{The loss function $\ell_3$}\\\hline
& \multicolumn{2}{|c|}{Training Loss Residual}& \multicolumn{2}{|c|}{$\norms{G_{\eta}(w_T)}$}& \multicolumn{2}{|c|}{Training Accuracy}& \multicolumn{2}{|c|}{Test Accuracy}\\\cline{2-9}
& \cell{20th ep.} & \cell{40th ep.} & \cell{20th ep.} & \cell{40th ep.} &  \cell{20th ep.} & \cell{40th ep.} & \cell{20th ep.} & \cell{40th ep.}\\\hline
\multirow{1}{*}{\text{~~~Algorithms}} & \multicolumn{8}{|c||}{\textbf{rcv1\_train.binary ($\boldsymbol{n=20,242,p=47,236}$)}}\\\hline
\cell{\textbf{ProxHSGD-SL}}  & \cell{5.112e-02} & \cell{1.915e-02} & \cell{4.465e-03} & \cell{2.663e-03} & \cell{0.934} & \cell{0.940} & \cell{0.924} & \cell{0.929}\\\hline
\cell{\textbf{ProxHSGD-RS1}}  & \cellbest{2.224e-02} & \cellbest{6.492e-06} & \cellbest{2.842e-03} & \cellbest{1.626e-03} & \cellbest{0.939} & \cellbest{0.943} & \cellbest{0.929} & \cellbest{0.936}\\\hline
\cell{ProxSpiderBoost}  & \cellbetter{2.850e-02} & \cellbetter{4.161e-03} & \cellbetter{3.188e-03} & \cellbetter{1.842e-03} & \cellbetter{0.938} & \cellbest{0.943} & \cellbetter{0.928} & \cellbetter{0.935}\\\hline
\cell{ProxSVRG}  & \cell{2.797e-02} & \cell{2.943e-03} & \cell{3.157e-03} & \cell{1.775e-03} & \cell{0.938} & \cell{0.943} & \cellbetter{0.928} & \cellbetter{0.935}\\\hline
\cell{ProxSGD2}  & \cell{1.316e-01} & \cell{9.255e-02} & \cell{9.639e-03} & \cell{7.599e-03} & \cell{0.924} & \cell{0.929} & \cell{0.919} & \cell{0.922}\\\hline 
\multirow{1}{*}{\text{~~~Algorithms}} & \multicolumn{8}{|c||}{\textbf{\textbf{real-sim ($\boldsymbol{n=72,309,p=20,958}$)}}}\\\hline
\cell{\textbf{ProxHSGD-SL}}  & \cellbetter{2.873e-02} & \cellbetter{1.020e-02} & \cellbetter{1.859e-03} & \cellbetter{1.040e-03} & \cellbetter{0.962} & \cellbetter{0.968} & \cell{0.687} & \cell{0.662}\\\hline
\cell{\textbf{ProxHSGD-RS1}}  & \cellbest{1.196e-02} & \cellbest{9.608e-07} & \cellbest{1.114e-03} & \cellbest{6.256e-04} & \cellbest{0.968} & \cellbest{0.970} & \cell{0.663} & \cell{0.647}\\\hline
\cell{ProxSpiderBoost}  & \cell{3.606e-02} & \cell{1.442e-02} & \cell{2.208e-03} & \cell{1.218e-03} & \cell{0.960} & \cell{0.967} & \cell{0.692} & \cell{0.666}\\\hline
\cell{ProxSVRG}  & \cell{4.091e-02} & \cell{1.862e-02} & \cell{2.438e-03} & \cell{1.403e-03} & \cell{0.958} & \cell{0.966} & \cellbetter{0.698} & \cellbetter{0.672}\\\hline
\cell{ProxSGD2}  & \cell{1.003e-01} & \cell{5.911e-02} & \cell{5.591e-03} & \cell{3.493e-03} & \cell{0.930} & \cell{0.952} & \cellbest{0.763} & \cellbest{0.719}\\\hline\hline
\multicolumn{9}{|c|}{The loss function $\ell_4$}\\\hline
& \multicolumn{2}{|c|}{Training Loss Residual}& \multicolumn{2}{|c|}{$\norms{G_{\eta}(w_T)}$}& \multicolumn{2}{|c|}{Training Accuracy}& \multicolumn{2}{|c|}{Test Accuracy}\\\cline{2-9}
& \cell{20th ep.} & \cell{40th ep.} & \cell{20th ep.} & \cell{40th ep.} &  \cell{20th ep.} & \cell{40th ep.} & \cell{20th ep.} & \cell{40th ep.}\\\hline
\multirow{1}{*}{\text{~~~Algorithms}} & \multicolumn{8}{|c||}{\textbf{rcv1\_train.binary ($\boldsymbol{n=20,242,p=47,236}$)}}\\\hline
\cell{\textbf{ProxHSGD-SL}}  & \cellbest{1.401e-02} & \cellbest{1.819e-04} & \cellbest{4.727e-03} & \cellbest{2.674e-03} & \cellbest{0.959} & \cellbest{0.964} & \cellbest{0.954} & \cellbest{0.958}\\\hline
\cell{\textbf{ProxHSGD-RS1}}  & \cellbetter{2.044e-02} & \cellbetter{4.125e-03} & \cellbetter{5.745e-03} & \cellbetter{3.253e-03} & \cellbetter{0.956} & \cellbetter{0.962} & \cellbetter{0.951} & \cellbetter{0.957}\\\hline
\cell{ProxSpiderBoost}  & \cell{2.422e-01} & \cell{1.322e-01} & \cell{4.045e-02} & \cell{2.489e-02} & \cell{0.926} & \cell{0.935} & \cell{0.920} & \cell{0.927}\\\hline
\cell{ProxSVRG}  & \cell{2.635e-01} & \cell{1.448e-01} & \cell{4.304e-02} & \cell{2.679e-02} & \cell{0.925} & \cell{0.933} & \cell{0.920} & \cell{0.926}\\\hline
\cell{ProxSGD2}  & \cell{3.686e-02} & \cell{1.163e-02} & \cell{9.593e-03} & \cell{6.269e-03} & \cell{0.952} & \cell{0.959} & \cell{0.943} & \cell{0.953}\\\hline 
\multirow{1}{*}{\text{~~~Algorithms}} & \multicolumn{8}{|c||}{\textbf{\textbf{real-sim ($\boldsymbol{n=72,309,p=20,958}$)}}}\\\hline
\cell{\textbf{ProxHSGD-SL}}  & \cellbest{9.009e-03} & \cellbest{1.070e-04} & \cellbest{2.386e-03} & \cellbest{1.279e-03} & \cellbest{0.981} & \cellbest{0.984} & \cell{0.683} & \cell{0.676}\\\hline
\cell{\textbf{ProxHSGD-RS1}}  & \cellbetter{1.106e-02} & \cellbetter{1.936e-03} & \cellbetter{2.450e-03} & \cellbetter{1.444e-03} & \cellbest{0.981} & \cellbest{0.984} & \cell{0.679} & \cell{0.674}\\\hline
\cell{ProxSpiderBoost}  & \cell{1.526e-01} & \cell{8.450e-02} & \cell{2.428e-02} & \cell{1.168e-02} & \cell{0.930} & \cell{0.960} & \cellbetter{0.767} & \cellbetter{0.707}\\\hline
\cell{ProxSVRG}  & \cell{1.883e-01} & \cell{1.076e-01} & \cell{3.079e-02} & \cell{1.558e-02} & \cell{0.915} & \cell{0.950} & \cellbest{0.797} & \cellbest{0.726}\\\hline
\cell{ProxSGD2}  & \cell{1.889e-02} & \cell{6.334e-03} & \cell{3.569e-03} & \cell{2.267e-03} & \cell{0.979} & \cell{0.982} & \cell{0.676} & \cell{0.678}\\\hline\hline
\end{tabular}
}
\end{center}
\vspace{-1.5ex}
\caption{The performance of $5$ different algorithms on two small datasets: \textbf{The mini-batch case}.}\label{tbl:small_datasets}
\end{table}

As \nhan{can be seen} from Table~\ref{tbl:small_datasets}, either \texttt{ProxHSGD-SL} or \texttt{ProxHSGD-RS1} is the best in these two datasets since they produce lower objective value and smaller gradient mapping norm.
\revise{Three competitors \texttt{ProxSpiderBoost}, \texttt{ProxSVRG}, and \texttt{ProxSGD2} are comparable with our methods in many cases, but in some other cases, our methods highly outperform these schemes.
While our methods may \nhan{produce} better objective value and gradient mapping norm, we can observe that due to some overfitting issues, their test accuracy is slower than those of \texttt{ProxSGD2} or \texttt{ProxSVRG}. 
This can be recognized by comparing the training accuracy and the corresponding test accuracy.}

In addition, we run these five algorithms on the two larger datasets: \texttt{news20.binary} and \texttt{url\_combined}.
The results are reported in Table~\ref{tbl:large_datasets}.

\begin{table}[hpt!]
\newcommand{\cellbest}[1]{{\!\!}{\color{blue}\textbf{#1}}{\!\!}}
\newcommand{\cellbetter}[1]{{\!\!}{\color{blue}#1}{\!\!}}
\newcommand{\cellr}[1]{{}{\color{red}#1}{}}
\newcommand{\cell}[1]{{\!\!}#1{\!\!}}
\begin{center}
\resizebox{\textwidth}{!}{
\begin{tabular}{|l||c|c||c|c||c|c||c|c||}
\hline
 \multicolumn{9}{|c|}{The loss function $\ell_1$}\\\hline
 & \multicolumn{2}{|c|}{Training Loss Residual}& \multicolumn{2}{|c|}{$\norms{G_{\eta}(w_T)}$}& \multicolumn{2}{|c|}{Training Accuracy}& \multicolumn{2}{|c|}{Test Accuracy}\\\cline{2-9}
 & \cell{20th ep.} & \cell{40th ep.} & \cell{20th ep.} & \cell{40th ep.} &  \cell{20th ep.} & \cell{40th ep.} & \cell{20th ep.} & \cell{40th ep.}\\\hline
\multirow{1}{*}{\text{~~~Algorithms}} & \multicolumn{8}{|c||}{\textbf{news20.binary} ($\boldsymbol{n=19,996,p=1,355,191}$)}\\\hline
\cell{\textbf{ProxHSGD-SL}}  & \cellbest{6.614e-02} & \cellbest{1.990e-05} & \cellbest{7.203e-03} & \cellbest{4.471e-03} & \cellbest{0.865} & \cellbest{0.892} & \cellbest{0.687} & \cellbetter{0.698}\\\hline
\cell{\textbf{ProxHSGD-RS1}}  & \cellbetter{9.289e-02} & \cellbetter{1.764e-02} & \cellbetter{8.470e-03} & \cellbetter{5.173e-03} & \cellbetter{0.850} & \cellbetter{0.881} & \cellbetter{0.659} & \cellbest{0.705}\\\hline
\cell{ProxSpiderBoost}  & \cell{2.932e-01} & \cell{1.653e-01} & \cell{1.558e-02} & \cell{1.259e-02} & \cell{0.626} & \cell{0.822} & \cell{0.003} & \cell{0.544}\\\hline
\cell{ProxSVRG}  & \cell{3.073e-01} & \cell{1.816e-01} & \cell{1.301e-02} & \cell{1.356e-02} & \cell{0.625} & \cell{0.812} & \cell{0.001} & \cell{0.500}\\\hline
\cell{ProxSGD2}  & \cell{2.733e-01} & \cell{1.426e-01} & \cell{2.212e-02} & \cell{1.847e-02} & \cell{0.649} & \cell{0.832} & \cell{0.015} & \cell{0.585}\\\hline
\multirow{1}{*}{\text{~~~Algorithms}} & \multicolumn{8}{|c||}{\textbf{url\_combined} ($\boldsymbol{n=2,396,130,p= 3,231,961}$)}\\ \hline
\cell{\textbf{ProxHSGD-SL}}  & \cellbest{5.934e-03} & \cellbest{9.244e-05} & \cellbest{6.157e-04} & \cellbest{3.985e-04} & \cellbest{0.968} & \cellbest{0.970} & \cellbest{0.969} & \cellbetter{0.971}\\\hline
\cell{\textbf{ProxHSGD-RS1}}  & \cellbetter{7.405e-03} & \cellbetter{1.279e-03} & \cellbetter{6.615e-04} & \cellbetter{4.257e-04} & \cellbetter{0.968} & \cellbest{0.970} & \cellbetter{0.968} & \cellbetter{0.970}\\\hline
\cell{ProxSpiderBoost}  & \cell{2.467e-02} & \cell{1.396e-02} & \cell{1.789e-03} & \cell{1.015e-03} & \cell{0.964} & \cell{0.966} & \cell{0.964} & \cell{0.966}\\\hline
\cell{ProxSVRG}  & \cell{4.581e-02} & \cell{2.690e-02} & \cell{3.824e-03} & \cell{1.992e-03} & \cell{0.962} & \cell{0.964} & \cell{0.963} & \cell{0.964}\\\hline
\cell{ProxSGD2}  & \cell{2.013e-02} & \cell{1.101e-02} & \cell{1.478e-03} & \cell{8.988e-04} & \cell{0.965} & \cell{0.967} & \cell{0.965} & \cell{0.967}\\\hline\hline
 \multicolumn{9}{|c|}{The loss function $\ell_2$}\\\hline
 & \multicolumn{2}{|c|}{Training Loss Residual}& \multicolumn{2}{|c|}{$\norms{G_{\eta}(w_T)}$}& \multicolumn{2}{|c|}{Training Accuracy}& \multicolumn{2}{|c|}{Test Accuracy}\\\cline{2-9}
 & \cell{20th ep.} & \cell{40th ep.} & \cell{20th ep.} & \cell{40th ep.} &  \cell{20th ep.} & \cell{40th ep.} & \cell{20th ep.} & \cell{40th ep.}\\\hline
\multirow{1}{*}{\text{~~~Algorithms}} & \multicolumn{8}{|c||}{\textbf{news20.binary} ($\boldsymbol{n=19,996,p=1,355,191}$)}\\\hline
\cell{\textbf{ProxHSGD-SL}}  & \cell{3.636e-02} & \cell{2.190e-02} & \cell{2.977e-03} & \cell{2.006e-03} & \cell{0.791} & \cell{0.826} & \cell{0.451} & \cell{0.604}\\\hline
\cell{\textbf{ProxHSGD-RS1}}  & \cellbest{1.055e-02} & \cellbest{1.831e-06} & \cellbest{1.367e-03} & \cellbest{8.706e-04} & \cellbest{0.844} & \cellbest{0.864} & \cellbest{0.638} & \cellbest{0.633}\\\hline
\cell{ProxSpiderBoost}  & \cellbetter{3.476e-02} & \cellbetter{2.042e-02} & \cellbetter{2.849e-03} & \cellbetter{1.899e-03} & \cellbetter{0.797} & \cellbetter{0.829} & \cellbetter{0.481} & \cellbetter{0.615}\\\hline
\cell{ProxSVRG}  & \cell{3.726e-02} & \cell{2.203e-02} & \cell{3.025e-03} & \cell{1.997e-03} & \cell{0.788} & \cell{0.826} & \cell{0.445} & \cell{0.606}\\\hline
\cell{ProxSGD2}  & \cell{6.471e-02} & \cell{5.196e-02} & \cell{7.270e-03} & \cell{6.595e-03} & \cell{0.625} & \cell{0.691} & \cell{0.003} & \cell{0.112}\\\hline 
\multirow{1}{*}{\text{~~~Algorithms}} & \multicolumn{8}{|c||}{\textbf{url\_combined} ($\boldsymbol{n=2,396,130,p= 3,231,961}$)}\\ \hline
\cell{\textbf{ProxHSGD-SL}}  & \cellbetter{6.055e-03} & \cellbetter{3.834e-03} & \cellbetter{3.613e-04} & \cellbetter{2.308e-04} & \cellbetter{0.966} & \cellbetter{0.969} & \cellbetter{0.966} & \cellbetter{0.969}\\\hline
\cell{\textbf{ProxHSGD-RS1}}  & \cellbest{2.241e-03} & \cellbest{3.597e-08} & \cellbest{1.584e-04} & \cellbest{1.220e-04} & \cellbest{0.971} & \cellbest{0.973} & \cellbest{0.971} & \cellbest{0.973}\\\hline
\cell{ProxSpiderBoost}  & \cell{6.305e-03} & \cell{3.949e-03} & \cell{3.643e-04} & \cell{2.210e-04} & \cellbetter{0.966} & \cellbetter{0.969} & \cellbetter{0.966} & \cellbetter{0.969}\\\hline
\cell{ProxSVRG}  & \cell{1.058e-02} & \cell{6.839e-03} & \cell{7.481e-04} & \cell{4.049e-04} & \cell{0.964} & \cell{0.965} & \cell{0.964} & \cell{0.966}\\\hline
\cell{ProxSGD2}  & \cell{1.449e-02} & \cell{9.151e-03} & \cell{1.173e-03} & \cell{6.126e-04} & \cell{0.962} & \cell{0.964} & \cell{0.963} & \cell{0.964}\\\hline\hline
 \multicolumn{9}{|c|}{The loss function $\ell_3$}\\\hline
 & \multicolumn{2}{|c|}{Training Loss Residual}& \multicolumn{2}{|c|}{$\norms{G_{\eta}(w_T)}$}& \multicolumn{2}{|c|}{Training Accuracy}& \multicolumn{2}{|c|}{Test Accuracy}\\\cline{2-9}
 & \cell{20th ep.} & \cell{40th ep.} & \cell{20th ep.} & \cell{40th ep.} &  \cell{20th ep.} & \cell{40th ep.} & \cell{20th ep.} & \cell{40th ep.}\\\hline
\multirow{1}{*}{\text{~~~Algorithms}} & \multicolumn{8}{|c||}{\textbf{news20.binary} ($\boldsymbol{n=19,996,p=1,355,191}$)}\\\hline
\cell{\textbf{ProxHSGD-SL}}  & \cell{4.182e-02} & \cell{1.890e-02} & \cell{3.482e-03} & \cell{2.518e-03} & \cell{0.691} & \cell{0.787} & \cell{0.115} & \cell{0.451}\\\hline
\cell{\textbf{ProxHSGD-RS1}}  & \cellbest{2.163e-02} & \cellbest{7.406e-06} & \cellbest{2.636e-03} & \cellbest{1.737e-03} & \cellbest{0.782} & \cellbest{0.819} & \cellbest{0.435} & \cellbest{0.586}\\\hline
\cell{ProxSpiderBoost}  & \cellbetter{2.663e-02} & \cellbetter{4.377e-03} & \cellbetter{2.846e-03} & \cellbetter{1.902e-03} & \cellbetter{0.765} & \cellbetter{0.814} & \cellbetter{0.365} & \cellbetter{0.564}\\\hline
\cell{ProxSVRG}  & \cell{3.048e-02} & \cell{6.860e-03} & \cell{3.012e-03} & \cell{2.002e-03} & \cell{0.750} & \cell{0.810} & \cell{0.314} & \cell{0.549}\\\hline
\cell{ProxSGD2}  & \cell{7.544e-02} & \cell{6.249e-02} & \cell{6.868e-03} & \cell{6.169e-03} & \cell{0.623} & \cell{0.625} & \cell{0.001} & \cell{0.003}\\\hline 
\multirow{1}{*}{\text{~~~Algorithms}} & \multicolumn{8}{|c||}{\textbf{url\_combined} ($\boldsymbol{n=2,396,130,p= 3,231,961}$)}\\ \hline
\cell{\textbf{ProxHSGD-SL}}  & \cell{1.048e-02} & \cell{5.507e-03} & \cell{5.294e-04} & \cell{2.958e-04} & \cell{0.964} & \cell{0.966} & \cellbetter{0.965} & \cell{0.966}\\\hline
\cell{\textbf{ProxHSGD-RS1}}  & \cellbest{2.601e-03} & \cellbest{3.495e-08} & \cellbest{1.926e-04} & \cellbest{1.193e-04} & \cellbest{0.968} & \cellbest{0.970} & \cellbest{0.969} & \cellbest{0.970}\\\hline
\cell{ProxSpiderBoost}  & \cellbetter{7.921e-03} & \cellbetter{3.722e-03} & \cellbetter{4.022e-04} & \cellbetter{2.293e-04} & \cellbetter{0.965} & \cellbetter{0.967} & \cellbetter{0.965} & \cellbetter{0.968}\\\hline
\cell{ProxSVRG}  & \cell{1.615e-02} & \cell{8.911e-03} & \cell{8.411e-04} & \cell{4.496e-04} & \cell{0.963} & \cell{0.965} & \cell{0.964} & \cell{0.965}\\\hline
\cell{ProxSGD2}  & \cell{3.364e-02} & \cell{1.886e-02} & \cell{1.945e-03} & \cell{1.005e-03} & \cell{0.962} & \cell{0.963} & \cell{0.962} & \cell{0.964}\\\hline\hline
\multicolumn{9}{|c|}{The loss function $\ell_4$}\\\hline
& \multicolumn{2}{|c|}{Training Loss Residual}& \multicolumn{2}{|c|}{$\norms{G_{\eta}(w_T)}$}& \multicolumn{2}{|c|}{Training Accuracy}& \multicolumn{2}{|c|}{Test Accuracy}\\\cline{2-9}
& \cell{20th ep.} & \cell{40th ep.} & \cell{20th ep.} & \cell{40th ep.} &  \cell{20th ep.} & \cell{40th ep.} & \cell{20th ep.} & \cell{40th ep.}\\\hline
\multirow{1}{*}{\text{~~~Algorithms}} & \multicolumn{8}{|c||}{\textbf{news20.binary} ($\boldsymbol{n=19,996,p=1,355,191}$)}\\\hline
\cell{\textbf{ProxHSGD-SL}}  & \cellbest{4.252e-02} & \cellbest{6.483e-04} & \cellbest{7.475e-03} & \cellbest{4.661e-03} & \cellbest{0.883} & \cellbest{0.914} & \cellbest{0.694} & \cellbest{0.669}\\\hline
\cell{\textbf{ProxHSGD-RS1}}  & \cellbetter{6.343e-02} & \cellbetter{1.603e-02} & \cellbetter{9.342e-03} & \cellbetter{5.527e-03} & \cellbetter{0.865} & \cellbetter{0.901} & \cellbetter{0.676} & \cellbetter{0.675}\\\hline
\cell{ProxSpiderBoost}  & \cell{2.381e-01} & \cell{2.017e-01} & \cell{1.962e-02} & \cell{1.812e-02} & \cell{0.624} & \cell{0.627} & \cell{0.001} & \cell{0.004}\\\hline
\cell{ProxSVRG}  & \cell{2.426e-01} & \cell{2.072e-01} & \cell{2.005e-02} & \cell{1.828e-02} & \cell{0.624} & \cell{0.626} & \cell{0.001} & \cell{0.003}\\\hline
\cell{ProxSGD2}  & \cell{1.007e-01} & \cell{3.930e-02} & \cell{2.068e-02} & \cell{1.704e-02} & \cell{0.843} & \cell{0.884} & \cell{0.573} & \cell{0.688}\\\hline 
\multirow{1}{*}{\text{~~~Algorithms}} & \multicolumn{8}{|c||}{\textbf{url\_combined} ($\boldsymbol{n=2,396,130,p= 3,231,961}$)}\\ \hline
\cell{\textbf{ProxHSGD-SL}}  & \cellbest{4.997e-03} & \cellbest{7.821e-05} & \cellbetter{8.616e-04} & \cellbest{4.849e-04} & \cellbest{0.953} & \cellbest{0.956} & \cellbetter{0.951} & \cellbest{0.955}\\\hline
\cell{\textbf{ProxHSGD-RS1}}  & \cellbetter{5.779e-03} & \cellbetter{1.309e-03} & \cellbest{8.034e-04} & \cellbetter{5.427e-04} & \cellbetter{0.951} & \cellbetter{0.954} & \cell{0.950} & \cellbetter{0.953}\\\hline
\cell{ProxSpiderBoost}  & \cell{2.159e-02} & \cell{1.574e-02} & \cell{3.250e-03} & \cell{1.895e-03} & \cell{0.949} & \cell{0.947} & \cell{0.948} & \cell{0.946}\\\hline
\cell{ProxSVRG}  & \cell{4.104e-02} & \cell{2.319e-02} & \cell{9.529e-03} & \cell{3.680e-03} & \cell{0.954} & \cell{0.949} & \cellbest{0.953} & \cell{0.948}\\\hline
\cell{ProxSGD2}  & \cell{8.181e-03} & \cell{3.501e-03} & \cell{1.555e-03} & \cell{7.355e-04} & \cell{0.950} & \cell{0.952} & \cell{0.949} & \cell{0.951}\\\hline\hline
\end{tabular}
}
\caption{The performance of $5$ different algorithms on two large datasets: \textbf{The mini-batch case}.}\label{tbl:large_datasets}
\end{center}
\end{table}

Again, we observe the same performance \nhan{among} these methods.
Either \texttt{ProxHSGD-SL} or \texttt{ProxHSGD-RS1} works best for every case.
\revise{Three other competiors: \texttt{ProxSGD}, \texttt{ProxSVRG}, and \texttt{ProxSpiderBoost} work well and are relatively comparable with our methods in some cases, but they are still slower than our methods overall.}

\beforesubsec
\subsection{\bf Feedforward neural network training problem}
\aftersubsec
In the last example, we consider the following composite nonconvex optimization problem obtained from a fully connected feedforward neural network training task:
\myeq{eq:fwnn_exam}{
\min_{x\in\R^p}\set{ F(x) := \frac{1}{n}\sum_{i=1}^n\ell\big( F_L(x, a_i), b_i\big) +  \psi(x)},
}
where we concatenate all the weight matrices and bias vectors of the neural network in one vector of variable $x$, $\set{(a_i, b_i)}_{i=1}^n$ is a training dataset, $F_L(\cdot)$ is a composition between all linear transforms and activation functions as $F_L(x, a) := \bsigma_L(W_L\bsigma_{L-1}(W_{L-1}\bsigma_{L-2}(\cdots \bsigma_0(W_0a + \mu_0) \cdots ) + \mu_{L-1}) + \mu_L)$, where $W_i$ is a weight matrix, $\mu_i$ is a bias vector, $\bsigma_i$ is an activation function, $L$ is the number of layers,  $\ell(\cdot)$ is a cross-entropy loss, and $\psi(x) := \lambda\norms{x}_1$ is the $\ell_1$-norm regularizer  for some $\lambda > 0$ to obtain sparse weights.
By defining $f_i(x) := \ell(F_L(x, a_i), b_i)$ for $i=1,\cdots, n$, we can formulate \eqref{eq:fwnn_exam} into the  composite finite-sum setting \eqref{eq:finite_sum}.

We implement mini-batch variants of Algorithm~\ref{alg:A1} and Algorithm~\ref{alg:A2}, and compare with two other methods: \texttt{ProxSVRG} and \texttt{ProxSpiderBoost} in \href{https://www.tensorflow.org}{TensorFlow} using the well-known dataset \texttt{mnist} to evaluate their performance.

In the first experiment, we use an \nhan{one-hidden-layer-fully-connected} neural network: $784 \times 128 \times 10$, while in the second test, we increase the number of neurons in the hidden layer to obtain another \nhan{fully-connected} neural network: $784 \times 800 \times 10$.
The activation function $\bsigma_i$ for the hidden layer is ReLU, and for the output layer is soft-max.

Our experiment configuration is as follows.
We choose $\lambda := \frac{1}{n}$ to obtain sparse weights.
We set $\gamma = 0.95$ for all methods and tune $\eta$ to obtain the best results.
Here, we obtain $\eta = 0.3$ for \texttt{ProxHSGD-SL} and \texttt{ProxHSGD-RS1}. 
We also tune $\eta$ in \texttt{ProxSpiderBoost} and \texttt{ProxSVRG} to obtain the best results. 
We finally get $\eta = 0.12$ for \texttt{ProxSpiderBoost} and $\eta = 0.2$ for  \texttt{ProxSVRG}.
We set $\hat{b} := 100$ for our algorithms, $\hat{b} := \lfloor\sqrt{n}\rfloor$ for \texttt{ProxSpiderBoost}, and $\hat{b} := \lfloor n^{2/3}\rfloor$ for \texttt{ProxSVRG} and set the epoch length $m := \lfloor\frac{n}{\hat{b}}\rfloor$.
The performance of the four algorithms running on the first network is reported in Figure~\ref{fig:neural_net1}.

\begin{figure}[hpt!]
\begin{center}
\includegraphics[width = 1\textwidth]{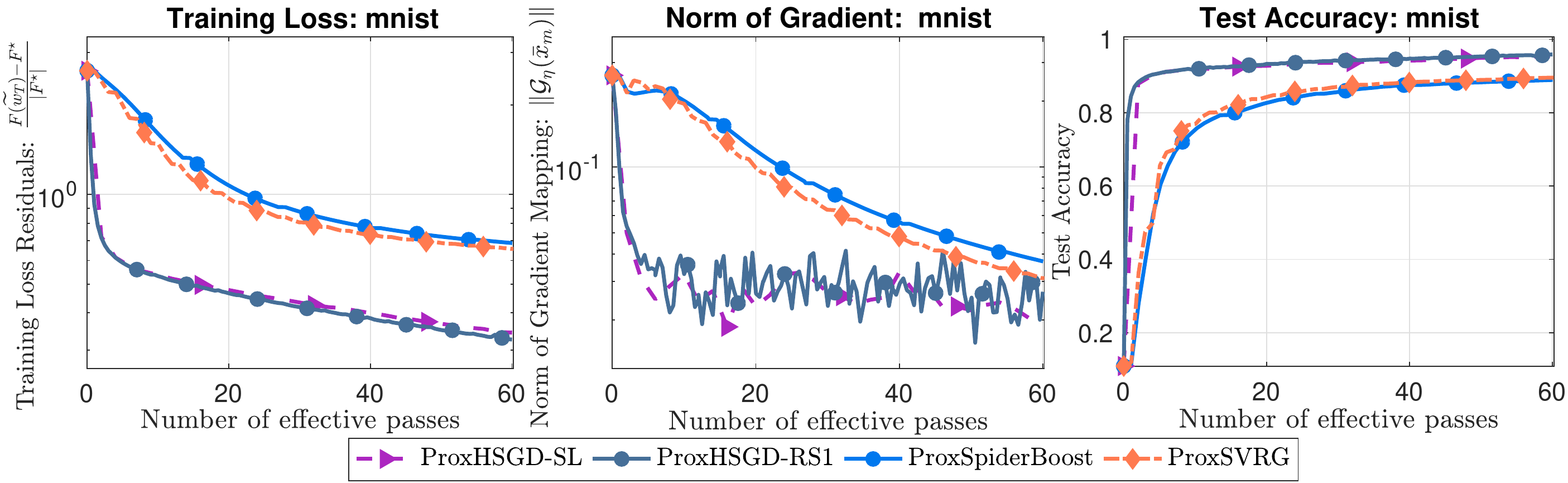}
\vspace{-2ex}
\caption{The performance (training loss, norm of gradient mapping, and test accuracy) of $4$ algorithms on the \texttt{mnist} dataset for solving \eqref{eq:fwnn_exam}: A fully-connected $784\times128\times 10$ neural network.}\label{fig:neural_net1}
\end{center}
\vspace{-1ex}
\end{figure}

From Figure~\ref{fig:neural_net1}, both \texttt{ProxHSGD-SL} and \texttt{ProxHSGD-RS1} work relatively well in this example, and outperform two other methods.
On one hand, our methods achieve better training loss values, norms of gradient mapping, and test accuracy than both \texttt{ProxSpiderBoost} and \texttt{ProxSVRG}.
On the other hand, the restarting variant \texttt{ProxHSGD-RS1} appears to be slightly better than \texttt{ProxHSGD-SL}.

\begin{figure}[hpt!]
\vspace{-2ex}
\begin{center}
\includegraphics[width = 1\textwidth]{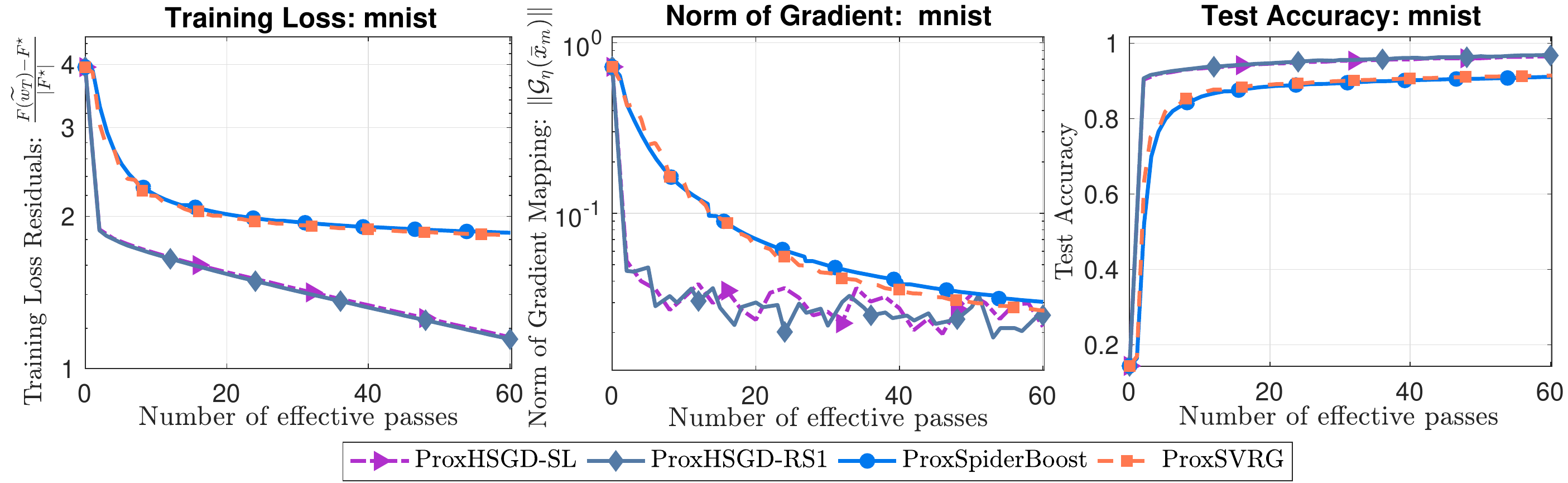}
\vspace{-2ex}
\caption{The performance (training loss, norm of gradient mapping, and test accuracy)  of $4$ algorithms on the \texttt{mnist} dataset for solving \eqref{eq:fwnn_exam}: A fully-connected $784\times800\times 10$ neural network.}\label{fig:neural_net2}
\end{center}
\vspace{-5ex}
\end{figure}

Besides, the results when running these algorithms on the second neural network are given in Figure~\ref{fig:neural_net2}.
We can observe the same behavior of these four algorithms as in Figure~\ref{fig:neural_net1}, but  \texttt{ProxHSGD-RS1} \nhan{does not exhibit clear advantage over} \texttt{ProxHSGD-SL}.

\appendix
\beforesec
\section{Appendix: Properties of Hybrid Stochastic Estimators}\label{sec:appendix1}
\aftersec
This appendix provides the full proof of our theoretical results in Section~\ref{sec:stochastic_estimators}.
However, we also need the following lemma in the sequel.
Hence, we prove it here.

\begin{lemma}\label{le:adaptive_step_size}
Given $L > 0$, $\delta > 0$, $\epsilon > 0$, and $\omega \in (0, 1)$, let $\sets{\gamma_t}_{t=0}^m$ be the sequence updated by
\myeq{eq:update_of_eta_t_proof}{
\gamma_m := \frac{\delta}{L}~~~~\text{and}~~~~\gamma_t := \frac{\delta}{L + \epsilon L^2\big[\omega\gamma_{t+1} + \omega^2\gamma_{t+2} + \cdots + \omega^{(m-t)}\gamma_m\big]},
}
for $t=0,\cdots, m-1$.
Then  
\myeq{eq:stepsize_pros}{
0 < \gamma_0 < \gamma_1 < \cdots < \gamma_m = \frac{\delta}{L}~~~\text{and}~~~\Sigma_m := \sum_{t=0}^m\gamma_t \geq \frac{\delta(m+1)\sqrt{1-\omega}}{L\left[\sqrt{1-\omega} + \sqrt{1 - \omega + 4\delta\omega\epsilon }\right]}.
}
\end{lemma}

\begin{proof}
First, from \eqref{eq:update_of_eta_t_proof} it is obvious to show that $0 < \gamma_0 < \cdots < \gamma_{m-1} = \frac{\delta}{L(1+\epsilon\omega)} < \gamma_m = \frac{\delta}{L}$. 
At the same time, since $\omega \in (0, 1)$, we have $1 \geq \omega \geq \omega^2 \geq \cdots \geq \omega^{m}$.
By Chebyshev's sum inequality, we have
\myeq{eq:estimate1}{
\begin{array}{ll}
(m-t)\big(\omega\gamma_{t+1} + \omega^2\gamma_{t+2} + \cdots + \omega^{m-t}\gamma_m\big) &\leq \big(\sum_{j=t+1}^m\gamma_i\big)\left(\omega + \omega^2 + \cdots + \omega^{m-t}\right) \vspace{1ex}\\
&\leq \frac{\omega}{1-\omega}\big(\sum_{j=t+1}^m\gamma_i\big).
\end{array}
}
From the update \eqref{eq:update_of_eta_t_proof}, we also have
\myeq{eq:estimate2}{
\left\{\begin{array}{ll}
\epsilon L^2\gamma_0(\omega\gamma_1 + \omega^2\gamma_2 + \cdots + \omega^{m}\gamma_m) &= \delta - L\gamma_0 \vspace{1ex}\\
\epsilon  L^2\gamma_1(\omega\gamma_2 + \omega^2\gamma_3 + \cdots + \omega^{m-1}\gamma_{m}) &= \delta - L\gamma_1 \vspace{1ex}\\
\cdots & \cdots \vspace{1ex}\\
\epsilon  L^2\gamma_{m-1}\omega\gamma_m &= \delta - L\gamma_{m-1} \vspace{1ex}\\
0  &= \delta - L\gamma_{m}.
\end{array}\right.
}
Substituting \eqref{eq:estimate1} into \eqref{eq:estimate2}, we get
\myeqn{
\left\{\begin{array}{ll}
\frac{\omega \epsilon L^2}{1-\omega}\gamma_0(\gamma_0 + \gamma_1   + \cdots + \gamma_m) &\geq \delta m  - mL\gamma_0 + \frac{\omega \epsilon L^2}{1-\omega}\gamma_0^2 \vspace{1ex}\\
\frac{\omega \epsilon L^2}{1-\omega}\gamma_1(\gamma_0 + \gamma_1  + \cdots + \gamma_{m}) &\geq \delta(m-1) - (m-1)L\gamma_1 + \frac{\omega \epsilon L^2}{1-\omega}(\gamma_1\gamma_0 + \gamma_1^2) \vspace{1ex}\\
\cdots & \cdots \vspace{1ex}\\
\frac{\omega \epsilon L^2}{1-\omega}\gamma_{m-1}(\gamma_0 + \gamma_1 + \cdots + \gamma_m) &\geq \delta - L\gamma_{m-1} + \frac{\omega \epsilon L^2}{1-\omega}(\gamma_{m-1}\gamma_0 + \cdots + \gamma_{m-1}^2) \vspace{1ex}\\
\frac{\omega \epsilon L^2}{1-\omega}\gamma_{m}(\gamma_0 + \gamma_1 + \cdots + \gamma_m) &\geq \delta - L\gamma_{m} + \frac{\omega \epsilon L^2}{1-\omega}(\gamma_{m}\gamma_0 + \cdots + \gamma_{m}^2).
\end{array}\right.
}
Let us define $\Sigma_m := \sum_{t=0}^m\gamma_t$ and $S_m := \sum_{t=0}^m\gamma_t^2$.
Summing up both sides of the above inequalities, we get 
\myeqn{
\frac{\omega \epsilon L^2}{1-\omega}\Sigma_m^2 \geq  \frac{\delta(m^2 + m + 2)}{2} - L(m\gamma_0 + (m-1)\gamma_1 + \cdots + \gamma_{m-1} + \gamma_m) + \frac{\omega \epsilon L^2}{2(1-\omega)}\big(S_m + \Sigma_m^2\big).
}
Using again Chebyshev's sum inequality, we have
\myeqn{
m\gamma_0 + (m-1)\gamma_1 + \cdots + \gamma_{m-1} + \gamma_m \leq \frac{m^2 + m + 2}{2(m+1)}\left(\sum_{t=0}^m\gamma_t\right) = \frac{(m^2 + m + 2)\Sigma_m}{2(m+1)}.
}
Note that $(m+1)S_m \geq \Sigma_m^2$ by Cauchy-Schwarz's inequality, which shows that $S_m + \Sigma_m^2 \geq \big(\frac{m+2}{m+1}\big)\Sigma_m^2$.
Combining three last inequalities, we obtain the following quadratic inequation in $\Sigma_m > 0$:
\myeqn{
\frac{m\omega \epsilon L^2}{(1-\omega)}\Sigma_m^2 + L(m^2 + m + 2)\Sigma_m - \delta(m+1)(m^2 + m + 2) \geq 0.
}
Solving this inequation with respect to $\Sigma_m > 0$, we obtain
\myeqn{
\begin{array}{ll}
\Sigma_m &\geq \frac{(1-\omega)\big[\sqrt{(m^2 + m + 2)^2 + \frac{4m(m+1)(m^2 + m + 2)\omega\epsilon \delta}{1-\omega}} - (m^2 + m + 2)\big]}{2\epsilon \omega m L} \vspace{1ex}\\
&= \frac{2\delta(m+1)}{L\left[1 + \sqrt{1 + \frac{4m(m+1)\omega\delta\epsilon }{(1-\omega)(m^2 + m+2)}}\right]} \vspace{1ex}\\
&\geq \frac{2\delta(m+1)\sqrt{1-\omega}}{L\left[\sqrt{1-\omega} + \sqrt{1 - \omega + 4\delta\omega\epsilon }\right]} ~~~~~\text{since}~~\frac{m(m+1)}{m^2+m+2} < 1.
\end{array}
}
This proves \eqref{eq:stepsize_pros}.
\end{proof}

\beforesubsec
\subsection{\bf The proof of Lemma~\ref{le:key_pro_of_vhat_t}: Variance estimate with mini-batch}\label{apdx:le:key_pro_of_vhat_t}
\aftersubsec
The proof of the first expression of \eqref{eq:key_pro_of_vhat_t} is the same as in Lemma~\ref{le:key_estimate10}.
We only prove the second one. 
Let $\Delta_{\Bc_t} := \frac{1}{b_t}\sum_{i\in\Bc_t}\left[ G_{\xi_i}(x_t) - G_{\xi_i}(x_{t-1})\right]$, $\Delta_t := G(x_t) - G(x_{t-1})$, $\hat{\delta}_t := \hat{v}_t - G(x_t)$, and $\delta{u}_t := u_t -  G(x_t)$.
Clearly, we have 
\myeqn{
\Exps{\Bc_t}{\Delta_{\Bc_t}} = \Delta_t~~~\text{and}~~~\Exps{\hat{\Bc}_t}{\delta{u}_t} = 0.
}
Moreover, we can rewrite $\hat{v}_t$ as
\myeqn{
\hat{\delta}_t  = \beta_{t-1}\hat{\delta}_{t-1} + \beta_{t-1}\Delta_{\Bc_t} + (1-\beta_{t-1})\delta{u}_t - \beta_{t-1}\Delta_t.
}
Therefore, using these two expressions, we can derive
\myeqn{
\begin{array}{lcl}
\Exps{(\Bc_t,\hat{\Bc}_t)}{\norms{\hat{\delta}_t}^2}  &=& \beta_{t-1}^2\norms{\hat{\delta}_{t-1}}^2 +  \beta_{t-1}^2\Exps{\Bc_t}{\norms{\Delta_{\Bc_t}}^2} + (1-\beta_{t-1})^2\Exps{\hat{\Bc}_t}{\norms{\delta{u}_t}^2}  + \beta_{t-1}^2\norms{\Delta_t}^2\vspace{1ex}\\
&& + {~} 2\beta_{t-1}^2\iprods{\hat{\delta}_{t-1},\Exps{\Bc_t}{\Delta_{\Bc_t}}} + 2\beta_{t-1}(1-\beta_{t-1})\iprods{\hat{\delta}_{t-1}, \Exps{\hat{\Bc}_t}{\delta{u}_t}} - 2\beta_{t-1}^2\iprods{\hat{\delta}_{t-1},\Delta_t} \vspace{1ex}\\
&& + {~} 2\beta_{t-1}(1-\beta_{t-1})\Exps{(\Bc_t,\hat{\Bc}_t)}{\iprods{\Delta_{\Bc_t}, \delta{u}_t}} - 2\beta_{t-1}^2\iprods{\Exps{\Bc_t}{\Delta_{\Bc_t}}, \Delta_t} \vspace{1ex}\\
&& - {~} 2\beta_{t-1}(1-\beta_{t-1})\iprods{\Exps{\hat{\Bc}_t}{\delta{u}_t}, \Delta_t} \vspace{1ex}\\
&=&  \beta_{t-1}^2\norms{\hat{\delta}_{t-1}}^2 +  \beta_{t-1}^2\Exps{\Bc_t}{\norms{\Delta_{\Bc_t}}^2} + (1-\beta_{t-1})^2\Exps{\hat{\Bc}_t}{\norms{\delta{u}_t}^2}  - \beta_{t-1}^2\norms{\Delta_t}^2.
\end{array}
}
Similar to the proof of \cite[Lemma 2]{Pham2019}, for the finite-sum case (i.e., $\vert\Omega\vert = n$), we can show that
\myeqn{
\Exps{\Bc_t}{\norms{\Delta_{\Bc_t}}^2} = \frac{n(b_t-1)}{(n-1)b_t}\norms{\Delta_t}^2 + \frac{(n-b_t)}{(n-1)b_t}\Exps{\xi}{\norms{G_{\xi}(x_t) - G_{\xi}(x_{t-1})}^2}. 
}
For the expectation case, we have
\myeqn{
\Exps{\Bc_t}{\norms{\Delta_{\Bc_t}}^2} =  \left(1- \frac{1}{b_t}\right)\norms{\Delta_t}^2 + \frac{1}{b_t}\Exps{\xi}{\norms{G_{\xi}(x_t) - G_{\xi}(x_{t-1})}^2}. 
}
Using the definition of $\rho$ in Lemma~\ref{le:upper_bound_new_batch}, we can unify these two expressions as
\myeqn{
\Exps{\Bc_t}{\norms{\Delta_{\Bc_t}}^2} =  \left(1- \rho\right)\norms{\Delta_t}^2 + \rho\Exps{\xi}{\norms{G_{\xi}(x_t) - G_{\xi}(x_{t-1})}^2}. 
}
 Substituting the last expression into the previous one, we obtain the second expression of \eqref{eq:key_pro_of_vhat_t}.
\Eproof

\beforesubsec
\subsection{\bf The proof of Lemma~\ref{le:upper_bound_new_batch}: Upper bound of  mini-batch variance}\label{apdx:le:upper_bound_new_batch}
\aftersubsec
From Lemma~\ref{le:key_pro_of_vhat_t}, taking the expectation with respect to $\Fc_{t+1} := \sigma(x_0,\Bc_0, \hat{\Bc}_0, \cdots, \Bc_t, \hat{\Bc}_t)$, we have
\myeqn{
\begin{array}{lcl}
\Exp{\norms{\hat{v}_t - G(x_t)}^2} &\leq&  \beta_{t-1}^2\Exp{\norms{\hat{v}_{t-1}  - G(x_{t-1})}^2} \vspace{1ex}\\
&& + {~} \rho L^2\beta_{t-1}^2 \Exp{\norms{x_t - x_{t-1}}^2} + (1-\beta_{t-1})^2\Exps{\hat{\Bc}_t}{\norms{u_t - G(x_t)}^2}.
\end{array}
}
In addition, from \cite[Lemma 2]{Pham2019}, we have $\Exps{\hat{\Bc}_t}{\norms{u_t - G(x_t)}^2} \leq \hat{\rho}\Exps{\xi}{\norms{G_{\xi}(x_t) - G(x_t)}^2} = \hat{\rho}\sigma_t^2$, where $\sigma_t^2 := \Exps{\xi}{\norms{G_{\xi}(x_t) - G(x_t)}^2}$.

Let $A_t^2 := \Exp{\norms{\hat{v}_t - G(x_t)}^2} $ and $B_t^2 := \Exp{\norms{x_{t +1}- x_{t}}^2}$.
Then, the above estimate can be upper bounded as follows:
\myeqn{
A_t^2 \leq \beta_{t-1}^2A_{t-1}^2 + \rho L^2\beta_{t-1}^2B_{t-1}^2 + \hat{\rho}(1-\beta_{t-1})^2\sigma_t^2.
}
By following inductive step as in the proof of Lemma \ref{le:upper_bound_new}, we obtain from the last inequality that
\myeqn{
\begin{array}{lcl}
A_t^2 &\leq& \left(\beta_{t-1}^2\cdots \beta_0^2 \right)A_0^2 + \rho L^2 \left(\beta_{t-1}^2\cdots \beta_0^2\right)B_0^2 + \cdots + \rho L^2 \beta_{t-1}^2B_{t-1}^2 \vspace{1ex}\\
&&+ {~} \hat{\rho}\left[\left(\beta_{t-1}^2\cdots\beta_1^2\right)(1-\beta_0)^2\sigma_1^2   + \cdots + (1-\beta_{t-1})^2\sigma_{t}^2  \right].
\end{array}
}
Using the definition of $\omega_{t}$, $\omega_{i, t}$, and $S_{t}$ from \eqref{eq:other_quantities}, the previous inequality becomes
\myeqn{
A_t^2 \leq \omega_{t}A_0^2 + \rho L^2 \sum_{i=0}^{t-1}\omega_{i,t}B_i^2 + \hat{\rho} S_t,
}
which is the same as \eqref{eq:vt_variance_bound_new_batch} by substituting the definition of $A_t$ and $B_t$ above into it.
\Eproof

\beforesec
\section{The Proof of Technical Results in Section~\ref{sec:sgd_algs}: The Single Sample Case}\label{sec:apdx:sgd_algs}
\aftersec
We provide the full proof of technical results in Section~\ref{sec:sgd_algs}.

\beforesubsec
\subsection{\bf The proof of Lemma~\ref{le:upper_bound_new}: Key estimate}\label{apdx:le:upper_bound_new}
\aftersubsec
From the update $x_{t+1} := (1-\gamma_t)x_t + \gamma_t\widehat{x}_{t+1}$ at Step~\ref{step:i3} of Algorithm~\ref{alg:A1}, we have $x_{t+1} - x_t = \gamma_t(\widehat{x}_{t+1} - x_t)$. From the $L$-average smoothness condition in Assumption~\ref{as:A1}, one can write
\myeq{eq:lm3_est1}{
\begin{array}{ll}
f(x_{t+1}) &\leq f(x_t) + \iprods{\nabla f(x_t), x_{t+1} - x_t} + \frac{L}{2}\norms{x_{t+1} - x_t}^2\\
&= f(x_{t}) + \gamma_t \iprods{\nabla f(x_t),\widehat{x}_{t+1} - x_t} + \frac{L\gamma_t^2}{2}\norms{\widehat{x}_{t+1} - x_t}^2.
\end{array}
}
Using convexity of $\psi$, we can show that
\myeq{eq:lm3_est2}{
\psi(x_{t+1}) \leq (1-\gamma_t)\psi(x_t) + \gamma_t \psi(\widehat{x}_{t+1}) \leq \psi (x_t) + \gamma_t \iprods{\nabla \psi(\widehat{x}_{t+1}), \widehat{x}_{t+1} - x_t},
}
where $\nabla \psi(\widehat{x}_{t+1}) \in \partial \psi(\widehat{x}_{t+1})$ is any subgradient of $\psi$ at $\widehat{x}_{t+1}$.

Utilizing the optimality condition of $\widehat{x}_{t+1} = \prox_{\eta_t \psi}(x_{t} - \eta_t v_t)$, we can show that $\nabla \psi(\widehat{x}_{t+1}) = -v_t - \frac{1}{\eta_t}(\widehat{x}_{t+1} - x_t)$ for some $\nabla \psi(\widehat{x}_{t+1}) \in \partial \psi(\widehat{x}_{t+1})$. 
Substituting this relation into \eqref{eq:lm3_est2}, we get
\myeq{eq:lm3_est3}{
\psi(x_{t+1}) \leq \psi (x_t) - \gamma_t \left\langle v_t, \widehat{x}_{t+1} - x_t\right\rangle - \frac{\gamma_t}{\eta_t}\|\widehat{x}_{t+1} - x_t\|^2.
}
Combining \eqref{eq:lm3_est1} and \eqref{eq:lm3_est3}, and using $F(x) := f(x) + \psi(x)$ from \eqref{eq:ncvx_prob}, we obtain
\myeq{eq:lm3_est4}{
F(x_{t+1}) \leq F (x_t) + \gamma_t \left \langle\nabla f(x_t) -v_t, \widehat{x}_{t+1} - x_t\right\rangle -\left( \frac{\gamma_t}{\eta_t} - \frac{L\gamma_t^2}{2}\right)\|\widehat{x}_{t+1} - x_t\|^2.
}
For any $c_t > 0$, we can always write
\myeqn{
\begin{array}{ll}
\iprods{\nabla f(x_t) -v_t, \widehat{x}_{t+1} - x_t} &= \frac{1}{2c_t}\norms{ \nabla f(x_t) -v_t }^2 + \frac{c_t}{2} \norms{ \widehat{x}_{t+1} - x_t }^2 \vspace{1ex}\\
& - \frac{1}{2c_t}\norms{ \nabla f(x_t) -v_t - c_t( \widehat{x}_{t+1} - x_t) }^2.
\end{array}
}
Utilizing this expression, we can rewrite \label{eq:lm3_est4} as
\myeqn{ 
F(x_{t+1}) \leq F (x_t) + \frac{\gamma_t}{2c_t}\| \nabla{f}(x_t -v_t \|^2 -\left( \frac{\gamma_t}{\eta_t} - \frac{L\gamma_t^2}{2} - \frac{\gamma_tc_t}{2}\right)\norms{\widehat{x}_{t+1} - x_t}^2 - \frac{\tilde{\sigma}_t^2}{2}.
}
where $\tilde{\sigma}_t^2 := \frac{\gamma_t}{c_t}\norms{ \nabla{f}(x_t -v_t - c_t( \widehat{x}_{t+1} - x_t) }^2 \geq 0$.

\noindent
Taking expectation both sides of this inequality over the entire history $\Fc_{t+1}$, we obtain
\myeq{eq:lm3_est5}{
\begin{array}{ll}
\Exp{F(x_{t+1})} &\leq \Exp{F(x_t)} + \frac{\gamma_t}{2c_t}\Exp{\Vert\nabla{f}(x_t) - v_t\Vert^2} \vspace{1ex}\\
& - {~} \Big(\frac{\gamma_t}{\eta_t} - \frac{L\gamma_t^2}{2}- \frac{\gamma_tc_t}{2}\Big)\Exp{\Vert\widehat{x}_{t+1} - x_t\Vert^2} - \frac{1}{2}\Exp{\tilde{\sigma}_t^2}.
\end{array}
}
Next, from the definition of gradient mapping $\Grad_{\eta}(x):= \frac{1}{\eta}\left(x - \prox_{\eta \psi}(x - \eta\nabla f(x))\right)$ in \eqref{eq:grad_map}, we can see that
\myeqn{
\eta_{t}\norms{\Grad_{\eta_t}(x_t)} = \norms{x_t - \prox_{\eta_t \psi}\left(x_t - \eta_t \nabla f(x_t)\right)}.
}
Using this expression, the triangle inequality, and the nonexpansive property $\norms{\prox_{\eta\psi}(z) - \prox_{\eta\psi}(w)} \leq \norms{ z - w}$ of $\prox_{\eta\psi}$, we can derive that
\myeqn{
\begin{array}{ll}
\eta_t\norms{\Grad_{\eta_t}(x_t)} &\leq \norms{\widehat{x}_{t+1} - x_t} + \norms{ \mathrm{prox}_{\eta_t\psi}(x_t - \eta_t\nabla{f}(x_t)) - \widehat{x}_{t+1}} \vspace{1ex}\\
& = \norms{\widehat{x}_{t+1} - x_t} + \norms{\prox_{\eta_t\psi}(x_t - \eta_t\nabla{f}(x_t)) - \prox_{\eta_t\psi}(x_t - \eta_tv_t)} \vspace{1ex}\\
& \leq \norms{\widehat{x}_{t+1} - x_t} + \eta_t\norms{ \nabla{f}(x_t) - v_t}.
\end{array}
}
Now, for any $r_t > 0$, the last estimate leads to
\myeqn{ 
\eta_t^2\Exp{\norms{\Grad_{\eta_t}(x_t)}^2} \leq \left(1+\tfrac{1}{r_t}\right)\Exp{\norms{\widehat{x}_{t+1} - x_t}^2} +  (1+r_t)\eta_t^2\Exp{\norms{ \nabla{f}(x_t) - v_t}^2}.
}
Multiplying this inequality by $\frac{q_t}{2} > 0$ and adding the result to \eqref{eq:lm3_est5}, we finally get
\myeqn{ 
\begin{array}{ll}
\Exp{F(x_{t+1})} &\leq \Exp{F(x_t)} - \frac{q_t\eta_t^2}{2}\Exp{\norms{ \Grad_{\eta_t}(x_t)}^2} \vspace{1ex}\\
& +  {~} \frac{1}{2}\Big[\frac{\gamma_t}{c_t} +  (1+r_t)q_t\eta_t^2\Big] \Exp{\norms{\nabla{f}(x_t) - v_t}^2} \vspace{1ex}\\
&- {~} \frac{1}{2}\Big[\frac{2\gamma_t}{\eta_t} - L\gamma_t^2 -  \gamma_tc_t - q_t\left(1+\frac{1}{r_t}\right)\Big] \Exp{\norms{ \widehat{x}_{t+1} - x_t }^2}  - \frac{1}{2}\Exp{\tilde{\sigma}_t^2}.
\end{array}
}
Using the definition of $\theta_t$ and $\kappa_t$ from \eqref{eq:theta_kappa}, i.e.,:
\myeqn{
\theta_t:= \frac{\gamma_t}{c_t} + (1+r_t)q_t\eta_t^2~~~~\text{and}~~~~\kappa_t := \frac{2\gamma_t}{\eta_t} - L\gamma_t^2 - \gamma_tc_t - q_t\left(1 + \frac{1}{r_t}\right),
}
we can simplify this estimate as follows:
\myeq{eq:lm52_est10}{
\begin{array}{ll}
\Exp{F(x_{t+1})} &\leq \Exp{F(x_t)} - \frac{q_t\eta_t^2}{2}\Exp{\norms{ \Grad_{\eta_t}(x_t)}^2}  +  \frac{\theta_t}{2} \Exp{\norms{\nabla{f}(x_t) - v_t}^2}  \vspace{1ex}\\
& - {~}  \frac{\kappa_t}{2}\Exp{\norms{ \widehat{x}_{t+1} - x_t }^2}  - \frac{1}{2}\Exp{\tilde{\sigma}_t^2}.
\end{array}
}
This is exactly \eqref{eq:upper_bound_new0}.
\Eproof

\beforesubsec
\subsection{\bf The proof of Lemma~\ref{le:descent_pro}: Key estimate of Lyapunov function}\label{apdx:le:descent_pro}
\aftersubsec
From \eqref{eq:key_estimate10}, by taking the full expectation on the history $\Fc_{t+1}$ and using the $L$-average smoothness of $f$, we can show that
\myeqn{
{\!\!\!\!\!\!\!}\begin{array}{ll}
\Exp{\norms{v_{t+1} - \nabla{f}(x_{t+1})}^2} {\!\!\!}&\leq \beta_t^2\Exp{\norms{v_t - \nabla{f}(x_t)}^2} + \beta_t^2L^2\Exp{\norms{x_{t+1} - x_t}^2} + (1-\beta_t)^2\sigma_{t+1}^2 \vspace{1ex}\\
&= \beta_t^2\Exp{\norms{v_t - \nabla{f}(x_t)}^2} + \beta_t^2\gamma_t^2L^2\Exp{\norms{\widehat{x}_{t+1} - x_t}^2} + (1-\beta_t)^2\sigma_{t+1}^2,
\end{array}{\!\!\!\!\!\!\!}
}
where $\sigma_t^2 := \Exp{\norms{\nabla{f}_{\zeta_t}(x_t) - \nabla{f}(x_t)}^2}$.

Let $V$ be the Lyapunov  function defined by \eqref{eq:Lyapunov_func}.
Then, by multiplying the last inequality by $\frac{\alpha_{t+1}}{2} > 0$, adding the result to \eqref{eq:lm52_est10}, and then using this Lyapunov function we can show that
\myeq{eq:lm52_est12}{
\begin{array}{ll}
V(x_{t+1}) &\leq V(x_t) -  \frac{q_t\eta_t^2}{2}\Exp{\norms{ \Grad_{\eta_t}(x_t)}^2}  - \frac{1}{2}(\alpha_t - \beta_t^2\alpha_{t+1} - \theta_t)\Exp{\norms{v_t - \nabla{f}(x_t)}^2} \vspace{1.25ex}\\
&- {~} \frac{1}{2}(\kappa_t - \alpha_{t+1}\beta_t^2\gamma_t^2L^2)\Exp{\norms{\widehat{x}_{t+1} - x_t}^2} +  \frac{1}{2}(1-\beta_t)^2\alpha_{t+1}\sigma_{t+1}^2 - \frac{1}{2}\Exp{\tilde{\sigma}_t^2}.
\end{array}
}
Let us choose $\gamma_t$, $\eta_t$, and other parameters such that the conditions \eqref{eq:key_cond1001} hold, i.e.:
\myeqn{
\alpha_t - \beta_t^2\alpha_{t+1} - \theta_t \geq 0 ~~~\text{and}~~\kappa_t - \alpha_{t+1}\beta_t^2\gamma_t^2L^2 \geq 0.
}
In this case, \eqref{eq:lm52_est12} can be simplified as follows:
\myeqn{
V(x_{t+1}) \leq V(x_t) -  \frac{q_t\eta_t^2}{2}\Exp{\norms{ \Grad_{\eta_t}(x_t)}^2}  +  \frac{1}{2}\alpha_{t+1}(1-\beta_t)^2\sigma_{t+1}^2.
}
This proves \eqref{eq:key_bound1001a}.

Finally, summing up this inequality from $t := 0$ to $t := m$, we obtain
\myeqn{
\sum_{t=0}^m\frac{q_t\eta_t^2}{2}\Exp{\norms{ \Grad_{\eta_t}(x_t)}^2}  \leq V(x_0) - V(x_{m+1}) + \frac{1}{2}\sum_{t=0}^m\alpha_{t+1}(1-\beta_t)^2\sigma_{t+1}^2.
}
Note that $V(x_{m+1}) := \Exp{F(x_{m+1})} + \frac{\alpha_{m+1}}{2}\Exp{\norms{v_{m+1} - \nabla{f}(x_{m+1})}^2} \geq \Exp{F(x_{m+1})} \geq F^{\star}$ by Assumption~\ref{as:A0} and $V(x_0) = F(x_0) + \frac{\alpha_0}{2}\Exp{\norms{v_0 - \nabla{f}(x_0)}^2}$.
Using these estimates into the last inequality, we obtain the key estimate \eqref{eq:key_bound1001}.
\Eproof

\beforesubsec
\subsection{\bf The proof of Theorem~\ref{th:singe_loop_adapt_step}: The adaptive step-size case}\label{apdx:th:singe_loop_adapt_step}
\aftersubsec
Let $\sets{(x_t, \hat{x}_{t})}$ be generated by Algorithm \ref{alg:A1}.
Let us again choose  $c_t := L$, $r_t := 1$ and $q_t := \frac{L\gamma_t}{2}$ and fix $\eta_t := \eta  \in (0, \frac{1}{L})$ in Lemma~\ref{le:upper_bound_new} as done in Theorem~\ref{th:constant_stepsize_convergence}.
Then, from \eqref{eq:theta_kappa}, we have
\myeqn{ 
\theta_t := \frac{(1 + L^2\eta^2)\gamma_t}{L}~~~~~\text{and}~~~~~\kappa_t := \left(\frac{2}{\eta} - L\gamma_t - 2L\right)\gamma_t.
}
Using these parameters into  \eqref{eq:upper_bound_new0} and summing up the result from $t := 0$ to $t := m$, and then using \eqref{eq:vt_variance_bound_new} from Lemma~\ref{le:upper_bound_new}, we obtain 
\myeqn{ 
\begin{array}{lcl}
\Exp{F(x_{m+1})} &\leq & \Exp{F(x_0)} +\dfrac{L^2}{2}\displaystyle\sum_{t=0}^m \theta_t\sum_{i=0}^{t-1}\gamma_i^2\omega_{i,t}\Exp{\norms{\widehat{x}_{i+1} - x_{i}}^2}  \vspace{1ex}\\
&& - {~} \dfrac{1}{2}\displaystyle\sum_{t=0}^m\kappa_t \Exp{\Vert\widehat{x}_{t+1} - x_t\Vert^2}  - \dfrac{\eta^2L}{4}\displaystyle\sum_{t=0}^m \gamma_t\Exp{\Vert \Grad_{\eta}(x_t)\Vert^2} \vspace{1ex}\\
&& - {~} \displaystyle\sum_{t=0}^m\Exp{ \tilde{\sigma}_t} + \dfrac{1}{2}\displaystyle\sum_{t=0}^m\theta_t \omega_t \bar{\sigma}^2  +  \dfrac{1}{2}\displaystyle\sum_{t=0}^m \theta_t S_t,
\end{array}
}
where $\bar{\sigma}^2 := \Exp{\norms{v_0 - \nabla f(x_0)}^2} \ge 0$, $\tilde{\sigma}_t^2 := \dfrac{\gamma_t}{2}\norms{\nabla f(x_t) - v_t - L(\hat{x}_{t+1} - x_t)}^2 \ge 0$, and $\omega_{i,t}$, $\omega_t$, and $S_t$ are defined in Lemma \ref{le:upper_bound_new}. 

\noindent 
By ignoring the nonnegative term $\Exp{ \tilde{\sigma}_t^2}$, and using the expression of $\theta_t$ and $\kappa_t$ above, we can further estimate the last inequality as follows:
\myeq{eq:upper_bound_new}{
\begin{array}{lcl}
\Exp{F(x_{m+1})} &\leq & \Exp{F(x_0)} -  \frac{\eta^2L}{4}\displaystyle\sum_{t=0}^m\gamma_t\Exp{\Vert \Grad_{\eta}(x_t)\Vert^2} \vspace{1ex}\\
&& + ~\frac{(1+L^2\eta^2)\bar{\sigma}^2}{2L}{\displaystyle\sum_{t=0}^m\omega_t\gamma_t} + \frac{(1+L^2\eta^2)}{2L}{\displaystyle\sum_{t=0}^m\gamma_tS_t} + \frac{\Tc_m}{2},
\end{array}
}
where $\Tc_m$ is defined as follows:
\myeq{eq:T_m}{
{\!\!\!}\Tc_m := L(1+L^2\eta^2)\sum_{t=0}^m\gamma_t\sum_{i=0}^{t-1}\omega_{i,t}\gamma_i^2\Exp{\norms{\widehat{x}_{i+1} - x_i}^2} - \sum_{t=0}^m\gamma_t\left(\frac{2}{\eta} - 2L - L\gamma_t\right)\Exp{\norms{\widehat{x}_{i+1} - x_i}^2}.{\!\!\!}
}
Now, with the choice of $\beta_t = \beta := 1- \frac{1}{\sqrt{\tilde{b}(m+1)}} \in (0, 1)$, we can easily show that $\omega_t = \beta^{2t}$, $\omega_{i,t} = \beta^{2(t-i)}$, and  $s_t := \big(\prod_{j=i+2}^{t}\beta_{j-1}^2\big)(1-\beta_i)^2 =  (1-\beta)^2\Big[\frac{1-\beta^{2t}}{1-\beta^2}\Big] < \frac{1-\beta}{1+\beta}$ due to Lemma \ref{le:upper_bound_new}. 

Let $w_i^2 := \Exp{\norms{\widehat{x}_{i+1} - x_i}^2}$.
To bound the quantity $\Tc_m$ defined by \eqref{eq:T_m}, we note that
\myeqn{
\begin{array}{lcl}
\displaystyle\sum_{t=1}^m\gamma_t\sum_{i=0}^{t-1}\beta^{2(t-i)}\gamma_i^2w_i^2 &=& \beta^2\gamma_0^2\big[\gamma_1 + \beta^2\gamma_2 + \cdots + \beta^{2(m-1)}\gamma_m\big]w_0^2 \vspace{1ex}\\
&& + {~} \beta^2\gamma_1^2\big[\gamma_2 + \beta^2\gamma_3 + \cdots + \beta^{2(m-2)}\gamma_{m}\big]w_1^2 + \cdots \vspace{1.5ex}\\ 
&& + {~} \beta^2\gamma_{m-2}^2\big[\gamma_{m-1} + \beta^2\gamma_{m}\big]w_{m-2}^2 + \beta^2\gamma_{m-1}^2\gamma_m w_{m-1}^2. 
\end{array}
}
Using $\delta := \frac{2}{\eta} - 2L$, we can write $\Tc_m$ from  \eqref{eq:T_m} as
\myeqn{ 
\begin{array}{lcl}
\Tc_m &=&  \gamma_0\Big[ L(1+L^2\eta^2) \beta^2\gamma_0\big[\gamma_1 + \beta^2\gamma_2 + \cdots + \beta^{2(m-1)}\gamma_m\big]  - (\delta - L\gamma_0)\Big]w_0^2 \vspace{1ex}\\
&& + {~}  \gamma_1\Big[ L(1+L^2\eta^2) \beta^2\gamma_1\big[\gamma_2 + \beta^2\gamma_3 + \cdots + \beta^{2(m-2)}\gamma_{m}\big] - (\delta - L\gamma_1)\Big]w_1^2 + \cdots \vspace{1ex}\\
&& + {~}  \gamma_{m-1}\Big[ L(1+L^2\eta^2) \beta^2\gamma_{m-1}\gamma_m - (1 - L\gamma_{m-1})\Big]w_{m-1}^2 - \gamma_m(\delta - L\gamma_m)w_{m}^2.
\end{array}
}
To guarantee $\Tc_m \leq 0$, from the last expression of $\Tc_m$, we can impose the following condition:
\myeq{eq:cond_of_eta3}{
\left\{\begin{array}{lcl}
L(1+L^2\eta^2) \beta^2\gamma_0\big[\gamma_1 + \beta^2\gamma_2 + \cdots + \beta^{2(m-1)}\gamma_m\big]  - (\delta  - L\eta_0) &=& 0\vspace{1ex}\\
L(1+L^2\eta^2) \beta^2\gamma_1\big[\gamma_2 + \beta^2\gamma_3 + \cdots + \beta^{2(m-2)}\gamma_{m}\big] - (\delta  - L\eta_1)&=& 0\vspace{1ex}\\
\cdots  \cdots &\cdots &\vspace{1ex}\\
L(1+L^2\eta^2) \beta^2\gamma_{m-1}\gamma_m - (\delta  - L\gamma_{m-1}) &=&  0\vspace{1ex}\\
- (1 - L\eta_{m}) &= & 0.
\end{array}\right.
}
It is obvious to show that the condition \eqref{eq:cond_of_eta3} leads to the following update of $\gamma_t$:
\myeqn{ 
\gamma_m := \frac{\delta}{L}~~~\text{and}~~~\gamma_t := \frac{\delta}{L + L(1+L^2\eta^2)\big[\beta^2\gamma_{t+1} + \beta^4\gamma_{t+2} + \cdots + \beta^{2(m-t)}\gamma_m\big]},~~~t=0,\cdots, m-1,
}
which is exactly \eqref{eq:update_of_eta_t}.

\vspace{1ex}
\noindent{(a)}~
Since $\beta = 1 - \frac{1}{[\tilde{b}(m+1)]^{1/2}}$, we have 
\myeqn{
\frac{1}{[\tilde{b}(m+1)]^{1/2}} = 1 - \beta \leq 1 -  \beta^2 \leq \frac{2}{[\tilde{b}(m+1)]^{1/2}}.
}
Moreover, since $\eta \in (0, \frac{1}{L})$, with $\epsilon := \frac{1+L^2\eta^2}{L}$, $\delta := \frac{2}{\eta}-2L$, and $\omega := \beta^2 \in (0,1)$, using the last inequalities, we can easily show that 
\myeq{eq:deno_est1}{
\sqrt{1-\omega} + \sqrt{1 - \omega + 4\delta\omega\epsilon } = \sqrt{1-\beta^2} + \sqrt{1 - \beta^2 + \frac{4\delta \beta^2(1+L^2\eta^2)}{L} } \leq 2\sqrt{2}\left( \frac{1}{[\tilde{b}(m+1)]^{1/4}} + \sqrt{\frac{\delta}{L}}\right).
}
Using \eqref{eq:deno_est1}, $\sqrt{1-\omega} = \sqrt{1-\beta^2} \geq \frac{1}{(\tilde{b}(m+1))^{1/4}}$, and $\epsilon = \frac{1+L^2\eta^2}{L}$ into \eqref{eq:stepsize_pros} of Lemma~\ref{le:adaptive_step_size}, we can derive 
\myeq{eq:lower_bound_of_Sigma}{
\Sigma_m := \sum_{t=0}^m\gamma_t \geq  \frac{\delta(m+1)}{2\sqrt{2}\left(L + \sqrt{L\delta}[\tilde{b}(m+1)]^{1/4}\right)}.
}
Next, since $\omega_t = \beta^{2t}$, by Chebyshev's sum inequality, we have
\myeqn{
\sum_{t=0}^m\omega_t\gamma_t = \sum_{t=0}^m\beta^{2t}\gamma_t \leq \frac{\Sigma_m}{(m+1)}(1 + \beta^2 + \cdots + \beta^{2m}) \leq  \frac{\Sigma_m}{(m+1)(1-\beta^2)}.
}
Utilizing this estimate, $\bar{\sigma}^2 := \Exp{\norms{v_0 - \nabla{f}(x_0)}^2}  \leq \frac{\sigma^2}{\tilde{b}}$, and $S_t \leq \sigma^2 s_t \leq \frac{(1-\beta)\sigma^2}{1+\beta}$  into \eqref{eq:upper_bound_new}, and noting that $\Tc_m \leq 0$, we can further upper bound it as
\myeqn{ 
\displaystyle\frac{\eta^2L}{4}\displaystyle\sum_{t=0}^m\gamma_t\Exp{\norms{\Grad_{\eta}(x_t)}^2} \leq F(x_0) -  \Exp{F(x_{m+1})} + \frac{(1+L^2\eta^2)\sigma^2\Sigma_m}{2L(1-\beta^2)(m+1)\tilde{b}} +  \frac{(1+L^2\eta^2)(1-\beta)\sigma^2\Sigma_m}{2L(1+\beta)}.
}
By  Assumption~\ref{as:A0}, we have $\Exp{F(x_{m+1})} \geq F^{\star}$.
Substituting  this bound into the last estimate and then multiplying the result by $\frac{4}{L\eta^2\Sigma_m}$ we obtain
\myeqn{
\displaystyle\frac{1}{\Sigma_m}\displaystyle\sum_{t=0}^m\gamma_t\Exp{\norms{\Grad_{\eta}(x_t)}^2} \leq \frac{4}{L\eta^2\Sigma_m}[F(x_0) -  F^{\star}]  + \frac{2\sigma^2(1+L^2\eta^2)}{L^2\eta^2(1+\beta)}\left[\frac{1}{\tilde{b}(m+1)(1-\beta)} +  (1-\beta)\right].
}
Since  $\beta = 1- \frac{1}{\tilde{b}^{1/2}(m+1)^{1/2}}$, we have $\frac{1}{\tilde{b}(m+1)(1-\beta)} +  (1-\beta) = \frac{2}{\tilde{b}^{1/2}(m+1)^{1/2}}$.
Utilizing this expression,  \eqref{eq:lower_bound_of_Sigma}, $1+\eta^2L^2 \leq 2$, and $\beta\in [0, 1]$, we can further upper bound the last estimate as
\myeq{eq:key_estimate_301}{
\displaystyle\frac{1}{\Sigma_m}\displaystyle\sum_{t=0}^m\gamma_t\Exp{\norms{\Grad_{\eta}(x_t)}^2} \leq \frac{8\sqrt{2}\big(L + \sqrt{\delta L}[\tilde{b}(m+1)]^{1/4}\big)}{L\eta^2\delta(m+1)}\left[F(x_0) -  F^{\star}\right]  +  \frac{8\sigma^2}{L^2\eta^2[\tilde{b}(m+1)]^{1/2}}.
}
In addition, due to the choice of $\overline{x}_m \sim\Unip{\pb}{\sets{x_t}_{t=0}^m}$, we have $\Exp{\norms{\Grad_{\eta}(\overline{x}_m)}^2} = \displaystyle\frac{1}{\Sigma_m}\displaystyle\sum_{t=0}^m\gamma_t\Exp{\norms{\Grad_{\eta}(x_t)}^2}$.
Combining this expression and \eqref{eq:key_estimate_301}, we obtain \eqref{eq:adaptive_key_est}. 

\vspace{1ex}
\noindent{(b)}~
Let us choose $\tilde{b} := c_1^2(m+1)^{1/3}$ for some constant $c_1 > 0$.
Since $\beta = 1 - \frac{1}{[\tilde{b}(m+1)]^{1/2}}$, to guarantee $\beta \geq 0$, we need to impose $c_1 \geq \frac{1}{(m+1)^{2/3}}$.
With this choice of $\tilde{b}$, \eqref{eq:adaptive_key_est} reduces to
\myeqn{
\Exp{\norms{\Grad_{\eta}(\overline{x}_m)}^2} \leq \frac{8}{L^2\eta^2(m+1)^{2/3}}\left[\frac{\sqrt{2}L\big( L + \sqrt{c_1L\delta}\big)}{\delta}\big[F(x_0) - F^{\star}\big] + \frac{\sigma^2}{c_1}\right].
}
Let us denote by $\Delta_0 := \frac{8}{L^2\eta^2}\left[\frac{\sqrt{2}L( L + \sqrt{c_1L\delta})}{\delta}\big[F(x_0) - F^{\star}\big] + \frac{\sigma^2}{c_1}\right]$.
Then, similar to the proof of Theorem~\ref{th:constant_stepsize_convergence}, we can show that  the number of iterations $m$ is at most $m := \left\lfloor \frac{\Delta_0^{3/2}}{\varepsilon^3} \right\rfloor$, and the total number $\Tc_m$ of  stochastic gradient evaluations $\nabla{f}_{\xi}(x_t)$ is at most $\Tc_m := \left\lfloor \frac{c_1^2\Delta_0^{1/2}}{\varepsilon} + \frac{3\Delta_0^{3/2}}{\varepsilon^3}\right\rfloor$.
\Eproof

\beforesubsec
\subsection{\bf The proof of Theorem~\ref{th:double_loop_convergence}: The restarting variant}\label{apdx:th:double_loop_convergence}
\aftersubsec
\noindent{(a)}~
Since we use the adaptive variant of Algorithm~\ref{alg:A1} as stated in Theorem~\ref{th:singe_loop_adapt_step} for the inner loop of Algorithm~\ref{alg:A2}, from  \eqref{eq:key_estimate_301}, we can see that at each stage $s$, the following estimate holds
\myeq{eq:single_loop2}{
\displaystyle\frac{1}{\Sigma_m}\displaystyle\sum_{t=0}^m\gamma_t\Exp{\norms{\Grad_{\eta}(x_t^{(s)})}^2} \leq  \frac{8\sqrt{2}\tilde{b}^{1/4}\left(L + \sqrt{L\delta}\right)}{L\eta^2\delta(m+1)^{3/4}}\Exp{F(x_0^{(s)}) -  F(x_{m+1}^{(s)})}  +  \frac{8\sigma^2}{L^2\eta^2[\tilde{b}(m+1)]^{1/2}}.
}
where we use the superscript ``$^{(s)}$'' to indicate the stage $s$ in Algorithm~\ref{alg:A2}.
Summing up this inequality from $s := 1$ to $s := S$, and then multiplying the result by $\frac{1}{S}$ and using $\Exp{F(x_{m+1}^{(S)})} \geq F^{\star} > -\infty$, and $\Exp{\norms{\Grad_{\eta}(\overline{x}_T)}^2} = \dfrac{1}{S\Sigma_m}\displaystyle\sum_{s=1}^S\sum_{t=0}^m\Exp{\norms{\Grad_{\eta}(x_t^{(s)})}^2}$, we get \eqref{eq:double_loop_est}, i.e.:
\myeq{eq:est5d_1}{
\begin{array}{lcl}
\Exp{\norms{\Grad_{\eta}(\overline{x}_T)}^2} &= & \dfrac{1}{S\Sigma_m}\displaystyle\sum_{s=1}^S\sum_{t=0}^m\Exp{\norms{\Grad_{\eta}(x_t^{(s)})}^2} \vspace{1ex}\\
& \leq & \dfrac{8\sqrt{2}\tilde{b}^{1/4}\big(L + \sqrt{L\delta}\big)}{L\delta\eta^2S (m+1)^{3/4}} \big[F(\overline{x}^{(0)}) - F^{\star}\big] + \dfrac{8\sigma^2}{L^2\eta^2[\tilde{b}(m+1)]^{1/2}}.
\end{array}
}
\noindent{(b)}~Let $\Delta_F := F(\overline{x}^{(0)}) - F^{\star} > 0$ and choose $\tilde{b} := c_1^2(m + 1)$ for some constant $c_1 > 0$.
Since $\beta = 1 - \frac{1}{[\tilde{b}(m+1)]^{1/2}} \in (0, 1)$, we need to choose $c_1$ such that $c_1 \geq \frac{1}{m+1}$.

Now, for any tolerance $\varepsilon > 0$, to guarantee $\Exp{\norms{\Grad_{\eta}(\overline{x}_T)}^2} \leq \varepsilon^2$, from \eqref{eq:est5d_1}, we require
\myeqn{
\frac{8\sqrt{2}\tilde{b}^{1/4}\big(L + \sqrt{L\delta}\big)\Delta_F}{L\delta\eta^2S (m+1)^{3/4}} + \frac{8\sigma^2}{L^2\eta^2[\tilde{b}(m+1)]^{1/2}} 
= \frac{8\sqrt{2c_1} \big(L + \sqrt{L\delta}\big)\Delta_F}{L\delta\eta^2S (m+1)^{1/2}}  + \frac{8\sigma^2}{L^2\eta^2c_1(m+1)} \leq \varepsilon^2.
}
Let us break this inequality into two parts as
\myeqn{
\frac{8\sqrt{2c_1}\big(L + \sqrt{L\delta}\big)\Delta_F}{L\delta\eta^2S (m+1)^{1/2}} = \frac{\varepsilon^2}{2} ~~~\text{and}~~~ \frac{8\sigma^2}{L^2\eta^2c_1(m+1)} \leq \frac{\varepsilon^2}{2}.
}
Then, we have 
\myeqn{
S = \frac{16\sqrt{2c_1}\big(L + \sqrt{L\delta}\big)\Delta_F}{L\delta\eta^2(m+1)^{1/2}\varepsilon^2} ~~~\text{and}~~~ m+1 \geq \frac{16\sigma^2}{L^2\eta^2c_1\varepsilon^2}.
}
Let us choose $m+1 = \frac{16}{L^2\eta^2c_1}\cdot\frac{\max\set{1,\sigma^2}}{\varepsilon^2}$.
Then, $m + 1 \geq \frac{16}{L^2\eta^2c_1\varepsilon^2}$, and we can set
\myeqn{
S := \frac{16\sqrt{2c_1}\big(L + \sqrt{L\delta}\big)\Delta_F}{L\delta\eta^2\varepsilon^2} \cdot \frac{L\eta \sqrt{c_1}\varepsilon}{4}  = \frac{4\sqrt{2}c_1\big(L + \sqrt{L\delta}\big)\Delta_F}{\delta\eta  \varepsilon}.
}
This leads to \eqref{eq:S_iterations}.
Moreover, we can also show that
\myeqn{
(m+1)S = \frac{16\sqrt{2c_1}(L + \sqrt{L\delta})\Delta_F}{L\delta\eta^2\varepsilon^2}\sqrt{m+1} 
= \frac{64\sqrt{2}(L + \sqrt{L\delta})\Delta_F}{L^2\eta^3\delta} \cdot\frac{\max\set{1,\sigma}}{\varepsilon^3}.
}
Consequently,  the total number of stochastic gradient evaluations $\nabla{f}_{\xi}(x_t)$ is at most
\myeqn{
\begin{array}{lcl}
\Tc_{\nabla{f}} &:=&  \left[ \tilde{b} + 3(m+1) \right]S = (c_1^2 + 3)(m+1)S   = 64\sqrt{2}(c_1^2 + 3)\frac{(L + \sqrt{L\delta})\Delta_F}{L^2\eta^3\delta} \cdot\frac{\max\set{1,\sigma}}{\varepsilon^3} \vspace{1ex}\\
&= & \BigO{\max\set{\sigma, 1} \cdot \frac{\Delta_F}{\varepsilon^3}}. 
\end{array}
}
Since we choose $\tilde{b} := \frac{16c_1}{L^2\eta^2}\cdot\frac{\max\set{1,\sigma^2}}{\varepsilon^2}$, the final complexity is $\BigO{  \frac{\max\set{1,\sigma^2}}{\varepsilon^2} + \frac{\max\set{1,\sigma}}{\varepsilon^3}}$, where other constants independent of $\sigma$ and $\varepsilon$ are hidden.
The total number of proximal operators $\prox_{\eta\psi}$ is at most 
\myeqn{
\Tc_{\prox} := S(m+1) = \frac{64\sqrt{2}(L + \sqrt{L\delta})\Delta_F}{L^2\eta^3\delta} \cdot\frac{\max\set{1,\sigma}}{\varepsilon^3}
=  \BigO{\max\set{\sigma, 1} \cdot \frac{\Delta_F}{\varepsilon^3}}.
}
The estimate \eqref{eq:Toc_double_loop} follows from the bound of $\Tc_{\nabla{f}}$ above and the choice of $\tilde{b}$.
\Eproof

\beforesec
\section{The Proof of Technical Results in Section~\ref{sec:sgd_mini_batch}: The Mini-batch Case}\label{apdx:sec:sgd_mini_batch}
\aftersec
This appendix presents the full proof of the results in Section~\ref{sec:sgd_mini_batch} for the mini-batch case.

\beforesubsec
\subsection{\bf The proof of Theorem~\ref{th:mini_batch_constant_stepsize_convergence}: The single-loop variant}\label{apdx:th:mini_batch_constant_stepsize_convergence}
\aftersubsec
Using \eqref{eq:key_pro_of_vhat_t} from Lemma~\ref{le:key_pro_of_vhat_t} with $G := \nabla{f}$ and taking full expectation and using a constant weight $\beta_t := \beta \in (0, 1)$ and $b_t := b \in \Nbb_{+}$,  we have
\myeqn{ 
\begin{array}{lcl}
\Exp{\norms{\hat{v}_{t+1} -  \nabla{f}(x_{t+1})}^2} &\leq&  \beta^2\Exp{\norms{\hat{v}_{t}  - \nabla{f}(x_{t})}^2}  + \rho\beta^2\Exp{\norms{\nabla{f}_{\xi}(x_{t+1}) - \nabla{f}_{\xi}(x_{t})}^2} \vspace{1ex}\\
&& + {~} (1-\beta)^2\Exp{\norms{u_{t+1} - \nabla{f}(x_{t+1})}^2},
\end{array}
}
where $\rho := \frac{1}{b}$ since we solve \eqref{eq:ncvx_prob}.

Since $\Exp{\norms{\nabla{f}_{\xi}(x_{t+1}) - \nabla{f}_{\xi}(x_t)}^2} \leq L^2\Exp{\norms{x_{t+1} - x_t}^2} \leq L^2\gamma_{t}^2\Exp{\norms{\widehat{x}_{t+1} - x_t}^2}$ by Assumption~\ref{as:A1} and $\Exp{\norms{u_{t+1} - \nabla{f}(x_{t+1})}^2} \leq \frac{\sigma^2}{\hat{b}}$ by Assumption~\ref{as:A1b} and \cite[Lemma 2]{Pham2019}, the last estimate leads to 
\myeq{eq:th61_proof1}{
\Exp{\norms{\hat{v}_{t+1} -  \nabla{f}(x_{t+1})}^2} \leq  \beta^2\Exp{\norms{\hat{v}_t  - \nabla{f}(x_{t})}^2}  + \frac{\beta^2\gamma_t^2L^2}{b}\Exp{\norms{\widehat{x}_{t+1} - x_t}^2}  +  \frac{(1-\beta)^2\sigma^2}{\hat{b}}.
}
Next,  let us choose $\eta_t := \eta > 0$, $\gamma_t := \gamma > 0$, $c_t := L$, $r_t := 1$, and $q_t := \frac{L\gamma}{2} > 0$ in Lemma~\ref{le:upper_bound_new}.
Then, we have $\theta_t = \theta = \frac{(1 + L^2\eta^2)\gamma}{L}  > 0$ and $\kappa_t = \kappa = \left( \frac{2}{\eta} - L\gamma - 2L\right)\gamma > 0$.
Using these values into  \eqref{eq:upper_bound_new0}, we obtain
\myeqn{
\Exp{F(x_{t+1})} \leq  \Exp{F(x_t)} - \dfrac{\gamma\eta^2L}{4}\Exp{\norms{ \Grad_{\eta}(x_t)}^2}  +  \dfrac{\theta}{2} \Exp{\norms{\nabla{f}(x_t) - \hat{v}_t}^2}   -  \dfrac{\kappa}{2}\Exp{\norms{ \widehat{x}_{t+1} - x_t }^2}.
}
Multiplying \eqref{eq:th61_proof1} by $\frac{\alpha}{2}$ for some $\alpha > 0$, and adding the result to the above estimate, we obtain
\myeqn{
\begin{array}{lcl}
\Exp{F(x_{t+1})} + \dfrac{\alpha}{2}\Exp{\norms{\hat{v}_{t+1} -  \nabla{f}(x_{t+1})}^2} &\leq & \Exp{F(x_t)} + \dfrac{(\alpha\beta^2 + \theta)}{2}\Exp{\norms{\hat{v}_{t}  - \nabla{f}(x_{t})}^2}  \vspace{1ex}\\
&&- {~} \dfrac{\gamma\eta^2L}{4}\Exp{\norms{ \Grad_{\eta}(x_t)}^2}  +  \dfrac{\alpha(1-\beta)^2\sigma^2}{2\hat{b}}\vspace{1ex}\\
&& - {~}  \dfrac{1}{2}\left(\kappa - \frac{\alpha\beta^2\gamma^2L^2}{b}\right)\Exp{\norms{x_t - x_{t-1}}^2}.
\end{array}
}
Using the Lyapunov function $V$ defined by \eqref{eq:Lyapunov_func}, the last estimate leads to 
\myeqn{
\begin{array}{lcl}
V(x_{t+1}) &\leq & V(x_t) - \dfrac{\gamma\eta^2L}{4}\Exp{\norms{ \Grad_{\eta}(x_t)}^2}  +  \dfrac{\alpha(1-\beta)^2\sigma^2}{2\hat{b}}\vspace{1ex}\\
&& - {~}  \dfrac{1}{2}\left(\kappa - \dfrac{\alpha\beta^2\gamma^2L^2}{b}\right)\Exp{\norms{x_t - x_{t-1}}^2} - \dfrac{1}{2}\left[\alpha(1-\beta^2) - \theta\right]\Exp{\norms{\hat{v}_{t}  - \nabla{f}(x_{t})}^2}.
\end{array}
}
If we impose the following conditions
\myeq{eq:para_cond100}{
\kappa = \left( \frac{2}{\eta} -  L\gamma - 2L\right)\gamma \geq \frac{\alpha\beta^2\gamma^2L^2}{b} ~~~~\text{and}~~~~\theta = \frac{(1 + L^2\eta^2)}{L}\gamma \leq \alpha(1-\beta^2),
}
then we get from the last inequality that
\myeq{eq:descent_prof100}{
V(x_{t+1}) \leq V(x_t) - \frac{\gamma\eta^2L}{4}\Exp{\norms{ \Grad_{\eta}(x_t)}^2}  +  \frac{\alpha(1-\beta)^2\sigma^2}{2\hat{b}}.
}
The conditions \eqref{eq:para_cond100} can be simplified as
\myeq{eq:para_cond100b}{
\frac{2}{\eta} - 2L - L\gamma \geq \frac{\alpha\gamma \beta^2L^2}{b}~~~~~~\text{and}~~~~~~\frac{(1 + L^2\eta^2)}{L}\gamma \leq \alpha(1 - \beta^2).
}
Moreover, by induction, $V(x_{m+1}) \geq F^{\star}$, and $V(x_0) := F(x_0) + \frac{\alpha}{2}\Exp{\norms{\hat{v}_0 - \nabla{f}(x_0)}^2} \leq F(x_0) + \frac{\alpha\sigma^2}{2\tilde{b}}$, we can further derive from \eqref{eq:descent_prof100} that
\myeq{eq:descent_prof101}{
\frac{1}{m+1}\sum_{t=0}^m\Exp{\norms{ \Grad_{\eta}(x_t)}^2} \leq \frac{4}{L\eta^2\gamma(m+1)}\left[F(x_0) - F^{\star}\right] + \frac{2\alpha\sigma^2}{L\eta^2\gamma}\left[\frac{1}{\tilde{b}(m+1)} + \frac{(1-\beta)^2}{\hat{b}}\right].
}
By minimizing the last term on the right-hand side of \eqref{eq:descent_prof101} w.r.t. $\beta\in [0, 1]$, we get $\beta := 1 - \frac{\hat{b}^{1/2}}{[\tilde{b}(m+1)]^{1/2}}$.
Clearly, with this choice of $\beta$ if $1 \leq \hat{b}  \leq \tilde{b}(m+1)$, then  $\beta \in [0, 1)$.

\vspace{1ex}
\noindent{(a)}~
Next, we update $\eta := \frac{2}{L(3 + \gamma)}$.
Then, since $\gamma \in [0, 1]$ we have $\frac{1}{2L} \leq \eta \leq \frac{2}{3L}$.
Moreover, we have $\frac{2}{\eta} - 2L - L\gamma = L$ and $\frac{1 + L^2\eta^2}{L} \leq \frac{13}{9L}$.
In addition, noting that since $\beta \in [0, 1)$, we have $1-\beta^2 \geq 1 - \beta =  \frac{\hat{b}^{1/2}}{[\tilde{b}(m+1)]^{1/2}}$.
Consequently, the second condition of \eqref{eq:para_cond100b} holds if we choose $\gamma$ as
\myeqn{
0 < \gamma \leq \bar{\gamma} := \frac{9L\alpha\hat{b}^{1/2}}{13\tilde{b}^{1/2}(m+1)^{1/2}}.
}
Since $\beta \in [0, 1]$, the first condition of \eqref{eq:para_cond100b} holds if we choose $0 < \gamma \leq \bar{\gamma} := \frac{b}{L\alpha}$.
Combining both conditions on $\gamma$, we get $\frac{b}{L\alpha} = \frac{9L\alpha\hat{b}^{1/2}}{13\tilde{b}^{1/2}(m+1)^{1/2}}$, leading to $\alpha := \frac{\sqrt{13}\tilde{b}^{1/4}b^{1/2}}{3L\hat{b}^{1/4}(m+1)^{1/4}}$.
Therefore, we can update $\gamma$ as
\myeqn{
\gamma := \frac{3c_0\hat{b}^{1/4}b^{1/2}}{\sqrt{13}[\tilde{b}(m+1)]^{1/4}},
}
for some $c_0 > 0$.
Since $1 \leq \hat{b}  \leq \tilde{b}(m+1)$, we have $\gamma \leq \frac{3c_0b^{1/2}}{\sqrt{13}}$.
If we choose  $0 < c_0 \leq \frac{\sqrt{13}}{3b^{1/2}}$, then $\gamma \in (0, 1]$.
Consequently, we obtain \eqref{eq:constant_step_sizes}.

\vspace{1ex}
\noindent{(b)}~
Now, we note that the choice of $\alpha$ and $\gamma$ also implies that
\myeqn{
\frac{\alpha}{\gamma} =  \frac{13\tilde{b}^{1/2}(m+1)^{1/2}}{9L\hat{b}^{1/2}}~~~~~\text{and}~~~~~\frac{1}{\gamma} = \frac{\sqrt{13}\tilde{b}^{1/4}(m+1)^{1/4}}{3c_0\hat{b}^{1/4}b^{1/2}}. 
}
In addition, since $\overline{x}_m\sim\Uni{\set{x_t}_{t=0}^m}$, we have $\Exp{\norms{ \Grad_{\eta}(\overline{x}_m)}^2}  = \frac{1}{m+1}\sum_{t=0}^m\Exp{\norms{ \Grad_{\eta}(x_t)}^2}$.
Using these expressions and $L^2\eta^2 \geq \frac{1}{4}$ into \eqref{eq:descent_prof101}, we finally get
\myeqn{
\Exp{\norms{ \Grad_{\eta}(\overline{x}_m)}^2} \leq \frac{16\sqrt{13}L \tilde{b}^{1/4}}{3c_0\hat{b}^{1/4}b^{1/2}(m+1)^{3/4}} \left[F(x_0) - F^{\star}\right] + \frac{208\sigma^2}{9\hat{b}^{1/2}\tilde{b}^{1/2}(m+1)^{1/2}},
}
which proves \eqref{eq:th61_main_estimate}.

Let us choose $b = \hat{b} \in \Nbb_{+}$ and $\tilde{b} := c_1^2[b(m+1)]^{1/3}$ for some $c_1 > 0$. 
Then \eqref{eq:th61_main_estimate} reduces to 
\myeqn{
\Exp{\norms{ \Grad_{\eta}(\overline{x}_m)}^2} \leq \frac{16}{3[b(m+1)]^{2/3}}\left[\frac{\sqrt{13 c_1}L}{c_0} \left[F(x_0) - F^{\star}\right] + \frac{13\sigma^2}{3c_1}\right].
}
Denote $\Delta_0 := \frac{16}{3}\left[\frac{\sqrt{13 c_1}L}{c_0} \left[F(x_0) - F^{\star}\right] + \frac{13\sigma^2}{3c_1}\right]$.
For any tolerance $\varepsilon > 0$, to guarantee $\Exp{\norms{ \Grad_{\eta}(\overline{x}_m)}^2} \leq\varepsilon^2$, we need to impose $\frac{\Delta_0}{[b(m+1)]^{2/3}} = \varepsilon^2$.
This implies $b(m+1) = \frac{\Delta_0^{3/2}}{\varepsilon^3}$, which also leads to $m+1 =  \frac{\Delta_0^{3/2}}{b\varepsilon^3}$.
Therefore, the maximum number of iterations is at most $m :=  \left\lfloor\frac{\Delta_0^{3/2}}{b\varepsilon^3}\right\rfloor$.
This is also the number of proximal operations $\prox_{\eta\psi}$.

The number of stochastic gradient evaluations $\nabla{f_{\xi}}(x_t)$ is at most $\Tc_m := \tilde{b} + 3(m+1)b = \frac{c_1^2\Delta_0^{1/2}}{\varepsilon} + \frac{3\Delta_0^{3/2}}{\varepsilon^3}$.
Finally,  since $1 \leq b = \hat{b}  \leq \tilde{b}(m+1) = c_1^2b^{1/3}(m+1)^{4/3}$, we have $b \leq c_1^3(m+1)^2$, which is equivalent $c_1\geq\frac{b^{1/3}}{(m+1)^{2/3}}$.
In addition, since $\tilde{b} := c_1^2[b(m+1)]^{1/3}$ and $b=\hat{b}$, we have $\gamma := \frac{3c_0b^{2/3}}{\sqrt{13c_1}(m+1)^{1/3}}$.
\Eproof

\beforesubsec
\subsection{\bf The proof of Theorem~\ref{th:double_loop_convergence_minibatch}: The restarting mini-batch variant}\label{apdx:th:double_loop_convergence_minibatch}
\aftersubsec
(a)~Similar to the proof of Theorem~\ref{th:singe_loop_adapt_step}, summing up \eqref{eq:upper_bound_new0} from $t := 0$ to $t := m$ and using \eqref{eq:vt_variance_bound_new_batch} with $\rho := \frac{1}{b}$ and $\hat{\rho} := \frac{1}{\hat{b}}$ from Lemma~\ref{le:upper_bound_new_batch}, we obtain
\myeq{eq:co51_est1}{
\begin{array}{lcl}
\Exp{F(x_{m+1}^{(s)})} &\leq& \Exp{F(x_0^{(s)})} +\dfrac{L^2}{2b}\displaystyle\sum_{t=0}^m \theta_t\sum_{i=0}^{t-1}\gamma_i^2\omega_{i,t}\Exp{\norms{\hat{x}_{i+1}^{(s)} - x_{i}^{(s)}}^2}  \vspace{1ex}\\
&& - {~} \dfrac{1}{2}\displaystyle\sum_{t=0}^m\kappa_t \Exp{\Vert\widehat{x}_{t+1}^{(s)} - x_t^{(s)}\Vert^2}  - \displaystyle\sum_{t=0}^m\dfrac{\gamma_t\eta^2}{2}\Exp{\Vert \Grad_{\eta}(x_t^{(s)})\Vert^2} \vspace{1ex}\\
&& +{~} \dfrac{1}{2}\displaystyle\sum_{t=0}^m\theta_t \omega_t\Exp{\norms{\hat{v}_0^{(s)} - \nabla f(x_0^{(s)})}^2}  +  \dfrac{1}{2\hat{b}}\displaystyle\sum_{t=0}^m \theta_t S_t,
\end{array}
}
where $\gamma_t$, $\eta$, $\kappa_t$, $\theta_t$, $\omega_{i,t}$, $\omega_t$, and $S_t$ are defined in Lemma~\ref{le:upper_bound_new}.

Let us fix $c_t := L$,  $r_t := 1$, $q_t := \frac{L\gamma_t}{2}$, and $\beta_t := \beta \in [0, 1]$.
Then $\theta_t  = \frac{(1 + L^2\eta^2)}{L}\gamma_t$ and $\kappa_t =  \gamma_t\left(\frac{2}{\eta} - 2L - L\gamma_t\right)$ as before.
Moreover, $\omega_t = \beta^{2t}$, $\omega_{i,t} = \beta^{2(t-i)}$, and  $s_t =  (1-\beta)^2\Big[\frac{1-\beta^{2t}}{1-\beta^2}\Big] < \frac{1-\beta}{1+\beta}$ due to Lemma \ref{le:upper_bound_new}, and $\Exp{\norms{\hat{v}_0^{(s)} - \nabla f(x_0^{(s)})}^2} \leq \frac{\sigma^2}{\tilde{b}}$.

Using this configuration and noting that $\overline{x}^{(s)} = x_{m+1}^{(s)}$ and $\overline{x}^{(s-1)} = x^{(s)}_0$, following the same argument as \eqref{eq:key_estimate_301}, \eqref{eq:co51_est1} reduces to 
\myeq{eq:co51_est2}{
\begin{array}{lcl}
\Exp{F(\overline{x}^{(s)})} &\leq& \Exp{F(\overline{x}^{(s-1)})} -  \frac{L\eta^2}{4}\displaystyle\sum_{t=0}^m\gamma_t\Exp{\Vert \Grad_{\eta}(x_t^{(s)})\Vert^2} \vspace{1ex}\\
&&+ ~\frac{(1 + L^2\eta^2)\sigma^2}{2L(1+\beta)} \left[\frac{1}{\tilde{b}(m+1)(1-\beta)}  + \frac{(1-\beta)}{\hat{b}}\right]\Sigma_m + \frac{\widehat{\Tc}_m}{2},
\end{array}
}
where $\widehat{\Tc}_m$ is defined as follows:
\myeq{eq:T_m_hat}{
\begin{array}{lcl}
\widehat{\Tc}_m &:=& \frac{L(1 + L^2\eta^2)}{b}\sum_{t=0}^m\gamma_t\sum_{i=0}^{t-1}\beta^{2(t-i)}\gamma_i^2\Exp{\norms{\widehat{x}_{i+1}^{(s)} - x_i^{(s)}}^2} \vspace{1ex}\\
&& - {~} \sum_{t=0}^m\gamma_t\left(\frac{2}{\eta} - 2L - L\gamma_t\right)\Exp{\norms{\widehat{x}_{i+1}^{(s)} - x_i^{(s)}}^2}.
\end{array}
}
Similar to the proof of \eqref{eq:update_of_eta_t}, if we choose $\eta \in (0, \frac{1}{L})$, set $\delta := \frac{2}{\eta} - 2L > 0$, and update $\gamma$ as in \eqref{eq:step_size_gamma_t}:
\myeqn{
\gamma_m := \frac{\delta}{L}~~~~\text{and}~~~~\gamma_t := \frac{\delta b}{Lb + L(1 + L^2\eta^2)\big[\beta^2\gamma_{t+1} + \beta^4\gamma_{t+2} + \cdots + \beta^{2(m-t)}\gamma_m\big]},
}
then $\widehat{\Tc}_m \leq 0$.
Moreover, since $\beta \in [0, 1]$ and $1+L^2\eta^2 \leq 2$, \eqref{eq:co51_est2} can be simplified as
\myeqn{ 
\Exp{F(\overline{x}^{(s)})} \leq \Exp{F(\overline{x}^{(s-1)})} -  \frac{L\eta^2}{4}\displaystyle\sum_{t=0}^m\gamma_t\Exp{\Vert \Grad_{\eta}(x_t^{(s)})\Vert^2}  +  \frac{\sigma^2}{L}\left[\frac{1}{\tilde{b}(m+1)(1-\beta)}  + \frac{(1-\beta)}{\hat{b}}\right]\Sigma_m.
}
Summing up this inequality from $s := 1$ to $s := S$ and noting that $F(\overline{x}^{(S)}) \geq F^{\star}$, we obtain
\myeq{eq:co51_est3}{
{\!\!\!\!}\frac{1}{S\Sigma_m}\displaystyle\sum_{s=1}^S\sum_{t=0}^m\gamma_t\Exp{\Vert \Grad_{\eta}(x_t^{(s)})\Vert^2} \leq \frac{4\big[F(\overline{x}^{(0)}) - F^{\star}\big] }{L\eta^2S\Sigma_m}+  \frac{4\sigma^2}{L^2\eta^2}\left[\frac{1}{\tilde{b}(m+1)(1-\beta)}  + \frac{(1-\beta)}{\hat{b}}\right].{\!\!\!}
}
Let us first choose $\beta := 1 - \frac{\hat{b}^{1/2}}{\tilde{b}^{1/2}(m+1)^{1/2}}$.
Then, $1 - \beta^2 \leq \frac{2\hat{b}^{1/2}}{\tilde{b}^{1/2}(m+1)^{1/2}}$ and $\frac{(1+L^2\eta^2)\beta^2}{L} \leq \frac{2}{L}$.
Using these inequalities, similar to the proof of \eqref{eq:deno_est1}, we can upper bound 
\myeqn{
L\left[\sqrt{1-\beta^2} + \sqrt{1 - \beta^2 + \frac{4(1+L^2\eta^2)\beta^2\delta}{Lb}}\right] \leq 2\sqrt{2}\left[\frac{L\hat{b}^{1/4}b^{1/2} {~} + {~} [\tilde{b}(m+1)]^{1/4} \sqrt{L\delta}}{b^{1/2}[\tilde{b}(m+1)]^{1/4}}\right].
}
Using this bound, the update rule \eqref{eq:step_size_gamma_t} of $\gamma_t$, and $\sqrt{1-\beta^2} \geq \frac{\hat{b}^{1/4}}{[\tilde{b}(m+1)]^{1/4}}$, we apply Lemma~\ref{le:adaptive_step_size} with $\omega := \beta^2$ and $\epsilon := \frac{(1 + L^2\eta^2)}{Lb}$ to obtain
\myeqn{
\Sigma_m := \sum_{t=0}^m\gamma_t \geq \frac{\delta(m+1)\sqrt{1-\beta^2}}{L\Big[\sqrt{1-\beta^2} + \sqrt{1 - \beta^2 + \frac{4(1+L^2\eta^2)\beta^2\delta}{Lb}}\Big]}
\geq \frac{\delta(m+1)\hat{b}^{1/4}b^{1/2}}{2\sqrt{2}\left[L\hat{b}^{1/4}b^{1/2} + [\tilde{b}(m+1)]^{1/4} \sqrt{L\delta}\right]}.
}
Utilizing this bound into \eqref{eq:co51_est3} and noting that $\overline{x}_T \sim \Unip{\pb}{\sets{x^{(s)}_t}_{t=0\to m}^{s=1\to S}}$, we can upper bound it as
\myeqn{
\Exp{\Vert \Grad_{\eta}(\overline{x}_T)\Vert^2} \leq \frac{8\sqrt{2}\left[L\hat{b}^{1/4}b^{1/2} + [\tilde{b}(m+1)]^{1/4} \sqrt{L\delta}\right]}{L\eta^2\delta S(m+1)\hat{b}^{1/4}b^{1/2}}\big[F(\overline{x}^{(0)}) - F^{\star}\big] +  \frac{8\sigma^2}{L^2\eta^2[\hat{b}\tilde{b}(m+1)]^{1/2}},
}
which is exactly \eqref{eq:double_loop_est2}.

\vspace{1ex}
\noindent{(b)}~Now, let us choose $\hat{b} = b \in\Nbb_{+}$ and assume that $\tilde{b}  :=  c_1^2b(m+1)$ for some $c_1 > 0$.
In this case, the right-hand side of \eqref{eq:double_loop_est2} can be  upper bounded as
\myeqn{
\begin{array}{lcl}
\Rc_T &:= &  \frac{8\sqrt{2}\left[Lb^{3/4} {~}+{~} [\tilde{b}(m+1)]^{1/4} \sqrt{L\delta}\right]}{L\eta^2\delta S(m+1)b^{3/4}}\big[F(\overline{x}^{(0)}) - F^{\star}\big] +  \frac{8\sigma^2}{L^2\eta^2 c_1b(m+1)} \vspace{1ex}\\
&= & \frac{8\sqrt{2}\Delta_F}{\eta^2\delta S(m+1)} + \frac{8\sqrt{2 c_1}\Delta_F}{\sqrt{L\delta}\eta^2 S(m+1)^{1/2}b^{1/2}} +  \frac{8\sigma^2}{L^2\eta^2 c_1b(m+1)},
\end{array}
}
where $\Delta_F := F(\overline{x}^{(0)}) - F^{\star} > 0$.

For any $\varepsilon > 0$, to guarantee $\Exp{\Vert \Grad_{\eta}(\overline{x}_T)\Vert^2} \leq \varepsilon^2$, we impose $\Rc_T \leq \varepsilon^2$.
From the upper bound of $\Rc_T$, we can break its relaxed condition into three parts as
\myeq{eq:three_parts}{
\frac{8\sqrt{2}\Delta_F}{\eta^2\delta S(m+1)} \leq \frac{\varepsilon^2}{3},~~~~~\frac{8\sqrt{2 c_1}\Delta_F}{\sqrt{L\delta}\eta^2S(m+1)^{1/2}b^{1/2}} = \frac{\varepsilon^2}{3},~~~~~\text{and}~~~~ \frac{8\sigma^2}{L^2\eta^2 c_1b(m+1)} \leq \frac{\varepsilon^2}{3}.
}
Let us choose $m + 1 := \frac{24}{c_1L^2\eta^2 b\varepsilon^2}\max\set{\sigma^2,1}$.
Then, $\tilde{b} = \frac{24c_1}{L^2\eta^2 \varepsilon^2}\max\set{\sigma^2,1}$.
Moreover, the last condition of \eqref{eq:three_parts} holds and $m+1 \geq  \frac{24}{c_1L^2\eta^2 b\varepsilon^2}$. 
Hence, the second condition of \eqref{eq:three_parts} leads to 
\myeqn{
S = \frac{24\sqrt{2 c_1}\Delta_F}{\sqrt{L\delta}\eta^2 \varepsilon^2}\frac{1}{\sqrt{b(m+1)}} \leq \frac{4\sqrt{3L}c_1 \Delta_F}{\sqrt{\delta}\eta  \varepsilon}.
}
From the second condition of \eqref{eq:three_parts}, we also have 
\myeqn{
(m+1)bS = \frac{24\sqrt{2c_1}\Delta_F}{\eta^2\sqrt{L\delta} \varepsilon^2}(m+1)^{1/2}b^{1/2} = \frac{96\sqrt{3}\Delta_F}{\eta^3L\sqrt{L\delta} \varepsilon^3}\max\set{1,\sigma}.
}
From this expression,  to guarantee the first condition of \eqref{eq:three_parts}, we need to impose 
\myeqn{
(m+1)S = \frac{96\sqrt{3}\Delta_F}{\eta^3L\sqrt{L\delta} b\varepsilon^3}\max\set{1,\sigma} \geq \frac{24\sqrt{2}\Delta_F}{\eta^2\delta\varepsilon^2}, 
}
which leads to $1\leq b \leq \frac{2\sqrt{6\delta}}{L\sqrt{L}\eta\varepsilon}$.

Finally, the total number of stochastic gradient evaluations $\nabla{f}_{\xi}(x_t)$ is at most 
\myeqn{
\begin{array}{lcl}
\Tc_{\nabla{f}} &:= & \left[\tilde{b} + 3b(m+1)\right]S = \left\lfloor (c_1^2 + 3)b(m+1)S \right\rfloor \vspace{1ex}\\
&= & \left\lfloor (c_1^2 + 3)\frac{96\sqrt{3}\Delta_F}{\eta^3L\sqrt{L\delta} \varepsilon^3}\max\set{1,\sigma} \right\rfloor.
\end{array}
}
The total number of proximal operations $\prox_{\eta\psi}$ is at most $\Tc_{\prox} = (m+1)S = \left\lfloor \frac{96\sqrt{3}\Delta_F}{\eta^3L\sqrt{L\delta} b\varepsilon^3}\max\set{1,\sigma} \right\rfloor$.
\Eproof

\bibliographystyle{plain}

\end{document}